%% file: ms.tex
\documentclass{article}
\usepackage[left=2.2cm,right=2.2cm,top=3cm,bottom=2.5cm, marginparsep=-1mm]{geometry}

\usepackage{marginnote}
\usepackage{amsmath}
\usepackage{amsfonts}
\usepackage{booktabs}

\usepackage{amsthm}
\usepackage{xcolor}
\usepackage{graphicx}
\usepackage{float}
\usepackage{bussproofs}
\usepackage{tabularx}
\usepackage{tikz}
\usetikzlibrary{%
  positioning
}

\usepackage{proof}

\usepackage[colorlinks=true]{hyperref}
\usepackage[nameinlink]{cleveref}
\hypersetup{
    colorlinks,
    linkcolor={red!60!black},
    citecolor={blue!50!black},
    urlcolor={blue!80!black}
}

\newtheorem{theorem}{Theorem}[subsection]
\newtheorem{corollary}[theorem]{Corollary}
\newtheorem{lemma}[theorem]{Lemma}
\newtheorem{proposition}[theorem]{Proposition}

\newtheorem{remark}[theorem]{Remark}

\theoremstyle{definition}
\newtheorem{definition}{Definition}[section]

\crefname{subsection}{subsection}{subsections}
\Crefname{subsection}{Subsection}{Subsections}
\crefname{paragraph}{paragraph}{paragraphs}
\Crefname{paragraph}{Paragraph}{Paragraphs}

\newcommand\subtm{\preceq}
\newcommand\psubtm{\prec}
\newcommand\hole{\Box}
\newcommand\plug[2]{#1[#2]}

\newcommand\oeis[1]{\href{https://oeis.org/#1}{\textbf{#1}}}

\newcommand{\convergesInDist}{\overset{D}{\rightarrow}}

\newcommand*{\ditto}{---\texttt{//}---}

\newcommand{\singleton}{\mathcal{Z}}

\newcommand\closedTerms[1][]{{{\dot{\mathcal{T}}}_{[0]#1}}}
\newcommand{\lamClosedOGF}{T}

\newcommand{\closedAbsClass}{\closedTerms^{\lambda}}
\newcommand{\bridgelessTerms}{{\mathcal{\dot{B}}}_{[0]}}
\newcommand{\oneBridgeTerms}{\mathcal{\dot{O\!B}}_{[0]}}
\newcommand{\oneVarOpBridgeTerms}{{\mathcal{\dot{B}}}_{[1]}}

\newcommand{\contextClass}{\mathcal{K}}
\newcommand{\contextOGF}{K}
\newcommand{\qClass}{\mathcal{Q}}
\newcommand{\qOGF}{Q}

\newcommand{\lamClosedIdOGF}{T^{id}}
\newcommand{\lamClosedSubOGF}{T^{sub}}
\newcommand{\abstractionSubOGF}{T^{\lambda}_{sub}}

\usepackage{mathtools}

\newcommand{\oneThreeMaps}{\mathcal{T}}
\newcommand{\discoOneThreeMaps}{\oneThreeMaps^{d}}
\newcommand{\openTerms}{\mathcal{\dot{\oneThreeMaps}}}
\newcommand{\lamOpenOGF}{T}

\newcommand{\twoThreeMaps}{\mathcal{D}}
\newcommand{\rootedTwoThreeMaps}{\dot{\twoThreeMaps}}
\newcommand{\discoTwoThreeMaps}{\twoThreeMaps^{d}}
\newcommand{\affineTerms}{\mathcal{A}}
\newcommand{\twoThreeMapsOGF}{D}
\newcommand{\affineTermsOGF}{A}

\newcommand{\pathClass}{\mathcal{P}}

\newcommand{\balancedPart}[1][]{\mathrm{B}_{#1}}
\newcommand{\alphasSummed}[1][]{\mathrm{D}_{#1}}
\newcommand{\alphasWSummed}[1][]{\mathrm{C}_{#1}}

\usepackage[T1]{fontenc}
\newcommand*{\textcal}[1]{%
  \textit{\fontfamily{qzc}\selectfont#1}%
}

\newcommand{\loopMap}{%
\begin{tikzpicture}[scale=0.8, every loop/.style={}, baseline=-2mm, thick]
  \node [fill=black,circle,minimum size=4pt,inner sep=0pt,outer sep=0pt] (1) {} edge [in=-50,out=-130,loop] ();
  \node [draw=black,fill=white,circle,minimum size=2pt,inner sep=0pt,outer sep=0pt] (2) [above=0.1cm of 1] {};
  \draw (2.south) to (1.north);
\end{tikzpicture}%
}

\newcommand{\oneEdgeMap}{%
\begin{tikzpicture}[scale=1, every loop/.style={}, baseline=-2.5mm, thick]
\node [draw=black,fill=white,circle,minimum size=2pt,inner sep=0pt,outer sep=0pt] (1) {};
\node [fill=black,circle,minimum size=4pt,inner sep=0pt,outer sep=0pt] (2) [below=0.2cm of 1] {};
\draw (1.south) to (2.north);
\end{tikzpicture}%
}

\title{Asymptotic Distribution of Parameters in \linebreak  Trivalent Maps and Linear Lambda Terms}
\author{Olivier Bodini \and Alexandros Singh \and Noam Zeilberger}

\begin{document}
\maketitle

\abstract{Structural properties of large random maps and $\lambda$-terms may be gleaned by studying the limit distributions of various parameters of interest. In our work we focus on restricted classes of maps and their counterparts in the $\lambda$-calculus, building on recent bijective connections between these two domains. In such cases, parameters in maps naturally correspond to parameters in $\lambda$-terms and vice versa.
By an interplay between $\lambda$-terms and maps, we obtain various combinatorial specifications which allow us to access the distributions of pairs of related parameters such as: the number of bridges in rooted trivalent maps and of subterms in closed linear $\lambda$-terms, the number of vertices of degree 1 in $(1,3)$-valent maps and of free variables in open linear $\lambda$-terms etc. To analyse asymptotically these distributions, we introduce appropriate tools: a moment-pumping schema for differential equations and a composition schema inspired by Bender's theorem.}

\tableofcontents

\section{Introduction}
\input{sections/intro}

\section{Definitions and basic tools}
\input{sections/tools}

\section{First-order ODEs and Poisson distributions}
\input{sections/poisson}

\section{Compositions of divergent powerseries and universality of parameter distributions}
\input{sections/compositions}

\section{Conclusion and future directions}
\input{sections/conclusion}

\bibliography{biblio} 
\bibliographystyle{ieeetr}

\end{document}

%% file: sections/intro.tex
Building upon an ever-increasing body of work on the combinatorics of maps, the $\lambda$-calculus, and their interactions, we present here a study of the asymptotic behaviour of some structural properties of large random objects drawn from restricted subclasses of maps and $\lambda$-terms. 

\subsection{Motivation and main results}\label{subsection:intro}
Maps, or graphs embedded on surfaces, are an important object of study in modern combinatorics and their presence in various areas, ranging from algebra to physics, forms bridges between seemingly disparate subjects.
In recent years it has become apparent that such bridges extend to logic as well, stemming from bijections between various natural classes of rooted maps and certain subsystems of $\lambda$-calculus.
This includes a natural bijection between rooted trivalent maps and linear lambda terms \cite{bodini2013asymptotics}, as well as a somewhat more involved bijection between rooted planar maps of arbitrary vertex degrees and $\beta$-normal ordered linear lambda terms \cite{ZG2015corr}, both of which have led to further study of the combinatorial interactions between lambda terms and maps.

To make the above correspondence concrete, let us briefly recall here the bijection of \cite{bodini2013asymptotics} following the analysis of \cite{zeilberger2016linear}, which is itself inspired by Tutte's classical approach to map enumeration via repeated root edge decomposition \cite{tutte1968}.
Informally, a rooted trivalent map may be defined as a graph equipped with an embedding into an oriented surface of arbitrary genus, all of whose vertices have degree 3, and one of whose edges has been distinguished and oriented (see below for a more formal definition).
For reasons that will be quickly apparent, it is pertinent to slightly extend the class of rooted trivalent maps by embedding it into the class of (1,3)-valent maps, that is, maps whose vertices all have degree 3 or 1.
We will view 1-valent vertices as labelled ``external'' vertices, and the root itself as a distinguished external vertex.
By considering what happens around the root, it is clear that such a map falls into one of three categories: 

\begin{figure}[H]
\centering
\includegraphics[scale=1.09]{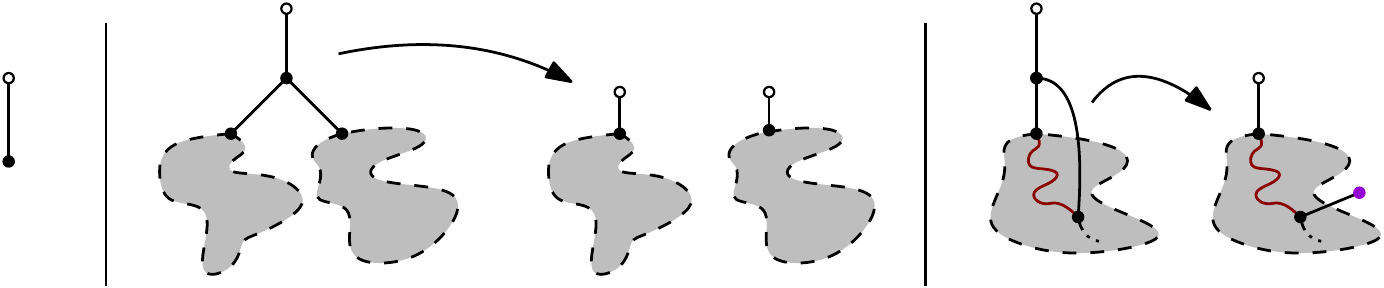}
\label{fig:tutte}
\end{figure}

\noindent namely, it is (from left to right) either the trivial one-edge map with no trivalent vertices and a single 1-valent vertex besides the root, a map in which the deletion of the root and its unique trivalent neighbour yields a pair of disconnected maps which may be canonically rooted, or finally a map in which the same operation yields a connected map which may be again rooted canonically and which in addition has a distinguished degree-1 vertex.

Quite remarkably, this decomposition \`a la Tutte exactly mirrors the standard inductive definition of linear lambda terms.
Informally, an arbitrary lambda term is either 
a \emph{variable} $x$, an \emph{application} $(t~u)$ of a term $t$ to another term $u$, or an \emph{abstraction} $\lambda x.t$ of a term $t$ in a variable $x$, with linearity imposing the condition that in an abstraction $\lambda x.t$, the variable $x$ has to occur exactly once in $t$.
All of the terminology will eventually be explained, but concretely, the differential equation resulting from this analysis
\begin{equation}\label{eq:diff_open_linlams}
    \lamOpenOGF(z,u) = zu + z\lamOpenOGF(z,u)^2 + z\frac{\partial}{\partial u} \lamOpenOGF(z,u) 
\end{equation}
can be seen as counting either (1,3)-valent maps or linear $\lambda$-terms, with the size variable $z$ tracking edges or subterms and the ``catalytic'' variable $u$ tracking non-root 1-valent vertices or free variables.
Setting $u = 0$ then allows us to recover the ordinary generating function $\lamOpenOGF(z,0)$ counting rooted trivalent maps in the classical sense
as well as closed linear lambda terms.

The bijection from $\lambda$-terms to maps is made even more evident by representing the terms as certain decorated syntactic diagrams, in the manner of \Cref{fig:synDiag}. Such diagrams yield rooted trivalent maps with external 1-valent vertices simply by forgetting the labels of trivalent nodes, while the above correspondence shows that this information can be uniquely reconstructed from a given map by a recursive decomposition.
A more comprehensive discussion of the correspondence between rooted trivalent maps and linear $\lambda$-terms is given in \cite{bodini2013asymptotics} and \cite{zeilberger2016linear}, and we will review it further below. 

\begin{figure}[h]
  \centering
  \includegraphics[scale=1.5]{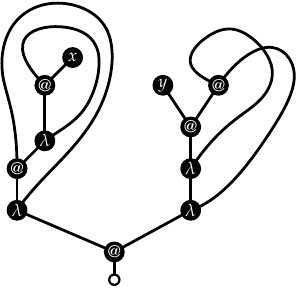}
  \caption{The syntactic diagram of the open linear $\lambda$-term $(\lambda a.(\lambda b. b x)) (\lambda c.\lambda d.y(c d))$.}
  \label{fig:synDiag}
\end{figure}

By exploiting such bijective correspondences between families of maps and $\lambda$-terms, we identify and study pairs of corresponding parameters natural to both classes, focusing on their limit distributions. The parameters studied in this work, together with their limit distributions, are listed in \Cref{table:pairs}.
\begin{table}[H]
\centering
\footnotesize
\begin{tabular}{@{}lll@{}}
\toprule
Parameter on maps  (number of) & Parameter on $\lambda$-terms  (number of) & Limit distribution \\ \midrule
Loops in rooted trivalent maps                & Identity-subterms in closed linear $\lambda$-terms   & $Poisson(1)$           \vspace{0.2cm} \\
Bridges in rooted trivalent maps              & Closed subterms in closed linear $\lambda$-terms     & $Poisson(1)$           \vspace{0.2cm} \\
Vertices of degree $1$ in rooted $(1,3)$-maps & Free variables in open linear $\lambda$-terms up to exchange        &  $\mathcal{N}\left( (2n)^{1/3}, (2n)^{1/3} \right)$ \vspace{0.2cm} \\ 
Vertices of degree $2$ in rooted $(2,3)$-maps & Unused abstractions in closed affine $\lambda$-terms & $\mathcal{N}\left( \frac{(2n)^{2/3}}{2},\frac{(2n)^{2/3}}{2} \right)$ \\ \bottomrule
\end{tabular}
\caption{Pairs of corresponding parameters in families of maps and $\lambda$-terms and their limit distributions, where $Poisson(\lambda)$ signifies the Poisson law of rate $\lambda$ and $\mathcal{N}(\mu, \sigma^2)$ is shorthand for the corresponding random variables $X_n$ converging in law to the standard normal distribution after being standardised as $\frac{X_n - \mu}{\sigma_n}$. }
\label{table:pairs}
\end{table}

The first step of our approach is obtaining combinatorial specifications which allow us to capture the behaviour of our parameters of interest. This is done via a number of new decompositions valid for restricted families of maps and $\lambda$-terms. We are then faced with the task of asymptotically analysing these specifications, a task made difficult by the fact that number of elements of a given size in these families exhibits rapid growth; this precludes a straightforward approach based on standard tools of analytic combinatorics as the corresponding generating functions are purely formal power series and do not represent functions analytic at 0. To facilitate our approach, we therefore develop two new schemas which serve to encompass the two general cases we have observed in our study: differential specifications giving rise to Poisson limit laws and composition-based specifications giving rise to Gaussian limit laws.

Our purpose in this present work is therefore twofold. On the one hand, we want to demonstrate how interesting insights on the typical structure of large random maps and $\lambda$-terms may be obtained by fruitfully making use of techniques drawn from the study of maps and $\lambda$-terms in tandem. On the other hand, we present two new tools which aid in the asymptotic analysis of parameters of fast-growing combinatorial classes; these tools are of independent interest, being applicable to the study of a wide class of combinatorial classes whose generating function is purely formal and obeys certain types of differential or functional equations.

\subsection{Related work}

The structure and enumeration, both exact and asymptotic, of maps by their genus has been the subject of much study; see, for example, \cite{cori1981planar,bender1986survey,banderier2001random} for the planar case and \cite{walsh1972counting, bessis1980quantum,marcus:inria-00100704} for the higher genus case. On the other hand, the investigation of enumerative and statistical properties of rooted maps, counted without regard to their genus, has received much less attention. One reason for this is the divergent nature of the generating functions involved in such studies: by results derived in \cite{arques2000rooted} one may show that the number of maps with $n$ edges is asymptotically $(2n-1)!!$. This poses a significant obstacle since, as Odlyzko notes in \cite{odlyzko1995asymptotic}: ``{\em There are few methods for dealing with asymptotics of formal power series, at least when compared to the wealth of techniques available for studying analytic generating functions}''. As such, the structure of large random such maps has only recently begun to be investigated, starting with the distribution of genus in bipartite random maps being derived \cite{carrance2019uniform}. More recently, the authors of \cite{bodini2018asymptotic} investigated the asymptotic distributions for the number of vertices, root isthmic parts, root edges, root degree, leaves, and loops in random maps. In particular, comparing their results to ours, we note that for general maps the authors derived a $Poisson(1)$ limit law for the number of leaves and a previously-unknown law for the number of loops. Both of these results stand in stark contrast to the case of leaves in $(1,3)$-valent maps, which we show is normally distributed when standardised using $\mu = \sigma^2 = n^{1/3}$, and to the case of loops in rooted trivalent maps which we show is $Poisson(1)$. In terms of techniques employed, the authors of \cite{bodini2018asymptotic} show that for most of the statistics considered in their work the corresponding bivariate generating functions are formal solutions to Riccati equations, which may be linearised to yield recurrences on the coefficients of said generating functions which are amenable to study. We note here that an instance of a Riccati-type differential equation appears in our work too, but this time it is a differential equation with respect to the variable coupled to the statistic we're interested in, unlike the instances of \cite{bodini2018asymptotic} where the derivative was taken with regards to the size-coupled variable.

As for the $\lambda$-calculus, while of central importance to logic and theoretical computer science, it is a relatively new subject of study for combinatorialists. The combinatorial study of closed linear and affine $\lambda$-terms and their relaxations was introduced in \cite{bodini2013asymptotics,bodini2013enumeration}. A comprehensive presentation of the combinatorics of open and closed linear $\lambda$-terms and their counterparts in maps is presented in \cite{zeilberger2016linear}. We note here that there exists a number of combinatorial studies of $\lambda$-terms which use a size notion different than ours and that of \cite{bodini2013asymptotics,bodini2013enumeration,zeilberger2016linear}. For example, there exists a number of works focusing on a unary de Bruijn notation based model as in \cite{bendkowski2017combinatorics}. This choice of size notion has the effect of altering the qualitative properties of our objects of study: in particular, the statistical results and the associated techniques of \cite{bendkowski2018statistical} are not applicable to our model.

%% file: sections/tools.tex
\subsection{Graphs and maps}
\label{subsec:maps}

We begin by establishing some notation and definitions pertaining to graphs and maps. A comprehensive treatment of maps and various of their aspects is presented \cite{lando2013graphs}.

\paragraph*{Graphs.} %
We will consider finite undirected graphs, allowing for loops and multiple edges. Given a graph $G$, we will denote the set of its vertices by $V(G)$ and that of its edges by $E(G)$. For an edge $\{u,v\} \in E(G)$ will write $e = uv = vu$ to denote the edge $e$ between $u$ and $v$ and we will call $u,v$ the {\em endpoints} of $e$. We will also say that $e$ {\em is incident to} $u,v$. 

Given a vertex $v \in V(G)$ we call the set $N_{G}(v) = \{ u \lvert u \in V(G), uv \in E(G) \}$ its {\em neighbourhood}. The {\em degree} or {\em valency} of $v$ is $\lvert N_{G}(v) \lvert$.

A {\em subgraph} of a graph $G$ is a graph $H$ such that $V(H) \subseteq V(G)$ and $E(H) \subseteq E(G)$. The subgraph {\em induced by $S \subseteq V(G)$}, denoted by $G[S]$, is the subgraph of $G$ consisting of all vertices in $S$ and all edges in $E(G)$ that have both endpoints in $S$. 

For a vertex $u \in V(G)$, we denote by $G \setminus u$ the subgraph induced by $V(G)\setminus u$. For an edge $e \in E(G)$, we denote by $G \setminus e$ the graph $(V(G), E(G)\setminus e)$. We will refer to the last two operations as {\em vertex deletion} and {\em edge deletion} respectively. We define $G / v$, the graph obtained from $G$ by {\em dissolving} a degree 2 vertex $v \in V(G)$, to be the graph obtained by deleting $v$ and adding an edge to its two neighbours adjacent.

An edge $e \in E(G)$ is a {\em bridge} if $G \setminus e$ has one more connected component than $G$.

We'll denote by $G + H$ the disjoint union of two graphs $G, H$.

\paragraph*{Maps as embedded graphs.} A {\em map} is an embedding of a connected graph into a connected, closed, oriented surface such that all faces are homeomorphic to open disks. Maps are considered up to orientation preserving homeomorphisms of the underlying surface. 

In \Cref{subsec:freeLams,subsec:freeVars} we'll also make use of the notion of a {\em disconnected map}. Such maps will be considered as embeddings of disconnected graphs not on a single surface but on the disjoint union of such surfaces, in a way such that each connected component of the graph is drawn on a different surface. When the need arises to consider both connected and disconnected maps as a single class, we will refer to them as {\em not-necessarily-connected maps}.
In this work we focus on embeddings of degree-constrained graphs. We shall refer to maps whose underlying graph's vertices are all of degree three as {\em trivalent}. More generally, if the set of allowed degrees is ${k_1, \dots, k_n} \in \mathbb{N}_{\geq 0}^{n}$ we'll talk about a $(k_1, \dots, k_n)$-valent map or just {\em $(k_1, \dots, k_n)$-map}.

Finally, we will often make use of graph-theoretic notions when referring to a map, which are to be interpreted as identifying properties of its underlying graph. For example, a {\em bridge} of a map is a bridge of its underlying graph, i.e., an edge whose deletion results in a disconnected graph. Similarly, a {\em path} in a map is a path in its underlying graph. A {\em spanning tree} of a map is a spanning tree of its underlying graph. 

A {\em submap} $m'$ of a map $m$ is an embedding of the subgraph corresponding to $m'$ in $m$. If $g$ is the graph of a map $m$ and $v \subseteq V(G)$, $m[v]$ will denote the embedding of the induced subgraph $g[v]$. For both of the above cases, the embedding chosen for a subgraph $m'$ is the restriction of the one of $m$, i.e., the embedding of $m'$ for which the neighbours of any vertex in $m'$ are oriented in exactly the same way as in $m$.

\paragraph*{Maps and permutations.}
It is well-known that embeddings of graphs may be represented, up to isomorphism, by certain systems of permutations \cite{lando2013graphs}.
In particular, maps on connected closed oriented surfaces have the following equivalent purely algebraic definition: a finite set $H$ of {\em half-edges} together with a pair $(v, e)$ of permutations on $H$ such that $e$ is a fixed-point-free involution and the group $\langle v, e \rangle$ acts transitively on $H$.
Such objects are sometimes referred to as {\em combinatorial maps}. 
More generally, if one drops the requirement that $\langle v, e\rangle$ acts transitively one obtains not-necessarily-connected maps.

Various properties of a map $M$ may be read off from the tuple $(v,e)$. For example, vertices of $M$ correspond to cycles of $v$, their degree being the length of said cycle. A cycle $(h_1, h_2)$ of $e$ similarly is to be interpreted as encoding an edge $e$ formed by gluing two half-edges $h_1, h_2$.
In particular, observe that any trivalent combinatorial map corresponds to a pair of a cubic permutation and an involution, thus yielding a representation of the modular group $PSL(2,\mathbb{Z}) \cong \langle x,y \mid x^3 = y^2 = 1\rangle$.
Finally, the faces of the map may be read off as the cycles of the permutation $f = e^{-1}v^{-1}$ (thus satisfying the identity $vef = 1$).

\begin{figure}[h]
  \centering
  \includegraphics[scale=1.5]{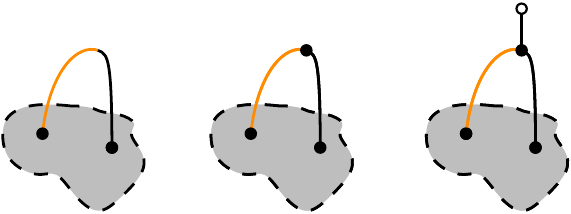}
  \caption{The bijection between half-edge rooted (1,3)-valent maps (left) and vertex-rooted open rooted trivalent maps (right) obtained by first subdividing the root-containing edge of the former (middle step) before finally adding a new vertex.}
  \label{fig:rootings}
\end{figure}

\paragraph*{Open rooted trivalent maps with external vertices.}\label{par:rootedMaps}

Under the permutation-based definition, a popular way of rooting a map is by simply choosing an arbitrary half-edge, in which case the unique $v\text{-}$, $e\text{-}$, and $f$-cycles it forms a part of are marked as the \emph{root vertex}, \emph{root edge}, and \emph{root face}, respectively.
We shall adopt a slightly different convention for rooting trivalent maps, which is partly motivated by their correspondance with $\lambda$-terms, but may also be motivated by considering rooted maps as embeddings of graphs on surfaces with boundary (cf.~Tutte's original definition of rooted planar triangulations as dissections of closed regions of the plane \cite{Tutte1962planartriangulations}).
Indeed, consider an embedding of a trivalent graph onto a compact oriented surface with a unique boundary component.
The condition that the faces defined by the complement of the graph are all homeomorphic to open disks implies that if we \emph{remove} the boundary of the surface, what is left is an embedding of a trivalent graph with some open edges, in the sense that they run into the boundary without including a vertex at the end.
In turn, such open ends of edges may be closed by the addition of 1-valent vertices, which should then be interpreted as being ``external'' to the map, and moreover should carry extra labelling information specifying their order of attachment to the boundary.
See \Cref{fig:boundary} for an illustration.

This leads to our definition of \emph{open rooted trivalent maps} as combinatorial maps equipped with the following data and properties:
\begin{itemize}
\item a distinguished 1-valent vertex $r \in V$, called the \emph{root};
\item an ordered list of 1-valent vertices $\Gamma = x_1,\dots,x_k \in V$ that are all mutually distinct and distinct from $r$;
\item such that the complement of $\Gamma \cup r$ in $V$ consists of 3-valent vertices.
\end{itemize}
The number $k \in \mathbb{N}$ of non-root external vertices is called the \emph{arity} of the open map, and in particular it is said to be \emph{closed} if it has arity 0.
In general, we refer to the $k+1$ 1-valent vertices as \emph{external,} and the remaining 3-valent vertices as \emph{internal}.
As a visual aid, we'll draw internal vertices as solid black vertices while for external vertices we'll use white vertices with a colored border.
Specifically for the root, we'll always represent it by a white vertex with black border.
Finally, let us note that the unrooted versions of such uni-trivalent diagrams have been studied under different names in many contexts, particularly in knot theory \cite{barnatan1995vassiliev} and physics \cite{Penrose1971}.

From this definition, it is clear that there is a trivial bijection between closed rooted trivalent maps and standard half-edge rooted trivalent maps, as shown in \Cref{fig:rootings}.\footnote{\label{foot:empty}Note that in the classical definition, the map with no half-edges is often treated as a special case and rooted ``by default''.  In this case that convention ensures that the loop map $\loopMap$ is in the image of the bijection...and shows one small advantage of using open rooted trivalent maps, that we can avoid making such special exceptions!}
Following \cite{bernardi2007bijective}, a rooted map may be called {\em $k$-near-trivalent} if its root vertex has degree $k$ and all other vertices have degree $3$, so a closed rooted trivalent map may also be called a 1-near-trivalent map.
As far as enumeration is concerned, going from a half-edge-rooted trivalent map to the corresponding open rooted trivalent map increases the number of edges by 2 (equivalently, the number of half-edges by 4).
At general arity $k$, we will be interested in enumerating open rooted trivalent maps modulo the $k!$ relabellings of the non-root external vertices, which of course is the same thing as enumerating unlabelled 1-valent-vertex-rooted (1,3)-valent maps, or equivalently, by the same bijection of \Cref{fig:rootings}, half-edge-rooted (1,3)-maps up to a size shift.

Certain graph-theoretic notions must be appropriately adapted to account for the internal vs.~external distinction.
In particular, we'll say that bridge is an \emph{internal bridge} if both of its ends are internal vertices.
Indeed, as suggested above, one can think of external vertices as being implicitly connected via a path along what was formerly the boundary of the map, so that a bridge involving external vertices is not ``morally'' a bridge.

\begin{figure}[H]
\centering
\includegraphics[scale=1.09]{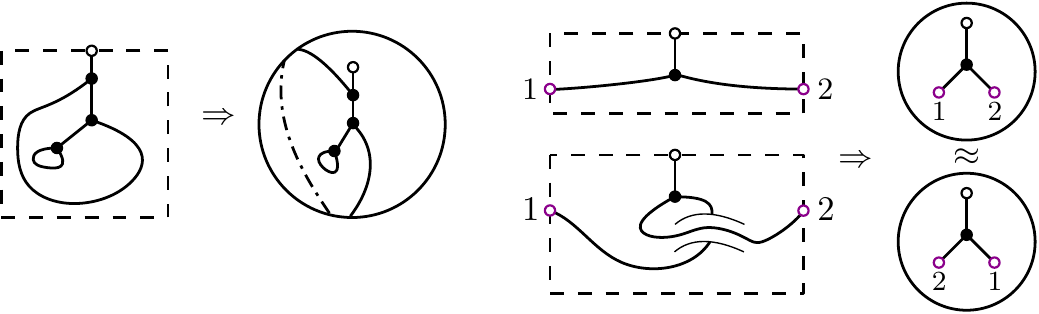}
\caption{On the left, a trivalent graph embedded on a square together with the corresponding vertex-rooted 1-near-trivalent map obtained by deleting the boundary and closing open edges. On the right, two trivalent graphs embedded in a rectangle and in a rectangle with a handle attached respectively, together with the corresponding open trivalent maps of arity 2, which are isomorphic up to relabelling of non-root external vertices.}
\label{fig:boundary}
\end{figure}

\subsection{$\lambda$-Calculus}
\label{subsec:lambda}

\paragraph*{Linear and affine $\lambda$-terms.} The {\em $\lambda$-calculus} is, among other things, a computationally universal programming language. Its terms are formed using the following grammar:

\begin{itemize}
    \item A variable (taken from an infinite set $\{x, y, z, \dots \}$) is a valid term.
    \item If $x$ a variable and $t$ is a valid term, then so is $(\lambda x.t)$. Such a term is called an {\em abstraction}, the variable $x$ in $(\lambda x.t)$ is considered {\em bound}, and we will refer to $t$  as the {\em body} of the abstraction.
    \item If $t$ and $u$ are valid terms, then so is $(t ~ u)$. Such a term is called an {\em application}.
\end{itemize}

When it aids in readability, we shall do away with some of the parentheses when writing out a $\lambda$-term, following standard conventions.
In particular, we omit outermost parentheses and associate applications to the left, while $\lambda$-abstraction is always assumed to take scope to the right by default, e.g., $\lambda x.\lambda y.\lambda z.x(yzw)$ means the same thing as $\lambda x.(\lambda y.(\lambda z.(x ((y z) w))))$.

While the theory of the $\lambda$-calculus is a vast and important subject, in this present work we will only deal with $\lambda$-terms as syntactic/combinatorial objects. 
As we shall see, even in this reduced capacity, considerations of $\lambda$-theoretic notions are quite fruitful and find natural counterparts in the realm of maps. 

Let us now introduce some technical vocabulary.
A term is called {\em closed} if all variables occurring in it are bound by some abstraction.
Otherwise such a term is called {\em open} and the variables not bound by an abstraction are referred to as {\em free}.
Two $\lambda$-terms are equivalent if, intuitively, they differ only in the names of variables.
The precise notion of equivalence, $\alpha$-equivalence necessitates the employment of {\em capture-avoiding substitutions} which we will not delve into.

An occurrence of a free or bound variable is called a \emph{use} of the variable.
A term is said to be {\em linear} if every (free or bound) variable is used exactly once.
For example the terms $\lambda x.x$ and $\lambda x. \lambda y. \lambda z. (x z)y$ are both linear while the terms $\lambda x.xx$ and $\lambda x. \lambda y. x$ are not.
The latter is an example of an \emph{affine} term, that is, a term in which every variable is used at most once.
We will refer to abstractions whose bound variable is never used as {\em unused abstractions}.

To make the above notions more precise, we can consider $\lambda$-terms as indexed explicitly by lists of free variables, defining the relation $\Gamma \vdash t$ between an ordered list of free variables $\Gamma = (x_1, \dots, x_k)$ and a linear $\lambda$-term $t$ by the following inductive rules:
\begin{center}
    \AxiomC{$\vphantom{\Gamma}$}
    \UnaryInfC{$x \vdash x $}
    \DisplayProof
    \hskip 1.5em
    \AxiomC{$\Gamma \vdash t \quad \Delta \vdash u$}
    \UnaryInfC{$\Gamma, \Delta \vdash (t~u) $}
    \DisplayProof
    \hskip 1.5em
    \AxiomC{$\Gamma, x \vdash t$}
    \UnaryInfC{$\Gamma \vdash \lambda x.t$}
    \DisplayProof
    \hskip 1.5em
    \AxiomC{$\Gamma, x, y, \Delta \vdash t$}
    \UnaryInfC{$\Gamma, y, x, \Delta \vdash t$}
    \DisplayProof
\end{center}
where we write $(\Gamma, \Delta)$ for the concatenation of two lists $\Gamma$ and $\Delta$.
From left to right, the first three rules express formation of variables, applications, and abstraction terms, respectively, while the fourth rule (called the \emph{exchange} rule) reflects the property that variables may be used in an arbitrary order in a linear term.
To define affine terms, we add one more rule (called \emph{weakening}):
\begin{center}
    \AxiomC{$\Gamma \vdash t$}
    \UnaryInfC{$\Gamma,x \vdash t$}
    \DisplayProof
\end{center}
which, reading from bottom to top, reflects the property that variables may be unused in an affine term.

\paragraph*{Subterms and one-hole contexts.}
The \emph{subterms} of a term $t$ are defined as follows:
\begin{itemize}
\item $t$ is a subterm of itself;
\item if $u$ is a subterm of $t_1$ or $t_2$ then $u$ is a subterm of $(t_1~t_2)$;
\item if $u$ is a subterm of $t$ then $u$ is a subterm of $\lambda x.t$.
\end{itemize}
The \emph{proper} subterms of $t$ are all of its subterms except for $t$ itself.
We write $u \subtm t$ to indicate that $u$ is a subterm of $t$, and $u \psubtm t$ that it is a proper subterm.

For many purposes, including ones of enumeration, it is important to distinguish between different occurrences of the same subterm (e.g., up to $\alpha$-equivalence, the \emph{identity term} $\lambda x.x$ occurs twice as a subterm of $x(\lambda y.y)(\lambda z.z)$).
A convenient way of doing so is through the notion of \emph{one-hole context} \cite{mcbride2001}.
In our setting, one-hole contexts may be defined inductively as follows:
\begin{itemize}
\item the \emph{identity context,} written $\hole$, is a one-hole context;
\item if $c$ is a one-hole context and $t$ is a term, then so are $(c ~ t)$ and $(t ~ c)$;
\item if $x$ is a variable and $c$ is a one-hole context then so is $\lambda x.c$.
\end{itemize}
The result of ``plugging'' the hole of a one-hole context $c$ with a term $u$ is a term $\plug{c}{u}$ defined inductively by:
\begin{align*}
   \plug{\hole}{u} &= u
\\ \plug{(c~t)}{u} &= (\plug{c}{u}~t)
\\ \plug{(t~c)}{u} &= (t~\plug{c}{u})
\\ \plug{(\lambda x.c)}{u} &= \lambda x.(\plug{c}u)
\end{align*}
It is easy to check that $u \subtm t$ (respectively $u \psubtm t$) iff there exists a one-hole context $c$ (resp.~$c \ne \hole$) such that $t = \plug{c}{u}$.
Moreover, by distinguishing different contexts $c_1$, $c_2$ one can distinguish between different occurrences of the same subterm $u$ within a term $t = \plug{c_1}{u} = \plug{c_2}{u}$.
Finally, there is an evident notion of composition of contexts, written $c_1 \circ c_2$, satisfying $\plug{(c_1\circ c_2)}{u} = \plug{c_1}{\plug{c_2}{u}}$ for all $u$.
Given two one-hole contexts $c_1, c_2$ $c_1$, we say $c_1$ is a {\em right subcontext} of $c_2$ if $\exists c_3 . c_2 = c_3 \circ c_1$.

Now we can define the size $\lvert t \rvert$ of a $\lambda$-term to be the number of its subterms $u \preceq t$ where we implicitly distinguish between different occurrences of the same subterm, or more formally as the number of distinct factorizations $t = \plug{c}{u}$ into a subterm and surrounding one-hole context.
Note this is equivalent to the following inductive definition:
\begin{align*}
  \lvert x \rvert &= 1 \\
  \lvert (t~u) \rvert &= 1 + \lvert t \rvert + \lvert u \rvert \\
  \lvert \lambda x.t \rvert &= 1 + \lvert t \rvert
\end{align*}
For example, $\lambda x.x$ has size two and $\lambda x.x(\lambda y.y)(\lambda z.z)$ has size 8 under this metric.
We define the size $\lvert c \rvert$ of a one-hole context similarly but assigning the identity context size zero:
\begin{align*}
  \lvert \hole  \rvert &= 0 \\
  \lvert (c~t) \rvert &= 1 + \lvert c \rvert + \lvert t \rvert \\
  \lvert (t~c) \rvert &= 1 + \lvert t \rvert + \lvert c \rvert \\
  \lvert \lambda x.c \rvert &= 1 + \lvert c \rvert
\end{align*}
so that we have the identity $\lvert \plug{c}{u} \rvert = \lvert c \rvert + \lvert u \rvert$ for all $c$ and $u$.

Finally, observe that for any term with at least one free variable $\Gamma,x \vdash t$ there is a unique one-hole context $c$ such that $\plug{c}{u} = t[u/x]$.
In this case, we say that $c$ is \emph{simple} and write $\Gamma \vdash c$.
By extension, we say that $c$ is closed if $\Gamma = \cdot$.

\paragraph*{Lambda terms as invariants of rooted maps.}\label{par:lambdaMapBij}

As recalled in \Cref{subsection:intro}, there is a natural bijection $\tau$ from rooted trivalent maps to linear lambda terms, which may be understood either via repeated root edge decomposition \`a la Tutte (as advocated in \cite{zeilberger2016linear}), or alternatively (as in the original construction \cite{bodini2013asymptotics}) as building a canonical depth-first search spanning tree of a map.
In either case, we adopt the viewpoint that the term $t = \tau(m)$ may be seen as an ``invariant'' of the map $m$, in other words that it extracts some important topological information.
In particular, $t$ describes a canonical spanning tree on $m$ obtained by deleting in the map the edges corresponding to the bound variables of the term.
We call this the \emph{$t$-tree} of $m$.
Moreover, following \cite{bernardi2007bijective}, we will call the unique path in the $t$-tree between two vertices of $m$ a {\em $t$-path}, and fixing some vertex $x$, we define the {\em parent} of $x$ to be its neighbour along the $t$-path between the root vertex and itself.
An example of two maps and their corresponding terms, with canonical spanning trees highlighted, is presented in \Cref{fig:prisms}.
\begin{figure}[h]
  \centering
  \includegraphics[scale=1.3]{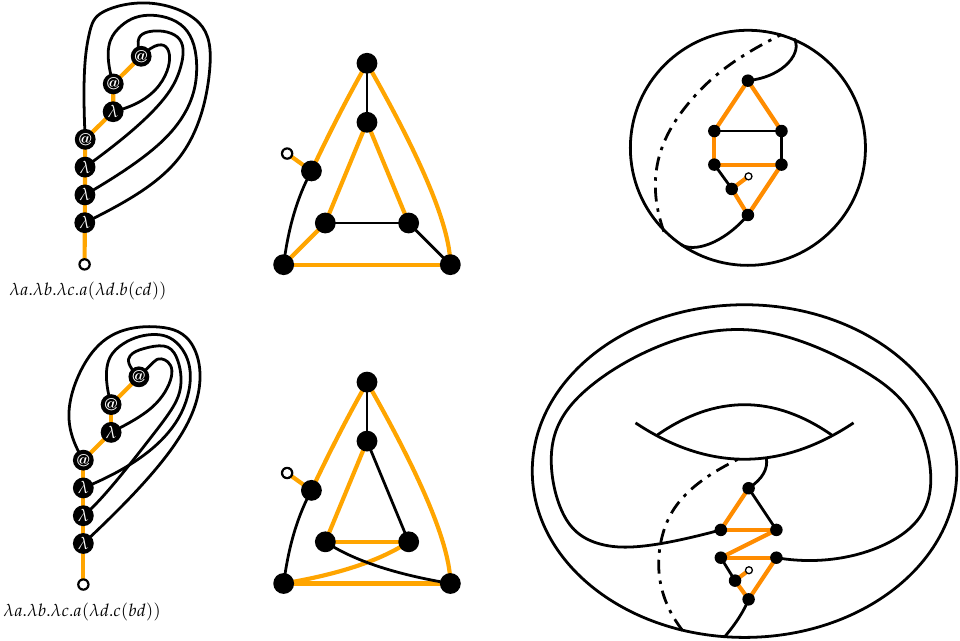}
  \caption{Two closed linear $\lambda$-terms and their equivalent rooted trivalent maps drawn as graphs with crossings and as embeddings. The spanning trees induced by the terms on each map are highlighted in orange.}
  \label{fig:prisms}
\end{figure}

The bijection $\tau$ allows us to establish a dictionary of correspondences between structural properties of linear lambda terms and rooted trivalent maps.
For example, it is not hard to see that loops in maps correspond to \emph{identity-subterms} of lambda terms, that is, subterms $\alpha$-equivalent to $\lambda x.x$, and that dually, (internal) bridges correspond to \emph{closed proper subterms} \cite{zeilberger2016linear}.
In fact, more generally any decomposition of a linear term $t = \plug{c}{u}$ into a subterm $u$ and surrounding one-hole context $c$ may be interpreted as a $(k+1)$-cut of the corresponding map, where $k$ is the number of free variables of $u$ (see \Cref{fig:cxtcut} for an illustration).
We represent the one-hole context itself as a map with a distinguished vertex, which we draw as a box, marking the hole.
In particular, a closed simple one-hole context $\cdot \vdash c$ may be considered as a (1,3)-valent map with \emph{two} marked 1-valent vertices, one representing the hole in addition to the one representing the root.
We write $\tau^{-1}(c)$ for this map, by extension of the original bijection.

Given these correspondences, we note that for many of the results in this work, the proofs may be given either purely in the language of lambda calculus or in the language of maps, and then automatically transported to the other side along a bijection.
Nevertheless, we will oftentimes include in our proofs both complementary arguments, even if not strictly required, to illuminate how the arguments translate from one class to the other.

\begin{figure}[h]
  \centering
  \includegraphics[scale=1]{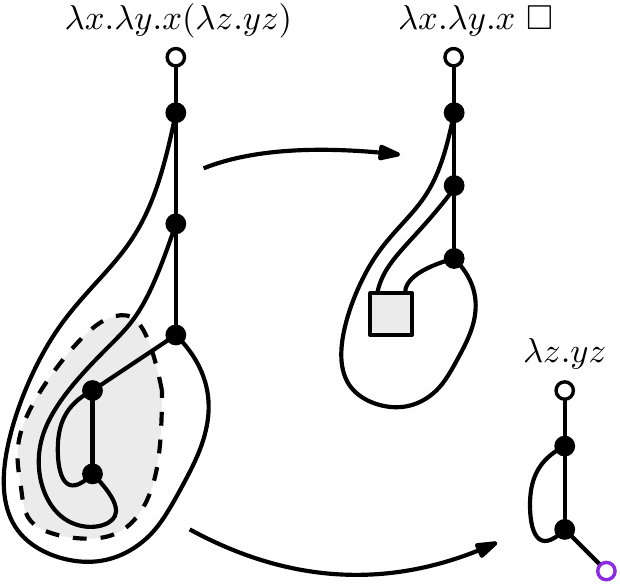}
  \caption{A decomposition of $t = \lambda x.\lambda y.x(\lambda z.y(z))$, as $c[u]$ for $c = \lambda x.\lambda y.x~\hole$ and $u = \lambda z.y z$, corresponding to a $2$-cut in the map $\tau^{-1}(t)$.}
  \label{fig:cxtcut}
\end{figure}

\subsection{Analytic combinatorics} 

\paragraph*{Combinatorial structures and the symbolic method.}
A {\em combinatorial class} is an at most countable set $\mathcal{A}$ equipped with a {\em size function} $\lvert\cdot\rvert : \mathcal{A} \to \mathbb{N}$, such that the set of elements $\mathcal{A}_n = \{ a \in \mathcal{A} \mid |a| = n\}$ of any given size $n$ has a finite cardinality $a_n = |\mathcal{A}_n|$.
To a combinatorial class $\mathcal{A}$ one can assign a power series $A(z)$, either a so-called {\em ordinary generating function} or an {\em exponential} one, defined as $\sum_n a_n \frac{z^n}{\omega_n}$ where the weight $\omega_n$ is given by $\omega_n = 1$ in the ordinary case and by $\omega_n = n!$ in the exponential case.
We'll make use of the coefficient extraction operator $[z^n] A(z)$ to denote the coefficients of $z^n$ in $A(z)$.
We refer the reader to \cite{flajolet2009analytic} for a description of the algebra of combinatorial classes and the corresponding algebra of generating functions. In particular, we'll make use of the operations of composition, disjoint union, cartesian product, and pointing which correspond to the composition, Cauchy product, and application of the operator $z \partial_z$ for powerseries respectively. We'll also make use of the {\em exponential Hadamard product} for exponential generating functions, defined as $A(z) \odot B(z) = \sum_n \frac{a_n b_n}{n!} z^n$, where $a_n = n! [z^n] A(z), b_n = n! [z^n] B(z)$. %

For a combinatorial class $\mathcal{A}$, we'll define a {\em combinatorial parameter}, or just parameter, to be a function $\chi: \mathcal{A} \rightarrow \mathbb{N}$.
Again, these can be of ordinary or of exponential type.
Then, if $a_{n,k} = \{ a \in \mathcal{A}_n \mid \chi(a) = k \}$ is the number of objects of size $n$ with parameter value $k$, the {\em bivariate generating function} $A(z,u)$ of $\mathcal{A}$ with respect to $\chi$ is $\sum_{n, k} \frac{a_{n,k}}{\omega_n \rho_k} z^n u^k$, where the weight $\rho_k$ is given by $\rho_k = 1$ if $\chi$ is of ordinary type and by $\rho_k = k!$ if it is of exponential type.
We'll say that the variable {\em $u$ marks $\chi$}.
For any $n \in \mathbb{N}$, the parameter $\chi$ determines a discrete random variable $X_n$ over $\mathcal{A}_n$: $\mathbb{P} (X_n = k) = \frac{a_{n,k}/\rho_k}{\sum_k a_{n,k}/\rho_k}$.
In such a case we'll say that $X_n$ corresponds to $\chi$ {\em taken over} $\mathcal{A}_n$.

Finally, we can also form new combinatorial classes $\mathcal{A}|_{\chi = k}$ by restricting $\mathcal{A}$ to a particular value for the parameter $\chi$, and keeping the same notion of size.
An important recurring case is when $\chi$ corresponds to a natural ``arity grading'' for $\mathcal{A}$ distinct from its size grading, and we introduce a special notation for this, writing $\mathcal{A}_{[k]}$ for the set of elements of arity $k$.
The use of iterated subscripts following these conventions should be clear from context.
For example, we write $\openTerms$ to denote the combinatorial class of all linear $\lambda$-terms, $\closedTerms$ for its restriction to the combinatorial class of closed terms (i.e., terms of arity 0), and $\closedTerms[n]$ for the finite set of closed linear terms with $n$ subterms.

\begin{table}[]
\centering
\begin{tabular}{lll}
\hline
Combinatorial Class           & Symbol & Size Notion      \\ \hline \addlinespace

Open rooted trivalent maps and linear $\lambda$-terms & $\openTerms$ & Num. of edges in map / subterms in term \\
Closed rooted trivalent maps and linear $\lambda$-terms & $\closedTerms$ & \ditto \\
Affine linear $\lambda$-terms & $\affineTerms $ & Num. of subterms    \\

Unrooted (1,3)-valent maps             & $\oneThreeMaps$ & Num. of edges    \\
Unrooted (2,3)-valent maps             & $\twoThreeMaps$ & \ditto   \\

\end{tabular}
\caption{The main combinatorial classes considered in this work.}
\label{table:classes}
\end{table}

A list of some of the main combinatorial classes to be considered in this work is given in \Cref{table:classes}.

\paragraph*{Divergent generating functions.} When enumerating various classes of non-planar maps and $\lambda$-terms, one quickly realises that the numbers involved grow rapidly (see, for example, \oeis{A062980}).

As such, the corresponding generating functions are everywhere divergent and, in particular, do not represent some function analytic at $0$. As such they are to be interpreted as purely formal power series.

Such generating functions are not always amenable to straightfoward analysis using standard tools of analytic combinatorics but instead require their own technical tools. One of our aims in this work is to develop such tools for analysing structural properties of combinatorial classes whose objects grow so rapidly so as to render their generating functions divergent.

We begin with some lemmas useful to the asymptotic and probabilistic study of such classes. The first such lemma shows that rooting a combinatorial structure does not affect the distribution of parameters over it.

\begin{lemma}\label{lemma:rootingDoesntAffectDist}
Let $\mathcal{F}$ be a combinatorial class and $\chi$ some parameter defined on it. Then the limit distribution of $\chi$ taken over $\mathcal{F}^{\bullet}_{n}$ is the same as that of $\chi$ taken over $\mathcal{F}_n$.
\end{lemma}

\begin{proof}
We have the following probability generating function for $\chi$ taken over $\mathcal{F}^{\bullet}_{n}$
\begin{equation}
  p_{n}(u) = \frac{[z^n] \mathcal{F}^{\bullet}(z,u)}{[z^n] \mathcal{F}^{\bullet}(z,1)} = \frac{n [z^n] \mathcal{F}(z,u)}{n [z^n] \mathcal{F}(z,1)} = \frac{[z^n] \mathcal{F}(z,u)}{[z^n] \mathcal{F}(z,1)}.
\end{equation}
\end{proof}

\begin{remark}
The above lemma can be iterated to show that applications of any operator of the form $z^n \partial_z^n$ result in limit distributions which converge in law to the limit distribution of $\chi$ taken over $\mathcal{F}_n$.
\end{remark}

The following lemma and its corrolary make rigorous the intuitive notion that for combinatorial classes of which the number of elements of size $n$ grows rapidly, the asymptotic number of tuplets of objects drawn from them is largely determined by the number of such tuplets for which all but one element are of the smallest possible size.

\reversemarginpar
\begin{lemma}\label{lemma:extremalCauchy}
Let $n_0 \in \mathbb{N}$ and $f(z) = \sum\limits_{n \geq n_0} f_n z^n, g(z) = \sum\limits_{n \geq n_0} g_n z^n$ be power series with $f_n, g_n > 0$ and $f_{n}/f_{n+1} = O(n^{-\sigma}), ~ g_{n}/g_{n+1} = O(n^{-\sigma})$ for some $\sigma \geq 1$ and $n \rightarrow \infty$. Then as $n \rightarrow \infty$, $[z^n] f(z)g(z) = (f_{n_0}g_{n-n_0} + f_{n-n_0}g_{n_0}) \left( 1 + O(n^{-\sigma}) \right)$.
\end{lemma}

\begin{proof}
Without loss of generality, let $n$ be odd. Then by isolating the two outer terms in the Cauchy product $f(z)g(z)$ we have

\begin{equation}
\sum\limits_{k=n_0}^{n-n_0} f_{k}g_{n - k} = \left( f_{n_0}g_{n-n_0} + f_{n - n_0}g_{n_0}\right) \left( 1 + \sum\limits_{k=n_0+1}^{n-n_0-1} \frac{f_{k}g_{n - k}}{f_{n_0}g_{n-n_0} + f_{n - n_0}g_{n_0}} \right) 
\end{equation}

In the last sum of the above expression, the extremal terms are $\frac{f_{n_0+1}g_{n-n_0-1} + f_{n-n_0-1}g_{n_0+1}}{f_{n_0}g_{n-n_0} + f_{n - n_0}g_{n_0}} = O(n^{-\sigma})$ while the rest are bound by $\frac{f_{n_0+2}g_{n-n_0-2} + f_{n-n_0-2}g_{n_0+2}}{f_{n_0}g_{n-n_0} + f_{n - n_0}g_{n_0}} = O(n^{-2\sigma})$ and there's at most $n-2n_0-3$ of them. Overall the sum is $O(n^{-\sigma})$, giving us the desired result.
\end{proof}

Applying the above lemma iteratively we obtain.
\begin{corollary}\label{lemma:tupletsAsympt}
Let $f(z) = \sum\limits_{n \geq 1} f_n z^n$ be a power series such that $f_1 = 1$, $f_n > 0$ for all $n \geq 2$, and $\frac{f_{n-1}}{f_n} = O(n^{-\sigma})$ for $\sigma \geq 1$. Then for $k, l \in \mathbb{N}^+$
$[z^n] \frac{kf(z)^l}{z^{l-1}} = k l f_n \left(1 + O(n^{-\sigma}) \right)$.
\end{corollary}

\begin{proof}
We proceed by induction. For $l = 2$ we have, by \Cref{lemma:extremalCauchy}, $[z^n] kf(z)^2 = 2k f_{1}f_{n-1}$. Dividing by $z$ amounts to a $f_{n} \mapsto f_{n-1}$ shift in the coefficients, yielding $[z^n] z^{-1} kf(z)^2 = 2k f_{n}$. 

Suppose now that the lemma holds for $l \leq L-1$. Notice that, by induction, $k \frac{f(z)^{L-1}}{z^{L-2}}$ satisfies the properties of \autoref{lemma:extremalCauchy} and with $n_0 = 1$ and $[z] \frac{f(z)^{L-1}}{z^{L-2}} = 1$.
Therefore, we may finally apply the lemma to $k \frac{f(z) f(z)^{L-1}}{z^{L-2}}$, yielding $[z^n] k \frac{f(z) f(z)^{L-1}}{z^{L-2}} = k (L-1) f_{n-1} + f_{n-1} = kLf_{n-1}$. A shift effected by dividing by $z$ completes the proof.
\end{proof}

More generaly, suppose that $f(z)$ enumerates a combinatorial class whose smallest possible structure has size not one but, say, $m$; this is the case for many of the classes discussed in the sequel. Then the above corollary becomes:

\begin{corollary}\label{lemma:tupletsAsymptShifted}
Let $f(z) = \sum\limits_{n \geq m} f_n z^n$ be a power series such that $f_m = 1$, $f_n > 0$ for all $n \geq m+1$, and $\frac{f_{n-1}}{f_n} = O(n^{-\sigma})$ for $\sigma \geq 1$. Then for $k, l \in \mathbb{N}^+$
$[z^n] \frac{kf(z)^l}{z^{m(l-1)}} = k l f_n \left(1 + O(n^{-\sigma}) \right)$.
\end{corollary}

%% file: sections/poisson.tex
In this section our aim is to explore the limit distribution of the number of bridges in trivalent maps and of closed subterms in closed linear $\lambda$-terms. Our approach will be based on combinatorial specifications of maps and terms in $\closedTerms$ respectively. As it turns out, these specifications yield differential equations governing the behaviour of our parameters of interest. To analyse these differential equations we will introduce, in \Cref{subsec:poissonSchema}, a schema providing sufficient conditions for the limit distribution of some combinatorial parameter of a divergent combinatorial class to weakly converge to a Poisson distribution of rate $1$, or a shifted version of such a distribution. Armed with this schema we will then proceed to first prove a special case of our desired result: the limit distribution of the number of loops in trivalent maps and of identity-subterms in closed linear $\lambda$-terms is $Poisson(1)$. Finally, we prove that the same holds for the number of bridges and subterms too.

As a warmup, we begin with a discussion of bridgeless trivalent maps and linear $\lambda$-terms.

\subsection{Bridgeless maps and linear $\lambda$-terms}

Let the class $\bridgelessTerms$ of {\em bridgeless} rooted trivalent maps and closed linear $\lambda$-terms be the subclass of $\closedTerms$ consisting of rooted trivalent maps with no internal bridges, or equivalently to closed linear $\lambda$-terms which have no closed proper subterms.

We begin by stating the following trivial isomorphism between $\bridgelessTerms$ and the class $\oneVarOpBridgeTerms$ of {\em one-variable-open} bridgeless linear terms, that is, linear terms $x \vdash t$ such that $t$ has no closed subterm.
Considered as maps, elements of $\oneVarOpBridgeTerms$ contain no internal bridges and exactly two external vertices (one corresponding to the root and the other to the free variable).

\begin{proposition}\label{lem:BEqlsZb} %
\begin{equation}\label{eq:BEqlsZb}
  \bridgelessTerms = \singleton\oneVarOpBridgeTerms
\end{equation}
\end{proposition}

\begin{proof}
Let $l = \lambda x.u \in \bridgelessTerms$. Then by deleting the outermost abstraction of $l$ we obtain $x \vdash u \in \oneVarOpBridgeTerms$. For the opposite direction, we have that any one-variable-open bridgeless term $x \vdash u \in \oneVarOpBridgeTerms$ uniquely yields a term $\lambda x.u \in \bridgelessTerms$.

In terms of maps, let $m \in \bridgelessTerms$ with root vertex $r$ and $a$ its unique neighbour. Then one direction of \Cref{eq:oneVarOpenBridgelessDecomp} corresponds to the observation that such a map is the one-edge map or is such that by deleting $r$ and the first edge $ax$ encountered after $ra$ in a counterclockwise tour of $a$, one obtains, after rooting at $a$, a map $m \setminus r \setminus ax \in \oneVarOpBridgeTerms$ which has two external bridges: one incident to $a$ and the other to $x$. For the other direction, we note that for a map $m \in \bridgelessTerms$, $m$ is either the one-edge map or we can use it to uniquely recreate a map $m' \oneVarOpBridgeTerms$ by adding a new edge between the root and the unique degree-1 vertex of $m$ before introducing a new root and an edge between it and the old one.
\end{proof}

To construct the bijections in the rest of this subsection, we will rely on the following lemma.

\begin{lemma}\label{lemma:LOrdedSTerms}
  Let $\Gamma,y \vdash t$ be a linear $\lambda$-term with some free variable $y$ (and possibly others $\Gamma$).  Then the set of subterms $S = \{ s \subtm t~\lvert~ y \vdash s \}$ is linearly ordered by the subterm relationship.
  Since it is moreover non-empty (with $y \in S$), it contains a unique maximal element.
\end{lemma}
\begin{proof}
We proceed by induction on $t$:

{\em Case 0: $t = y$ is a variable}.  Then $S = \{ y \preceq y \}$ is the trivial linear order on a one-element set.

{\em Case 1: $t = \lambda z.u$ is an abstraction}.  We have that $\Gamma,y,z \vdash u$, and by induction, the set $S' = \{ s \subtm u ~\lvert~ y \vdash s \}$ is linearly ordered.
But either $S = S'$ (if $\Gamma$ is non-empty) or else $S = S' \cup t$ (if $\Gamma$ is empty), in which case we can uniquely extend the linear order on $S'$ to $S$ noting that every element of $S'$ is a proper subterm of $t$.

{\em Case 2: $t = (t_1~t_2)$ is an application}.  By linearity, we have that $\Gamma_1 \vdash t_1$ and $\Gamma_2 \vdash t_2$ for some $\Gamma_1$ and $\Gamma_2$ such that $\Gamma,y$ is some shuffle of $\Gamma_1$ and $\Gamma_2$.  In particular, $y$ must appear free in one of $t_1$ or $t_2$, and without loss of generality suppose it is $t_1$ and that $\Gamma_1 = (\Gamma_1',y)$ for some $\Gamma_1'$.
Then by induction the set $S' = \{ s \subtm t_1 ~\lvert~ y \vdash s \}$ is linearly ordered, and again, either $S = S'$ (if $\Gamma$ is non-empty) or else $S = S' \cup t$ (if $\Gamma$ is empty), in which case we can uniquely extend the linear order on $S'$ to $S$.
\end{proof}

We now proceed with an equation for the class $\oneVarOpBridgeTerms$.
\begin{lemma} %
\begin{equation}\label{eq:oneVarOpenBridgelessDecomp}
  \oneVarOpBridgeTerms = \singleton + {\large \textcal{2}} \times \singleton^2 \times \oneVarOpBridgeTerms \times \oneVarOpBridgeTerms^{\bullet}
\end{equation}
where ${\large \textcal{2}} = \mathcal{E} + \mathcal{E}$ stands for the class with two neutral objects $\epsilon_1, \epsilon_2$ and $\oneVarOpBridgeTerms^{\bullet}$ denotes the pointing of $\oneVarOpBridgeTerms$, that is, the class of one-variable-open linear terms with no closed proper subterms and a marked subterm, or equivalently rooted trivalent maps with two external vertices, no internal bridges and a marked edge.
\end{lemma}
\begin{proof}
Let $x \vdash t \in \oneVarOpBridgeTerms$.
Then $t$ is either a variable, which is accounted for by the $\singleton$ summand, or else it must be an abstraction term.
Indeed $t$ cannot be an application term $t = (t_1~t_2)$ since, by linearity, either $t_1$ or $t_2$ would have to be closed, contradicting the assumption that $t$ has no closed subterms.
Assume then that $t = \lambda y.t'$ for some $x, y \vdash t'$ with two free variables.
Now, by \Cref{lemma:LOrdedSTerms}, let $y \vdash t_m$ be the subterm of $t'$ that is maximal among terms with free variable $y$, and let $x \vdash c_m$ be the corresponding context $t' = c_m[t_m]$. %
By assumption that $t$ has no closed subterms, $t_m$ must in occur in an application of the form $(t_m~u)$ or $(u~t_m)$ for some $u$, that is, the context $c_m$ must decompose as $c_m = c_m' \circ (\hole~u)$ or $c_m = c_m' \circ (u~\hole)$.
In either case, by plugging $u$ for the hole of $c_m'$ we are left with a one-variable-open term $x \vdash c_m'[u]$ with no closed proper subterms and a marked subterm.
But then the triple $(\epsilon_i, t_m, c_m'[u])$ forms an element of ${\large \textcal{2}} \times \oneVarOpBridgeTerms \times \oneVarOpBridgeTerms^{\bullet}$, where the choice of $\epsilon_i$ records which of the two cases ($c_m = c_m' \circ (\hole~u)$ or $c_m = c_m' \circ (u~\hole)$) we are in, and conversely any such triple uniquely determines a term $t = \lambda y.c_m'[t_m~u]$ or $t = \lambda y.c_m'[u~t_m]$.
This establishes the right summand on the right-hand side of \eqref{eq:oneVarOpenBridgelessDecomp}, with the extra factor of $\singleton$ accounting for the fact that we removed one application and one abstraction in passing from $t$ to $(\epsilon_i, t_m, c_m'[u])$. For a graphical example of \Cref{eq:oneVarOpenBridgelessDecomp} see \Cref{fig:oneVarOpenBridgelessDecomp}.
\end{proof}

\begin{figure}[h]
  \centering
  \includegraphics[scale=1.2]{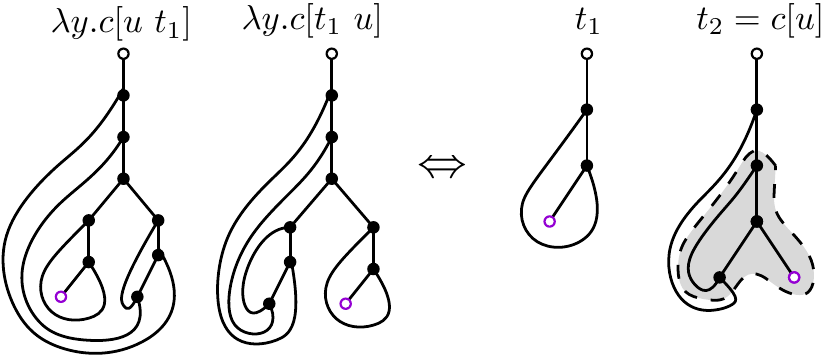}%
  \caption{Two maps in $\oneVarOpBridgeTerms$ together with the two corresponding maps in $t_1 \in \oneVarOpBridgeTerms$, $t_2 \in \oneVarOpBridgeTerms^{\bullet}$ used in their decomposition according to \Cref{eq:oneVarOpenBridgelessDecomp}.}
  \label{fig:oneVarOpenBridgelessDecomp}
\end{figure}

Combining \Cref{eq:BEqlsZb,eq:oneVarOpenBridgelessDecomp} also yields an equation for $\bridgelessTerms$.
\begin{corollary}
\begin{equation}\label{eq:bridgelessDecomp}
  \bridgelessTerms + 2\singleton\bridgelessTerms^2 = \singleton^2 + 2\singleton \times \bridgelessTerms \times \bridgelessTerms^{\bullet}
\end{equation}
\end{corollary}

The following lemma provides a bijection between non-identity/non-loop elements of $\bridgelessTerms$ and elements of $\closedTerms$ having exactly one internal bridge/closed proper subterm. %

\begin{lemma}\label{lemma:bijectionOneBridgeNoBridge}
The class $\bridgelessTerms \setminus \{\loopMap\}$ is in bijection with the subclass of $\oneBridgeTerms \subset \closedTerms$ consisting of rooted trivalent maps having exactly one internal bridge and closed linear $\lambda$-terms having exactly one proper closed subterm.
\end{lemma}
\begin{proof}
The bijection may be summarized schematically by the transformation
\[
  \lambda x.\lambda y.\plug{c}{u} \leftrightarrow \lambda x.\plug{c}{\lambda y.u}
\]
where the left-to-right direction is \emph{a priori} underspecified but can be fixed using \Cref{lemma:LOrdedSTerms}.
Visually, the bijection may also be summarized as a certain ``sliding'' operation on maps, see \Cref{fig:bijectionOneBridgeNoBridge}.

In more detail, let us write $\phi$ and $\phi^{-1}$ for the two directions of the correspondence
$\bridgelessTerms \setminus (\lambda x.x) \to \oneBridgeTerms$ and $\oneBridgeTerms \to \bridgelessTerms \setminus (\lambda x.x)$, respectively.
We begin by defining these functions on lambda terms, and then give the equivalent definition on maps.

{\em $\phi$ and $\phi^{-1}$ on lambda terms}:
Let $t \in \bridgelessTerms \setminus \{ \lambda x.x \}$ and note that $t$ must be of the form $t = \lambda x. \lambda y. t_0$.
Indeed, $t$ is necessarily an abstraction $t = \lambda x.t_1$ by \Cref{lem:BEqlsZb}, and if $t_1$ were an application $t_1 = (t_2~t_3)$ then, by linearity, one of $t_2$ or $t_3$ would be closed, a contradiction. Therefore $t_1$ is also an abstraction $t_1 = \lambda y.t_0$ by the assumption that $t \ne \lambda x.x$.
Consider now all possible ways of decomposing $t_0 = \plug{c}{u}$ into a subterm $y \vdash u$ with free variable $y$ and its surrounding context $x \vdash c$, and define
\[ \phi(t) = \lambda x.\plug{c_m}{\lambda y.u_m} \]
by taking the decomposition $t_0 = \plug{c_m}{u_m}$ such that $u_m$ is \emph{maximal,} which exists by \Cref{lemma:LOrdedSTerms}.
Since $u_m$ and $c_m$ do not contain any closed proper subterms by assumption, the term $\phi(t)$ has exactly one closed proper subterm $\lambda y.u_m$.

Conversely, if $t' \in \oneBridgeTerms$ is a term with exactly one closed proper subterm, then it necessarily decomposes as $t' = \lambda x.\plug{c}{\lambda y.u}$ for some closed subterm $\lambda y.u$ with surrounding context $\lambda x.c$, and we take $\phi^{-1}(t') = \lambda x.\lambda y.\plug{c}{u}$.
Observe that $u$ is maximal among subterms of $\phi^{-1}(t')$ with free variable $y$, which ensures that $\phi^{-1}$ really is an inverse to $\phi$.

This already completes the proof of the bijection, but we now describe it again on maps.

{\em Direction: $\phi: \bridgelessTerms \setminus \{\loopMap\} \rightarrow \oneBridgeTerms$ on maps}. Let $m \in \bridgelessTerms \setminus \{\loopMap\}$ be a bridgeless rooted trivalent map that is not a loop, let $\tau(m) = t$ be its corresponding linear term, and let $r$ be its root. Let, also, $x, y$ be the child and grandchild (in the $t$-tree of $m$) of the root $r$. Then, by bridgelessness of $m$, we have that neither of $x, y$ can be cut vertices and therefore there exists an edge $e = yz$ incident to $y$ which doesn't belong to the $t$-tree of $m$. We then construct a new map $m' = (m \setminus e) / y$ and distinguish two cases based on whether $m'$ is bridgeless or not. In the case where $m'$ is bridgeless, we create the map $\phi(m)$ by introducing a new vertex $q$ and two new edges $qq$ and $qz$ making $q$ a loop and a neighbour of $z$. In the second case in which $m'$ has bridges we note that they must all belong to unique path between $r$ and $z$ in the spanning tree $t' = t / y$ of $m'$; indeed if there was another another bridge which wasn't in the $t'$-path between $r$ and $z$ it would necessarily also be present in the initial map $m$, contradicting its bridgelessness. Therefore we can choose uniquely the bridge $b = vw$ whose endpoints lie closest to the root $r$ along the $r$-$z$ path and delete it to form the map $m' \setminus b'$ to which we then introduce a new vertex $q$ and three edges $qv, qw, qz$, making it adjacent to the former endpoints of $b'$ and also $z$. In all of the above cases the maps have a unique bridge incident to $q$, yielding an element of $\oneBridgeTerms$ as desired.

{\em Direction: $\phi^{-1}: \oneBridgeTerms \rightarrow \bridgelessTerms \setminus \{\loopMap\}$ on maps}. Conversely, let $m \in \oneBridgeTerms$ with $r, x, v$ the root and its child and grandchild (in the $\tau(m)$-tree). We denote the unique bridge of $m$ by $b = vq$ and the two connected components of $m \setminus b$ by $C_1, C_2$, with the convention that $C_1$ contains $r$ and $v$ while $C_2$ contains $q$. Note that since no other edge incident to $q$ can be a bridge, there exists some edge $e = qz$ which doesn't belong to the $\tau(m)$-tree of $m$.
If $q$ is a loop, then we form a new map $m' \setminus q$ and introduce to it a new vertex $y$ along with three new edges $xy, vy, zy$. Otherwise we form the map $m' = (m \setminus e \setminus xv) / q$ and introduce to it a new vertex $y$ and three new edges $xy$, $vy$, $zy$ making it adjacent to $x,v,$ and $z$. In either cases the new map $m'$, considered rooted at $r$, is trivalent and moreover is bridgeless since for any vertex formerly belonging to $C_2$ there now exists a path (via the newly added edge $yz$) connecting it to the any of the vertices formerly belonging to $C_1$.

Finally, to establish that the map operations $\phi, \phi^{-1}$ are inverses of eachother we let $m_1 \in \bridgelessTerms \setminus \{\loopMap\}$ be a non-loop bridgeless map, $m_2 \in \oneBridgeTerms$ be a one-bridge map, and we label their vertices $x,y,z,w,v,q$ as above. If $m_1 \setminus yz$ is bridgeless, then the map $\phi(m_1) \setminus q$ is by construction isomorphic to $(m_1 \setminus yz) / y$ which guarantees that $\phi^{-1}(\phi(m_1)) = m_1$ since $\phi^{-1}$ operates on $\phi(m_1)$ exactly by deleting the $q$ and introducing a vertex $y$ making it incident to $x, v, z$. Conversely, if $m_2 \in \oneBridgeTerms$ is a map with a unique bridge incident to a loop, then $(\phi^{-1}(m_2) \setminus yz) / y$ is by construction isomorphic to $m_2 \setminus q$ and so we have $\phi(\phi^{-1}(m)) = m$ since $\phi$ operates on $\phi^{-1}(m_2)$ by deleting $yz$, dissolving $y$, and introducing a new loop vertex $q$ making it a neighbour of $z$. Now, if $m_1 \setminus yz$ is not bridgeless, then the map $(\phi(m_1) \setminus qz) / q$ is by construction isomorphic to $(m_1 \setminus yz) / y$ and so once again by following the operation $\phi^{-1}$ on $\phi(m)$ we obtain $\phi^{-1}(\phi(m_1)) = m_1$. Finally, if $m_2 \in \oneBridgeTerms$ has no loop, then $(\phi^{-1} \setminus yz) / y$ is isomorphic to $(m_2 \setminus qz) / q$ once again giving $\phi(\phi^{-1}(m_2)) = m$ as desired. %

For graphical examples of the bijection, see \Cref{fig:bijectionOneBridgeNoBridge}. 

\end{proof}

\begin{figure}
  \centering
  \includegraphics[scale=1.5]{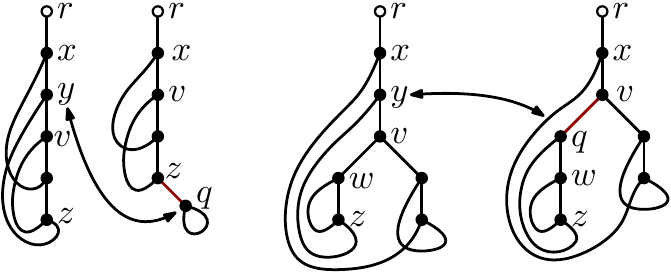}
  \caption{Two pairs of maps related by the bijection ``$\lambda x.\lambda y.\plug{c}{u} \leftrightarrow \lambda x.\plug{c}{\lambda y.u}$'' explained in \Cref{lemma:bijectionOneBridgeNoBridge}, with vertices labeled as in the proof.}%
  \label{fig:bijectionOneBridgeNoBridge}
\end{figure}

Returning to the specification given in \Cref{eq:oneVarOpenBridgelessDecomp}, we see that it yields the following differential equation satisfied by the generating function $b(z)$ of $\oneVarOpBridgeTerms$.
\begin{equation}\label{eq:bridgelessOGF1}
b(z) = z + 2z^3b(z) \frac{\partial}{\partial z}b(z)
\end{equation}
We note that, by \Cref{eq:BEqlsZb}, the generating function $B(z)$ of $\bridgelessTerms$ satisfies $B(z) = zb(z)$ and so we can focus on the easier-to-analyse $b(z)$ to obtain estimates for the asymptotic growth of $[z^n] B(z)$.

From \Cref{eq:bridgelessOGF1} one may extract the following recurrence for $b_n = [z^n] b(z)$, $n \geq 4$:
\begin{equation*}
b_n = 2\sum\limits_{k=4}^{n} b_{k-3} ~ (n - k + 1) ~  b_{n - k +1}.
\end{equation*}
We can obtain a lower bound for the sequence $b_n$ (which generates \oeis{A267827}) by first isolating the summands corresponding to $k = 4$ and $k = n$ in the above recurrence, and then translating it to a differential equation taking into account the initial values $b_1 = 1, b_2 = b_3 = 0$ to obtain
\begin{equation*}
\underline{b}(z) = z - 2z^4 + 2z^3\underline{b}(z) + 2z^4 \frac{\partial}{\partial z} \underline{b}(z)
\end{equation*}
where now $\underline{b}(z)$ is such that $[z^n] \underline{b}(z) \le b_n$ for all $n$.
Let $\underline{\hat{b}}(z) = \sum_{n} \frac{\underline{b}_n}{n!} z^n$ be the {\em Borel transform} of $\underline{b}$.
Then $\underline{\hat{b}}$ satisfies
\begin{equation*}
\underline{\hat{b}}(z) = z^2 - 2z\underline{\hat{b}}(z) + \frac{\partial^2}{\partial z^2} \underline{\hat{b}}(z)
\end{equation*}
which, for initial conditions $\underline{\hat{b}}(0), \underline{\hat{b}}'(0) = 0$, has a unique solution expressible in terms of the Airy functions
\begin{equation*}
\underline{\hat{b}}(z) = \frac{z}{2} + \frac{AiryBi(2^{1/3} z) 2^{2/3} 3^{1/3} \pi}{12 \Gamma(\frac{2}{3})} - \frac{AiryAi(2^{1/3} z) 2^{2/3} 3^{5/6} \pi}{12 \Gamma(\frac{2}{3})}
\end{equation*}
Using the following closed form of Taylor series for $AiryAi, AiryBi$ at $z=0$ 
\begin{equation*}
\frac{1}{3^{c} \pi} \sum_{n = 0}^{\infty} \frac{\Gamma\left(\frac{(n+1)}{3}\right)}{n!} (3^{1/3} z)^n \left\lvert \sin\left( \frac{2\pi(n+1)}{3} \right) \right\rvert
\end{equation*}
where $c = 2/3$ for $AiryAi$ and $c = 1/6$ for $AiryBi$, one obtains the following lower bound to $b_n$ for $n = 3k + 1$:
\begin{equation}
b_n \geq n![z^n] \underline{\hat{b}}(z) = \frac{6^{k + 1} \Gamma\left( k + \frac{2}{3} \right)}{12\Gamma\left(\frac{2}{3}\right)} = \omega(5^k k!)
\end{equation}

From this rough lower bound one is led to conjecture that bridgeless terms might make up a considerable percentage of all closed linear $\lambda$-terms (of which there are $O(6^k k!)$). This, coupled with \Cref{lemma:bijectionOneBridgeNoBridge}, gives us a first clue of what the limit distribution looks like: it must obey $\mathbb{P}[X = 0] = \mathbb{P}[X = 1]$ and for $k \geq 2$, $\mathbb{P}[X = k]$ seems to decay fast. These observations suggest that we are looking at a $Poisson(1)$ limit distribution for the number of bridges in $\closedTerms$. Indeed, the following subsection provides the tool which will help us prove this conjecture.

\subsection{Poisson distributions from differential equations}\label{subsec:poissonSchema}
\begin{lemma}[Poisson Schema]\label{lemma:poissonSchema} 
Let $F(z,u)$ be some bivariate powerseries. Furthermore, suppose that

\begin{itemize}

  \item The powerseries $F(z,u)$ satisfies a first order differential equation with respect to u, which may be rearranged as

  \begin{equation}
    \frac{\partial}{\partial u} F(z,u) = W_1(z, u, F(z,u))
  \end{equation}

  where $W_1$ is a rational function of $z, u, F(z,u)$.

  \item There exists a constant $\sigma \in \mathbb{R}^{>0}$ and a constant $b \subseteq {N}$\footnote{The periodicity condition here is not essential to the lemma at all, the same holds for powerseries with non-zero coefficients for all $n \geq 0$. However given the fact that in this section we shall deal with powerseries which have non-zero coefficients only for $n \equiv 2 \pmod{3}$, we include this condition for ease of use.}
  \begin{equation}
  [z^n] F(z,1) = 0 \qquad n \not\equiv b \pmod{a}
  \end{equation}
  \begin{equation}\label{eq:poissonRapidGrowth}
  \frac{[z^{n-b}] F(z,1)}{[z^n] F(z,1)} \sim O(n^{-\sigma}) \qquad n \equiv b \pmod{a}
  \end{equation}
\end{itemize} 

Define $W_N(z,u,f)$, for $N \in \mathbb{N}^{+}$, such that $W_N(z,u,F(z,u)) = \frac{\partial^N}{\partial u^N} F(z,u)$. Then if for all $N \geq 1$:

\begin{equation}\label{eq:poissonDiffByF}
    [z^n]  \left. \left(\partial_f W_{N+1} \cdot W_1 \right)\right\lvert_{f = F(z,1),~ u=1} \sim \left. [z^n] W_N \right\lvert_{f = F(z,1),~ u=1},
\end{equation}
\begin{equation}\label{eq:poissonDiffByV}
    [z^n]  \left. \left( \partial_v W_{N+1} \right) \right\lvert_{f = F(z,1),~ u=1} = O\left([z^{n-b}] F(z,1)\right),
\end{equation}
\begin{equation}\label{eq:poissonBootstrap}
  [z^n] \left. W_1 \right\lvert_{f = F(z,1),~ u=1} \sim [z^n] F(z,1).
\end{equation}

Then the random variables $X_n$ whose probability generating function is given by

\begin{equation}
p_n(u) = \frac{[z^n] F(z,u)}{[z^n] F(z,1)}
\end{equation}

converge in distribution to a $Poisson(1)$-distributed random variable $X$.

\end{lemma}

\begin{proof}
We have, by the chain rule, that
\begin{equation}
W_N = \frac{\partial}{\partial f}W_{N-1} ~W_1 + \frac{\partial}{\partial v} W_{N-1}.
\end{equation}
Evaluating the above at $u=1, f = F(z,1)$ and extracting coefficients we have, by \Cref{eq:poissonDiffByF,eq:poissonDiffByV},
\begin{equation}
[z^n] W_N(z,1,F(z,1)) \sim [z^n] W_{N-1}(z,1,F(z,1)) + O([z^{n-b}] F(z,1)) \qquad n \equiv b \pmod{a}
\end{equation}
By \Cref{eq:poissonBootstrap} we have that $[z^n] W_N(z,1,F(z,1))$ grows asymptotically as $[z^n] W_1(z,1,F(z,1)) \sim [z^n] F(z,1)$ and by \Cref{eq:poissonRapidGrowth} we have $O([z^{n-b}] F(z,1)) = o([z^n] F(z,1))$. Therefore we have the following chain of asymptotic equivalences:
\begin{equation*}
[z^n] W_N(z,1,F(z,1)) \sim  [z^n] W_{N-1}(z,1,F(z,1)) \sim \dots \sim [z^n]W_1(z,1,F(z,1)) \sim [z^n]F(z,1) \qquad n \equiv b \pmod{a}
\end{equation*}
which translates to the following chain of asymptotic equalities between the factorial moments of $X_n$
\begin{equation}
\frac{[z^n] W_N(z,1,F(z,1))}{[z^n] F(z,1)} \sim \frac{[z^n] W_{N-1}(z,1,F(z,1))}{[z^n] F(z,1)} \sim \dots \sim \frac{[z^n] W_{1}(z,1,F(z,1))}{[z^n] F(z,1)} \sim \frac{[z^n] F(z,1)}{[z^n] F(z,1)} = 1  \quad n \equiv b \pmod{a}
\end{equation}
Using the following relation between $r$-th factorial and power moments of a random variable:
\begin{equation}
\mathbb{E}(X_{n}^r) = \sum_{k=0}^{r} \mathbb{E}\left({X_{n}^{\underline{r}}}\right) {r \brace l}
\end{equation}
where $X_{n}^{\underline{r}} = X_{n}(X_{n}-1)\dots(X_{n}-r+1)$, we have that $ \lim_{n \rightarrow \infty} \mathbb{E}(X_{n}^r) = \mathbb{E}(X^r)$. Since the moment generating function $\mathbb{E}(e^{sX})$ exists in a neighbour of 0, we have that the $Poisson(1)$ distribution is determined by its moments (see \cite[Theorem 30.1]{billingsley2008probability}) and so, by the Markov-Fréchet-Shohat moment convergence theorem (see \cite[Theorem 30.2]{billingsley2008probability},\cite[Theorem C.2]{flajolet2009analytic}), we obtain our desired result. %
\end{proof}

\subsection{Identity subterms of closed linear terms and loops in trivalent maps.}

The goal of this subsection is to investigate the limit distribution of the number of identity-subterms, that is subterms equivalent to $\lambda x.x$, in closed linear $\lambda$-terms or equivalently the number of loops in trivalent maps. 

\begin{figure}
  \centering
  \includegraphics[scale=0.7]{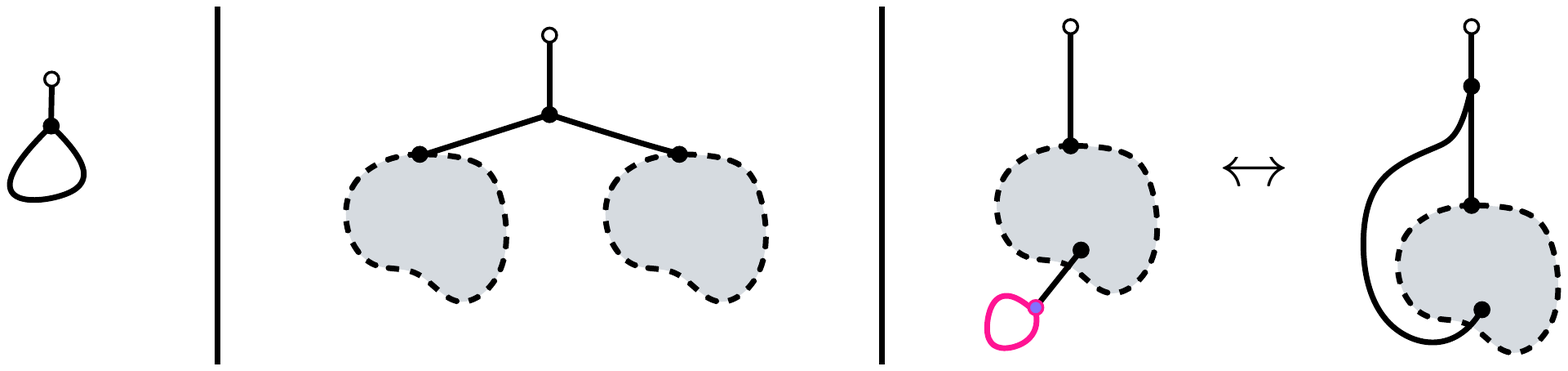}
  \caption{The three main cases of \Cref{eq:id_distr_u_diff_eq}: $z^2$ (left), $\lamClosedIdOGF(z,u)^2$ (middle), and $\frac{\partial}{\partial u} \lamClosedIdOGF(z,u)$ (right), together with a transformation of the last into an abstraction term.}
  \label{fig:loopEq}
\end{figure}

We begin by presenting a specification of the bivariate generating function $G(z,u)$ of closed linear $\lambda$-terms where $u$ tags identity subterms.

\begin{lemma}
Let $\lamClosedIdOGF(z,u)$ be the bivariate generating function enumerating closed linear $\lambda$-terms with respect to size and number of identity-subterms. Then,
\begin{equation}\label{eq:id_distr_u_diff_eq}
\lamClosedIdOGF(z,u) = (u-1)z^2 + z\lamClosedIdOGF(z,u)^2 + \frac{\partial}{\partial u} \lamClosedIdOGF(z,u)
\end{equation}
\end{lemma}
\begin{proof}
Consider the following rearrangement of the above equation:
\begin{equation}\label{eq:id_distr_u_diff_eq_combi}
    \lamClosedIdOGF(z,u) + z^2 =  uz^2 + z\lamClosedIdOGF(z,u)^2 + \frac{\partial}{\partial u} \lamClosedIdOGF(z,u)
\end{equation}
which may be intepreted combinatorially in the following fashion: a closed linear $\lambda$-term is either a term of the form $\lambda x.x$, or an application of two closed terms, or is formed by taking some closed linear $\lambda$-term $t$, selecting some identity-subterm $s$, and replacing it with a free variable which is then bound by a newly-introduced lambda on top, to form the term $\lambda a.t[s := a]$, as on the right side of \Cref{fig:loopEq}.

Two subtler points of the abstraction case of the above construction are worth discussing. Firstly, note the use of the plain differential operator as opposed to the more usual pointing ($z \frac{\partial}{\partial u}$) operator. This is due to the fact that identity-terms are of size $2$ which, after being pointed-at and replaced by a free variable, leads to a reduction of the size of the term by $1$. This is balanced by the introduction of the new abstraction which, having size $1$, causes the overall size to remain invariant. Secondly, we have the following ``edge case'': applying the abstraction construction to the term $\lambda x.x$ leaves it invariant but removes its $u$ mark. Such a term doesn't belong in $\lamClosedIdOGF(z,u)$, so we have to consider it separately, which is why the left-hand side of \Cref{eq:id_distr_u_diff_eq_combi} is $\lamClosedIdOGF(z,u) + z^2$ instead of just $\lamClosedIdOGF(z,u)$.

It is also of interest to note that this construction is highly reminiscent of the construction of the generating function enumerating open linear $\lambda$-terms with $u$ tagging free variables presented in \Cref{eq:diff_open_linlams}, with the term $uz^2$ enumerating the identity-term $\lambda x.x$ in \Cref{eq:id_distr_u_diff_eq} playing a role analogous to the term $uz$ enumerating the term $x$ in \Cref{eq:diff_open_linlams}.

\end{proof}

Before we proceed with the main result of this subsection, let us first address a technicality regarding the coefficient asymptotics of a power series which exhibit periodicities of the form

\begin{equation*}
F(z) = \sum\limits_{n = 0}^{\infty} f_{an + b} ~ z^{an + b},
\end{equation*}
with $f_n = 0$ for $n \not\equiv b \mod a$, $a \in \mathbb{N}^+, b \in \mathbb{N}$. For such a power series, expressions of the form $[z^n] F(z) \sim f(n),$ where $f(n)$ is some function of $n$, are not well defined since they represent divergent limits. To correct for this periodic appearance of zeros in $[z^n] F(z)$, one can manipulate the powerseries so as to ``skip'' the problematic values of $n$: 
\begin{equation}\label{eq:periodicSimFixed}
[z^n] \frac{ F(z^{1/a})}{z^{b/a}} \sim f(n).
\end{equation}

To avoid having to use cumbersome case-based notation and/or manipulation of powerseries, we will instead follow \cite{flajolet2009analytic} and write 

\begin{equation}
[z^n] F(z) \sim f(n),\qquad n \equiv b \pmod{a}
\end{equation}
to mean that the limit is taken for the subsequence with $n = ak + b, k \in \mathbb{N}$, and is zero otherwise.

In the specific case of the generating functions of closed linear $\lambda$-terms, $\lamClosedOGF(z)$, we have that: 
\begin{theorem}[{\cite[Theorem 3.3]{bodini2013asymptotics}}]\label{thm:closedTermAsympts}
The number of closed linear $\lambda$-terms of size $p \equiv 2 \pmod{3}$ is

\begin{equation}
[z^p] \lamClosedOGF(z) \sim \frac{3}{\pi} ~ 6^n n! ,\qquad p = 3n + 2.
\end{equation}
\end{theorem}

\begin{figure}[h]
  \centering
  \includegraphics[scale=0.7]{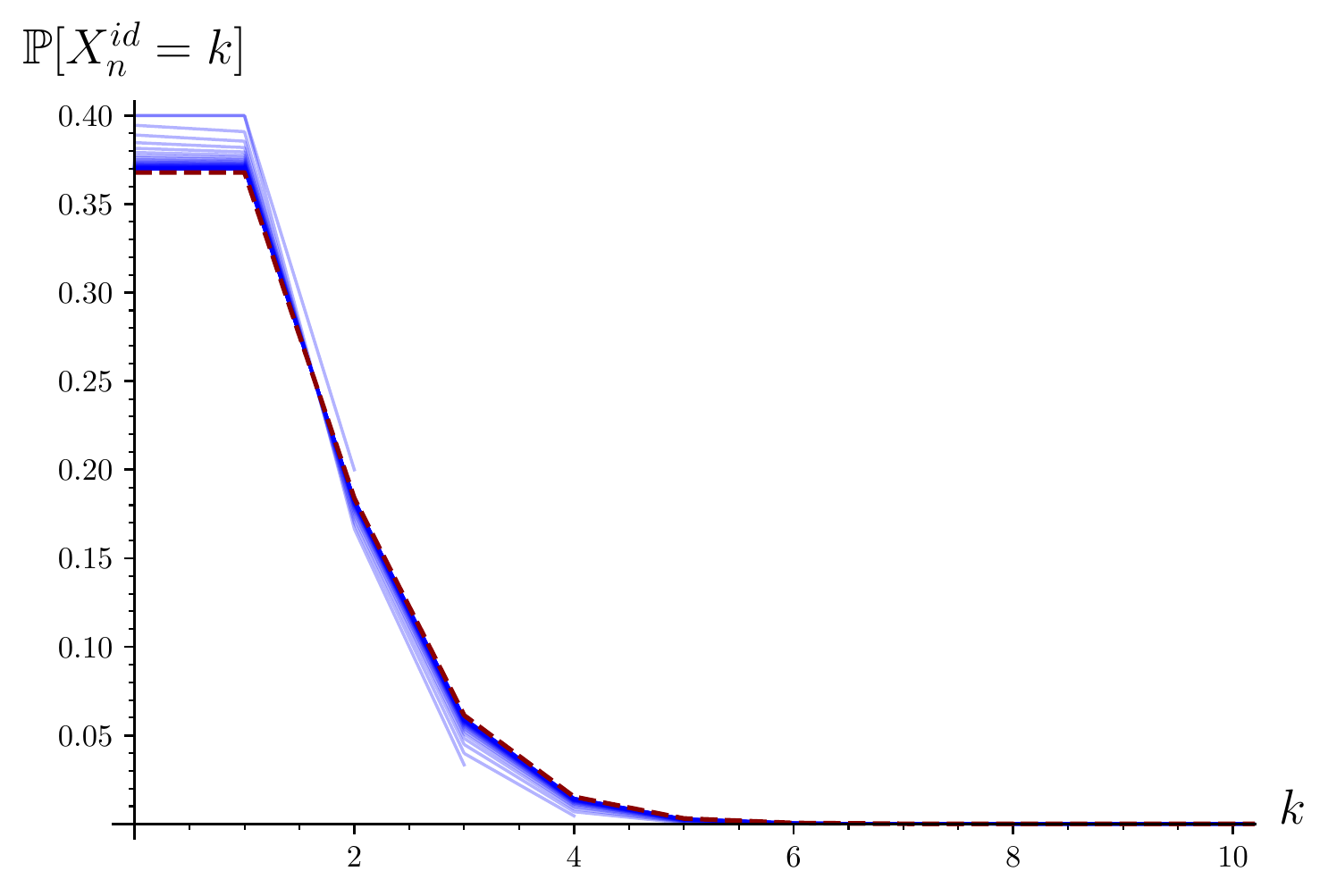}
  \caption{Density plots of $X^{id}_n$ for $n = 2 \dots 100$ along with $Poisson(1)$ in red.}
  \label{fig:distPlotId}
\end{figure}

We now proceed with the main result of this section.
\begin{theorem}
Let $\chi_{id}$ be the parameter corresponding to the number of identity-subterms in a closed linear $\lambda$-term or, equivalently, to loops in rooted trivalent maps. Then for the random variables $X^{id}_{n}$ corresponding to $\chi_{id}$ taken over $\closedTerms[n]$ respectively we have:

\begin{equation}
X^{id}_{n} \overset{D}{\rightarrow} Poisson(1)
\end{equation}
\end{theorem}

\begin{proof}
Let $W_N(z,u,f)$ be such that $W_N(z,u,\lamClosedIdOGF) = \frac{\partial^N}{\partial u^N} \lamClosedIdOGF(z,u)$. For $N=1$ we have, by \Cref{eq:id_distr_u_diff_eq},
\begin{equation*}
W_1 = f - zf^2 - (u-1)z^2.
\end{equation*}
By induction, will show that $W_N = \partial_f W_{N-1} W_1 + \partial_u W_{N-1}$ consists of terms of the form $f + R$ where $R$ is a finite sum of monomials of the form $c_i f^{j} u^{k} z^{l}$ where $c_i$ is a constant, $j,k \geq 0$, and  $l \geq 1$.

Indeed, $W_1$ is of this form and for every $N \geq 2$, if $W_{N-1}$ is of this form, we have 

\begin{align*}
\frac{\partial}{\partial f} W_{N-1} W_1 + \frac{\partial}{\partial u} W_{N-1} = \left(1 + \frac{\partial}{\partial f} R \right) \cdot W_{1} + \frac{\partial}{\partial u} R.
\end{align*} 

But a term-by-term differentiation of $R$ with respect to either $f$ or $u$ maintains all desired properties of $R$ and so finally, by grouping together the monomials of $\partial_f R$ and $\partial_u R$ as $R'$ and noticing that $W_1$ contributes the sole $f$ summand of $W_{N}$ we see that $W_{N} = f + R'$ with $R'$ as desired.

An application of \Cref{lemma:tupletsAsymptShifted} shows that $[z^n] z^k \lamClosedIdOGF(z,1)^l$ for $k \geq 1, l \geq 0$ is asymptotically neglibible compared to $[z^n] \lamClosedIdOGF(z,1)$, for $n \equiv 2\pmod{3}$. Since the $R$ summand of $W_{n-1} \lvert_{u=1}$ consists precisely of a finite number of such terms we have $\partial_f R = O\left([z^{n-3}] \lamClosedIdOGF(z,1)\right)$ and $\partial_u R = O\left([z^{n-3}] \lamClosedIdOGF(z,1)\right)$, again for $n \equiv 2\pmod{3}$. %

Finally, we have that 
\begin{align*}
[z^n]  \left. \frac{\partial}{\partial f} W_N \cdot W_1 \right\lvert_{f = \lamClosedIdOGF(z,1),~ u=1} &\sim [z^n] \left. W_N\right\lvert_{f = \lamClosedIdOGF(z,1),~ u=1}\\
\left. [z^n] \frac{\partial}{\partial v} W_N \right\lvert_{f = \lamClosedIdOGF(z,1), v=1} &= O\left(g_{n-3}\right) ,\quad\qquad n \equiv 2\pmod{3} \\
[z^n]  \left. W_1 \right\lvert_{f = \lamClosedIdOGF(z,1),~ u=1} &\sim [z^n] \lamClosedIdOGF(z,1) ,\qquad n \equiv 2\pmod{3}
\end{align*}
so by \Cref{lemma:poissonSchema} we obtain the desired result.

\end{proof}

\subsection{Closed proper subterms of closed linear $\lambda$-terms and internal bridges in trivalent maps}

Identity-subterms are the simplest case of a more general notion, that of closed proper subterms. Equivalently, they are the smallest possible rooted trivalent map which can appear at one side of a bridge. In this subsection we are going to generalise the result of the previous subsection by investigating the limit distribution of closed proper subterms of linear $\lambda$-terms and of internal bridges in rooted trivalent maps. As in the previous subsection, we will rely on a specification for the bivariate generating function $\lamClosedSubOGF(z,v)$ of closed linear lambda terms where $v$ tags closed proper subterms.

Before presenting said specification, we begin by defining the following classes which will provide the building blocks for our decomposition of $\closedTerms$.

\begin{definition}[The class $\contextClass$ of simple closed non-trivial one-hole contexts]
Let $\contextClass$ be the combinatorial class consisting of simple closed one-hole contexts other than $\hole$. In terms of maps, elements of $\contextClass$ correspond to \emph{doubly-rooted trivalent maps,} which have the following structure:

\begin{itemize}
   \item There is a distinguished vertex $r$ of degree 1, called the \emph{root vertex} (which as usual we'll draw as a white vertex with black border).
   \item There is a distinguished vertex $v$ of degree 1, called the \emph{box vertex} (which as at the end of \Cref{subsec:lambda} we'll draw as a gray vertex with black border).
   \item All other vertices have degree 3, and there is at least one such vertex.
\end{itemize} 
Elements of $\contextClass$ will be enumerated by their size as contexts, which equals the number of edges in the corresponding map \textbf{minus $1$}. 
\end{definition}

\begin{definition}[The class $\qClass$]
Let $\qClass \subseteq \contextClass$ be the combinatorial class formed by restricting $\contextClass$ to one-hole contexts of which every proper right subcontext is either $\hole$ or has a free variable.
Viewed as maps, the elements $q \in \qClass$ have the additional property that:

\begin{itemize}
   \item No edge belonging to the $\tau(q)$-path from $r$ to $v$ is an internal bridge, where $r$ is the root vertex, $v$ is the distinguished 1-valent vertex, and $\tau(q)$ is the one-hole context corresponding to $q$.
\end{itemize} 
\end{definition}

\begin{figure}[h]
  \centering
  \includegraphics[scale=1.7]{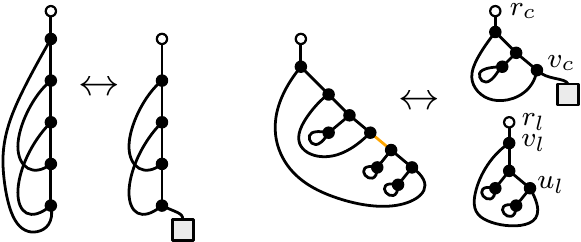}
  \caption{Maps corresponding to elements of $\closedTerms$ and their decompositions. For the subfigure on the left, we have a decomposition of $\lambda x.\lambda y. \lambda z. x z y$ into $\lambda y. \lambda z. \hole z y$, which is of the form $\singleton^2\qClass$. For the map on the right we have a decomposition of $\lambda x. \lambda y. x (\lambda a. a) (\lambda b. b) y (\lambda c. c)$ into $\lambda y. \hole y (\lambda c.c)$ and $\lambda x. x (\lambda a. b) (\lambda b.b)$ of the form $\qClass\closedAbsClass$. Vertices in the right subfigure are labelled as in the proof of \Cref{lemma:qdecompq} and the internal bridge $b$ is highlighted.}
  \label{fig:lqDecompFigure}
\end{figure}

Let $\closedAbsClass$ be the class consisting of {\em closed non-identity linear abstraction terms}; note that the map corresponding to an element of $\closedAbsClass$ is exactly a map of size bigger than 2 with the property that deleting the root and its unique neighbour leaves a connected map.

\begin{lemma}\label{lemma:qdecompq}
For the combinatorial classes $\closedAbsClass, \qClass$  we have
\begin{equation} \label{eq:ldecompq}
  \closedAbsClass = \singleton^2\qClass + \qClass\closedAbsClass
\end{equation}
At the level of generating functions we have
\begin{equation}\label{eq:ldecompqOGF}
\abstractionSubOGF(z,v) = z^2\qOGF(z,v) + \qOGF(z,v)\abstractionSubOGF(z,v)
\end{equation}
where $\abstractionSubOGF(z,v)$ is the generating function enumerating elements of $\closedAbsClass$ and $\qOGF(z,v)$ the one enumerating elements of $\qClass$ with $v$ in both cases tagging proper subterms of terms, or equivalently internal bridges in rooted trivalent maps.
\end{lemma}

\begin{proof}

We establish the desired bijection \Cref{eq:ldecompq} by providing two mappings $\psi, \psi^{-1}$ between the left-hand side and the right-hand side and vice-versa.
We then verify that this is indeed a bijection and that it leads to the equality generating functions presented in \Cref{eq:ldecompqOGF}.

{\em Direction $\psi: \closedAbsClass \rightarrow \singleton^2\qClass + \qClass\closedAbsClass$.} Let $l$ be some element of $\closedAbsClass$, viewed as a closed abstraction term and write $l = \lambda x.c[x]$ for some one-hole context $c$.
We now distinguish cases based on the nature of right subcontexts of $c$:

{\em Case 1.1: All proper right subcontexts of $c$ are either $\hole$ or have a free variable.} In this case $c$ is an element of $\qClass$ by definition. To account for the change of size resulting from the deletion of the outermost abstraction and its bound variable, a $\singleton^2$ factor is introduced, yielding the $\singleton^2\qClass$ summand of \Cref{eq:ldecompq}. Therefore we let $\psi(l) = (z^2, c)$, for $z^2 \in \singleton^2$.

In terms of maps, let $m$ be the map corresponding to $l$ and $r$ be the root, $a$ be the vertex representing the outermost abstraction, and $v$ be the vertex representing its bound variable. The map $m_c$ corresponding to the context $c$ is obtained from the map $(m \setminus av) / a$ by adding a new box vertex which is made adjacent to $v$. The current case then corresponds to the non-existence of an internal bridge along the $c$-path from the root of $m_c$ to its unique box vertex. As such, $m_c$ yields a unique member of $\qClass$.
 
{\em Case 1.2: There exists a proper closed right subcontext of $c$ other than $\hole$.} In this case let $c'$ be the biggest such proper right subcontext, with $c$ decomposing as $c = c_1\circ c'$ for some $c_1 \neq \hole$. 
Let $c''$ be an arbitrary proper right subcontext of $c_1$. Then $c_1 = c_0\circ c''$ for some context $c_0$ and $c = c_0\circ c'' \circ c'$. Notice then that any such $c''$ must either be $\hole$ or have a free variable, for otherwise $c''\circ c'$ would be a closed proper right subcontext of $c$ bigger than $c'$, violating maximality of of $c'$. Since $c''$ was an arbitrary proper right subcontext of $c_1$, we have that $c_1$ belongs to $\qClass$ by definition.

As for $c'$, since by linearity $x$ does not appear in $c'$, we have that $\lambda x. c'[x]$ belongs to $\closedAbsClass$.

Together, these yield the $\qClass\closedAbsClass$ summand of \Cref{eq:ldecompq} and so we let $\psi(l) = (c_1, \lambda x. c'[x])$.

In terms of maps, we have in this case that the map $m_c$ corresponding to $c$, constructed as detailed above, has at least one internal bridge in the $c$-path between the root and the box vertex. Let $b$ be the unique such bridge closest to the root (in terms of distances along the $c$-tree of $m_{c}$). Then the deletion of $b$ results in two connected components which yield $m_{q} = \tau^{-1}(q)$ and (after some manipulation) $m_{\lambda .c'[x]} = \tau^{-1}(\lambda .c'[x])$.

{\em Direction $\psi^{-1}: \singleton^2\qClass + \qClass\closedAbsClass \rightarrow \closedAbsClass$.} For the other direction, let $t \in \singleton^2\qClass + \qClass\closedAbsClass$. We distinguish the following two possibilities based on the nature of $t$.

{\em Case 2.1: $t \in \singleton^2\qClass$}. In this case, $t$ consists of an element of $\singleton^2$ together with a context $c \in \qClass$. Then, assuming without loss of generality that $x$ doesn't appear in $c$, $t' = \lambda x. c[x]$ together with its corresponding map belong in $\closedAbsClass$, with $\singleton^2$ accounting for the increase in size $\lvert t' \rvert = \lvert c \rvert + 2$, so that the preimage of $t$ is $\psi^{-1}(t) =(\bullet^2, \lambda x.c[x])$.

In terms of maps, let $r, v_{\hole}$ be the root and box vertices of $m_c = \tau^{-1}(c)$ respectively. Then this case amounts to identifying $r$ with $v_{\hole}$ and introducing a new root vertex $r'$ which we make adjacent to $r$.

{\em Case 2.2: $t \in \qClass\closedAbsClass$}. In this case, $t$ consists of a pair of $c \in \qClass$ and $l \in \closedAbsClass$. Suppose, without loss of generality, that $l = \lambda x. d[x]$. Then $\lambda x. c[d[x]]$ and its corresponding map belong in $\closedAbsClass$ and $\psi^{-1}(t) = \lambda x. c[d[x]]$.

In terms of maps, let $r_l$ be the root of $m_l = \tau^{-1}(l)$, let $v_l$ be the unique neighbour of $r_l$ and let $u_l$ be the neighbour of $v_l$ which is furthest from $v_l$ in the $l$-tree of $m$. Let also $r_c, v_{\hole}, v_c$ be the root, box vertex, and the unique neighbour the box vertex in $m_c = \tau^{-1}(c)$, respectively. Then by taking the disjoint union $(m_l \setminus r_l \setminus v_lu_l) + (m_c \setminus v_{\hole}) + r_{new}$, where $r_{new}$ is a new vertex, introducing two new edges $r_{new} r_c, r_c u_l,$ and identifying $v_c$ with $v_l$, we form a rooted trivalent map as desired. See the right subfigure of \Cref{fig:lqDecompFigure} for an example.

We now proceed to verify that $\psi^{-1}$ is a two-sided inverse of $\psi$. Let $l \in \closedAbsClass$ and write $l = \lambda x.c[x]$. We then have $\psi^{-1}(\psi(t)) = \psi^{-1}((\bullet^2, c)) = \lambda x.c[x]$ or $\psi^{-1}(\psi(t)) = \psi^{-1}((q, \lambda x. c'[x])) = \lambda x. q[c' [x]] = \lambda x. c[x]$ depending on the whether we fall under Case $1.1$ or $1.2$ respectively. Conversely, if $(\bullet^2, c) \in \singleton^2\qClass$, we have $\psi(\psi^{-1}((\bullet^2,c))) = \psi(\lambda x. c[x]) = (\bullet^2, c)$ since $c \in \qClass$ satisfies the properties of Case $1.1$. Lastly, if $(c, \lambda x. d[x]) \in \qClass\closedAbsClass$, $\psi(\psi^{-1}((c,\lambda x. d[x]))) = \psi(\lambda x. c[d[x]]) = (c, \lambda x. d[x])$ since $d$ by construction is closed and so we fall under Case $1.2$.

Finally, we note that in both directions of the above bijection, no new bridges/closed subterms are created, so that the number of bridges/closed subterms in the left-hand-side of \Cref{eq:ldecompq} equals that on the right, yielding \Cref{eq:ldecompqOGF}.
\end{proof}

\begin{figure}[h]
  \centering
  \includegraphics[scale=2]{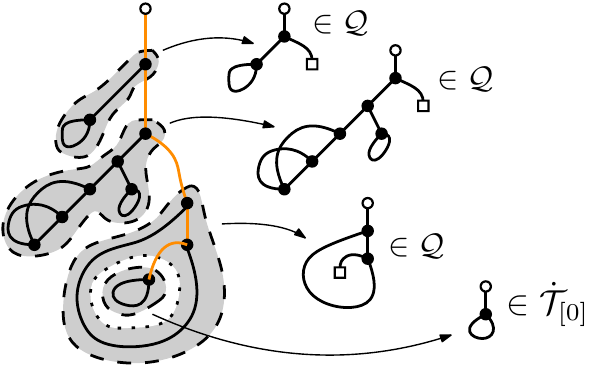}
  \caption{A map $m$ along with a highlighted path between the root and a bridge in $m$. 
  The decomposition of $m$ into components yielding maps in $\qClass$ and $\closedTerms$ is depicted by grey borders.}
  \label{fig:freeVarsEvolution}
\end{figure}

Now, in order to obtain the bivariate generating function for closed linear lambda terms by size and number of closed proper subterms, 
Let $\Theta_{sub}\closedAbsClass$ be the class of {\em closed linear $\lambda$-terms with a distinguished closed proper subterm} or, equivalently, {\em rooted trivalent maps with a distinguished internal bridge}.
The following proposition just recapitulates the fact such pointed objects may be decomposed in terms of the class of one-hole contexts/doubly-rooted maps $\contextClass$.

\begin{proposition}
For the combinatorial classes $\Theta_{sub}\closedTerms$, $\closedTerms$, and $\contextClass$ we have
\begin{equation}\label{eq:ThetaSub}
  \Theta_{sub}\closedTerms = \contextClass~ \closedTerms.
\end{equation}
At the level of generating functions, if $\contextOGF(z,v)$ is the generating function enumerating objects of $\contextClass$, with $v$ tagging internal bridges/closed proper subterms, then
\begin{equation}\label{eq:ThetaSubOGF}
v \frac{\partial}{\partial v} \lamClosedSubOGF(z,v) = \contextOGF(z,v)~ \lamClosedSubOGF(z,v).
\end{equation}
\end{proposition}
\begin{proof}
Let $t \in \Theta_{sub}\closedTerms$ be a closed linear $\lambda$-term with a distinguished closed proper subterm $u$. By definition, this means that $t = c[u]$ for some non-trivial one-hole context $c \in \contextClass$ and $u \in \closedTerms$, establishing \Cref{eq:ThetaSub}.

In terms of maps, if $m = \tau(t)$ is the map corresponding to $t$ with $b = vw$ a distinguished internal bridge, where without loss of generality we assume $v$ is an ancestor of $w$ in the $t$-tree of $m$, then $C$ is the component of $m \setminus b$ containing the root, with a new box vertex added and made adjacent to $v$, while $u$ is the remaining component with a new root vertex added and made adjacent to $w$.
\end{proof}

We now proceed to show that elements of $\contextClass$ factor into a non-empty sequence of elements in $\qClass$.
\begin{lemma} 
For the combinatorial classes $\contextClass$ and $\qClass$ we have
\begin{equation}\label{eq:QdecompW}
  \contextClass = SEQ_{\geq 1}(\mathcal{Q}) 
\end{equation}
At the level of generating functions, if $\qOGF(z,v)$ is the generating function enumerating objects of $\qClass$, with $v$ tagging internal bridges/closed proper subterms, then
\begin{equation}\label{eq:QdecompWOGF}
\contextOGF(z,v) = \frac{v\qOGF(z,v)}{1 - v\qOGF(z,v)}
\end{equation}
\end{lemma}
\begin{proof}

Let $m \in \contextOGF$ be a doubly-rooted trivalent map corresponding to a one-hole context $c$, with $b_s = uv$ the unique external bridge of $m$ adjacent to its box vertex $v$. Let, also, $r$ be the root of $m$ and $p$ be the $c$-path between the unique neighbour of $r$ and the box vertex $v$.

Let us label the, say, $i$ bridges belonging to $p$ as $b_1, b_2, \dots, b_{i}$, ordered by their proximity to the root (so that $b_{i} = b_s$). We note that $i \geq 1$ by definition of $p$. Let $C_1, C_2, \dots, C_{i+1}$ be the $i+1$ connected components of $m' = m \setminus b_1 \setminus  \dots  \setminus b_i \setminus r$, ordered in a way such that $C_i$ is the unique component of $m'$ containing the endpoint of $b_i$ closest to the root.

Consider now a connected component $C_k,  k \in \{1, \dots, i \}$, and let $p_k$ be the restriction of $p$ in $C_k$, i.e., $p_k = p[V(p) \cap V(C_k)]$. By construction, there exist exactly two degree-2 vertices in $C_k$, say $s_k$ and $t_k$, with $s_k$ being the one closer to the root in $m$, in terms of distances along the $c$-tree of $m$. Modifying each $C_k$ by adding a new root vertex $r_k$ which we make adjacent to $s_k$ as well as a new box vertex $v_{k}$ which we make adjacent to $t_k$ we produce a map $C'_{k}$ satisfying the restriction of maps in $\qClass$: the unique path between $r_k$ and $v_k$ along the $\tau(C'_{k})$-spanning tree of $C'_{k}$ is exactly $p_k$ together with the edges $r_k s_k$ and $t_k v_k$, which by construction contains no internal bridges of $C_k$. Finally, we note that $C_{i+1} = v$ is just the box vertex of $m$ which we are free to discard. 

In terms of contexts, the above decomposition translates uniquely to a factorisation $c = c_1 \circ \dots \circ c_i$ where $c_{k} = \tau(C'_{k})$ for $1 \leq k \leq i$.

The above arguments provide a mapping from structures in $\contextClass$ to a unique non-empty sequence of elements $ C'_k \in \qClass$. With respect to enumeration we note again that each $C'_{k}$, $1 \leq k \leq i$, can be paired with a unique bridge of $m$, the one incident to the vertex of $C_{k}$ closest to the root, and so we introduce $v$-factors in the generating function $\frac{v\qOGF(z,v)}{1-v\qOGF(z,v)}$ to keep track of that data. Furthermore, the decomposition of $m$ described above is invertible: given the sequence of maps $(C'_1, \dots, C'_{i}) \in \qClass^{i}$ we may uniquely reconstruct $m$ by deleting the box vertex of each $C'_k \in \{C'_1, \dots, C'_{i-1}\}$ and identifying the resulting unique degree-2 vertex of each $C'_k$, for $k < i$, with the root of $C_{k+1}$. This new map, rooted at the root of $C_1$, is an element of $\contextClass$. Furthermore, all bridges/closed subterms are accounted for and so we have the desired equality between the corresponding generating functions.

In terms of contexts this corresponds to the fact that from the sequence of contexts $(c_1, \dots, c_{i}) \in \qClass^{i}$, we may uniquely reconstruct the element $c \in \contextClass$ as $c = c_1 \circ \dots \circ c_{i}$. %
\end{proof}

From the above bijections we finally obtain
\begin{lemma}
Let $\lamClosedSubOGF$ be the generating function enumerating closed linear lambda terms where $v$ tags closed proper subterms. We have that
\begin{equation}\label{eq:diffvW}
\frac{\partial}{\partial v} \lamClosedSubOGF(z,v) = -\frac{v^{2} z \lamClosedSubOGF(z,v)^{3} + z^{2} \lamClosedSubOGF(z,v) - \lamClosedSubOGF(z,v)^{2}}{{\left(v^{3} - v^{2}\right)} z \lamClosedSubOGF(z,v)^{2} + v z^{2} - {\left(v - 1\right)} \lamClosedSubOGF(z,v)} .
\end{equation}
\end{lemma}
\begin{proof}
By \Cref{eq:ThetaSubOGF,eq:QdecompWOGF} we have

\begin{equation}\label{eq:TpointedAtBridges}
v \frac{\partial}{\partial v} \lamClosedSubOGF(z,v) = \contextOGF(z,v)~ \lamClosedSubOGF(z,v) = \frac{v\qOGF(z,v)}{1 - v\qOGF(z,v)} \lamClosedSubOGF(z,v)
\end{equation}
The result then follows by substituting $\qOGF = \frac{\abstractionSubOGF}{z^2 + \abstractionSubOGF}$ (which follows from \Cref{eq:ldecompqOGF}) and  $\abstractionSubOGF = \lamClosedSubOGF - zv^2(\lamClosedSubOGF)^2 - z^2$ (which follows from the definition of $\closedAbsClass$) into \Cref{eq:TpointedAtBridges} and finally dividing both sides by $v$.
\end{proof}

Before we proceed with the main result of this section, we present a number of definitions and lemmas useful for the proof.

\begin{figure}[h]
  \centering
  \includegraphics[scale=0.7]{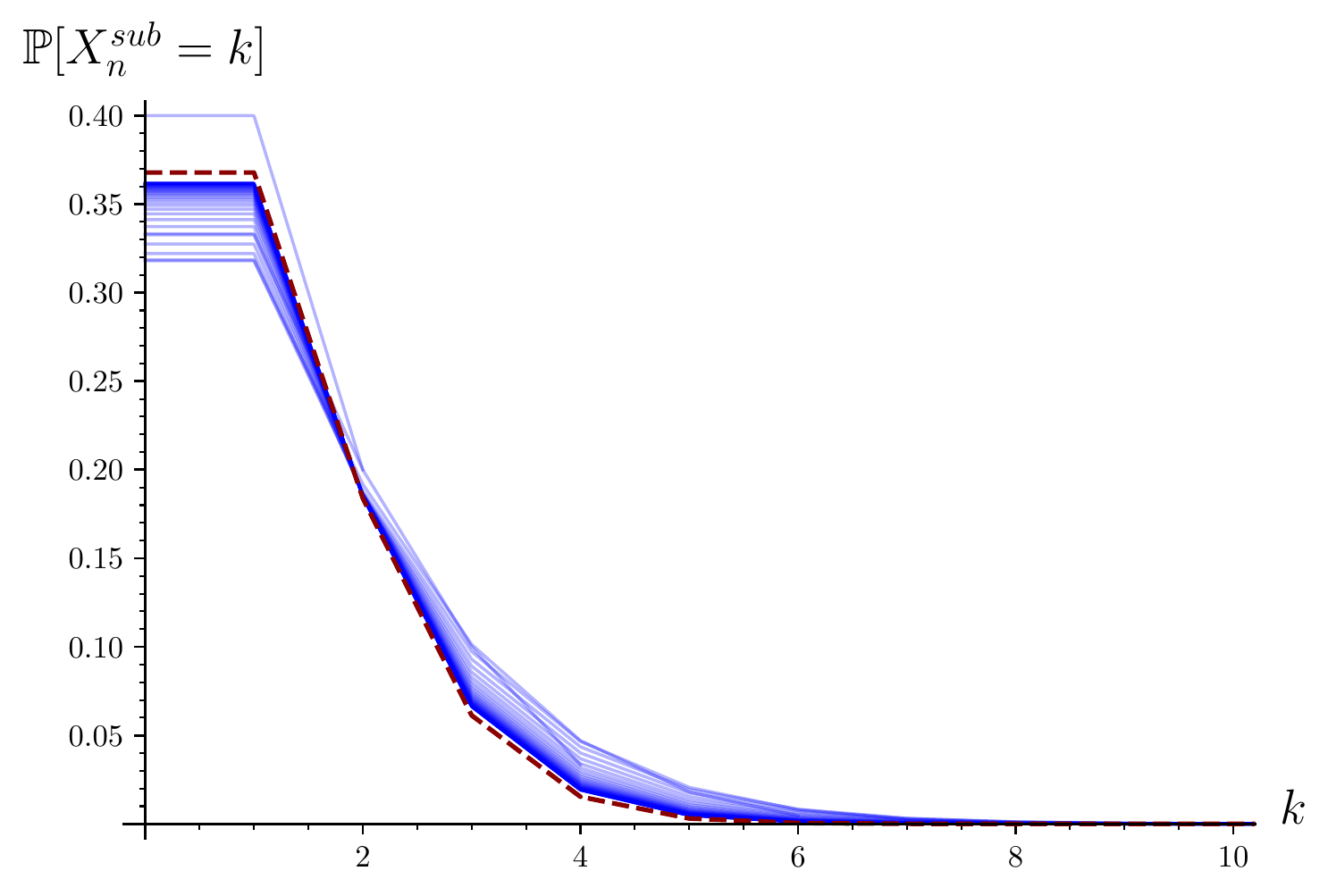}
  \caption{Density plots of $X^{sub}_n$ for $n = 2 \dots 100$ along with $Poisson(1)$ in red.}
  \label{fig:distPlotId}
\end{figure}

\begin{definition}[{Operators $\balancedPart[k]$}]
We define a family of operators $\balancedPart[k]$, parameterised by $k \in {-1} \cup \mathbb{N}$, which extract the {\em balanced part} of a polynomial $\eta \in \mathbb{Z}[f,z,v]$, defined as follows
\begin{equation}\label{eq:Bkdef}
\balancedPart[k] \left(\eta\right) = \sum_{\overset{i,j}{i \leq k+1}} \overline{\eta}_{j,i} v^j z^{2k - 2(i-1)} f^i.
\end{equation} 
where $\overline{\eta}_{j,i} = [v^j f^i z^{2k - 2(i-1)}] \eta$.
\end{definition}

\begin{definition}[$k$-admissible polynomial]
A polynomial $\eta \in \mathbb{Z}[f,z,v]$ is $k$-admissible if it can be written as a sum of monomials
\begin{equation}
\sum_{\overset{i,j,k}{l \geq \max\left\{0, 2k - 2(i-1) \right\}}} a_{j,i,l} f^i v^j z^l
\end{equation}
\end{definition}

\begin{lemma}\label{lemma:Blinearity}
The operator $\balancedPart[k]$ is linear, that is, if $\eta = \eta_1 + \eta_2$, 
\begin{equation}
\balancedPart[k](\eta) = \balancedPart[k](\eta_1) + \balancedPart[k](\eta_2)
\end{equation}
\end{lemma}
\begin{proof}
Follows immediately from \Cref{eq:Bkdef}.
\end{proof}

\begin{lemma}\label{lemma:Bmultiplicativity}
Let $\eta = \eta_1 \cdot \eta_2$, where $\eta_1$ is $k_1$-admissible and $\eta_2$ is $k_2$-admissible. Then $\eta$ is $(k_1 + k_2 + 1)$-admissible and
\begin{equation}\label{eq:balancedPartOfProduct}
\balancedPart[k_1 + k_2 + 1](\eta) = \balancedPart[k_1](\eta_1) \cdot \balancedPart[k_2](\eta_2)
\end{equation}
\end{lemma}
\begin{proof}
The product of two monomials $m_{j,i} = a_{j,i,l} z^l f^i v^j$ taken from $\eta_1$ and $m_{i',j'} = b_{j,i,l} z^{l'} f^{i'} v^{j'}$ taken from $\eta_2$ yields a monomial in $\balancedPart[k_1 + k_2 + 1](\eta)$ if and only if $l = 2k_1 - 2(i-1)$ and $l' = 2k_2 - 2(i'-1)$, i.e only if both monomials belong to $\balancedPart[k_1](\eta_1)$ and $\balancedPart[k_2](\eta_2)$ respectively. Othewise, if either $l > 2k_1 - 2(i-1)$ or $l' > 2k_2 - 2(i'-1)$, the product will result in a monomial of degree in $z$ at least $2(k_1 + k_2) - 2((i+i') - 1) + 1$. Since $\eta_1, \eta_2$ are $k_1$- and $k_2$-admissible respectively, these are the only cases of monomials possible therefore $\eta$ is $(k_1 + k_2 + 1)$ admissible and its balanced part is given by \Cref{eq:balancedPartOfProduct}.
\end{proof}

\begin{lemma}\label{lemma:BdiffByf}
Let $\eta$ be $k$-admissible for $k \geq 0$. Then $\partial_{f} \eta$ is $(k-1)$-admissible and
\begin{equation}
\balancedPart[k-1](\partial_f \eta) = \sum_{i,j} i \overline{\eta}_{i,j} z^{2k - 2(i-1)} f^{i} v^{j}
\end{equation}
\end{lemma}
\begin{proof}
The $k$-admissibility of $\eta$ implies that $\eta$ can be written as a sum of monomials of the form 
\begin{equation}
\balancedPart[k]\left( \eta \right) + R = \sum_{i,j} \overline{\eta}_{i,j} z^{2k - 2(i-1)} f^{i} v^{j} + R
\end{equation} 
where $R$ is a sum of monomials of the form $\underline{\eta}_{i,j} z^p f^i v^j$ for constants $\underline{\eta}_{i,j}$ and $p < 2k - 2(i-1)$. Then, any monomial $(i+1) \overline{\eta}_{i,j} z^l f^{i} v^{j}$ in $\partial_f \balancedPart[k](\eta)$ satisfies the conditions of $(k-1)$-admissibility, while any monomial $i \underline{\eta}_{i,j} z^{p} f^{i-1} v^j$ in $\partial_f R$ does not contribute to $\balancedPart[k-1](\eta)$ and, since $\eta$ is $k$-admissible, has a degree $p > 2k - 2(i-1) = 2(k-1) - 2i > 2(k-1) - 2(i-1)$ which also satisfies the conditions of $(k-1)$-admissibility.
\end{proof}

\begin{theorem}\label{lemma:poissonBridges}
Let $\chi_{sub}$ be the parameter corresponding to the number of closed proper subterms in a closed linear $\lambda$-term or, equivalently, of internal bridges in rooted trivalent maps. Then for the random variables $X^{sub}_{n}$ corresponding to $\chi_{sub}$ taken over $\closedTerms[n]$ respectively we have:

\begin{align}
X^{sub}_{n} &\overset{D}{\rightarrow} Poisson(1),
\end{align}
\end{theorem}

\begin{proof}

Let $W_N(z,v,f)$  be such that $W_N(z,v,\lamClosedSubOGF(z,v)) = \frac{\partial^N}{\partial v^N} \lamClosedSubOGF(z,v)$.
By successively differentiating \Cref{eq:diffvW} we have that $W_N$ is a rational function
\begin{equation}\label{eq:hNovergk}
W_N = \frac{h_N}{g^{k}}
\end{equation}
where $k = 2N - 1$, $h_N \in \mathbb{Z}[f,z,v]$, and

\begin{equation}
g = f^2v^3z - f^2v^2z + vz^2 - fv + f.
\end{equation}
By definition, $W_N\lvert_{f = \lamClosedSubOGF(z,v)}$ equals the $N$-th derivative of $\lamClosedSubOGF(z,v)$ with respect to $v$, which counts rooted trivalent maps with the $v$-mark erased from $N$ of its bridges. This implies that $W_N\lvert_{f = \lamClosedSubOGF(z,1), v=1}$ counts properly-sized objects: the coefficients $[z^n] W_N(z,v,\lamClosedSubOGF(z,v))$ will be 0 for $n \not\equiv 2 \pmod{3}$ and $n < 2$. Therefore, $[f^i] h_N$ has minimum degree in $z$ at least $2k - 2(i-1)$, otherwise 
\[
[z^p] \left. \frac{[f^i] h_N }{g^k} \right\rvert_{f = \lamClosedSubOGF(z,1), v=1} = [z^p] \frac{\left. [f^i] h_N \right\rvert_{f = \lamClosedSubOGF(z,1), v=1}}{z^{2k}}
\]
would be non-zero for some $p < 2$ which we know not to be the case. Therefore, using the $\balancedPart[k]$ operator, we may expand $h_N$ as the following sum of its {\em balanced} and {\em unbalanced} parts
\begin{equation}\label{eq:balancedAndRestHn}
h_N = \balancedPart[k](h_N) + R = \sum_{j} v^j \sum_{i} \alpha_{N,j,i} z^{2k - 2(i-1)} f^i+ R_j
\end{equation}
where $R = h_N - \balancedPart[k](h_N)$, $R_j = [v^j] R$, and $\alpha_{N,j,i} = [v^j f^i z^{2k-2(i-1)}] \balancedPart[k](h_N)$. The above argument then implies that $R$ is a sum of monomials of the form $\beta_{N,j,i} v^j f^i {z^l}$ where $i \geq 3, l > 2k - 2(i-1)$, and $\beta_{N,j,i} \in \mathbb{Z}$ and therefore $h_N$ is $k$-admissible.

Let us now introduce the following two operators obtained by evaluating $\partial_f \left( \balancedPart[k] (\cdot) \right)$ and $\balancedPart[k](\cdot) $ at $f = z = v = 1$:
\begin{align*}
\alphasWSummed[k](\eta) &=  \left.\left(\partial_f \balancedPart[k](\eta)\right) \right\rvert_{f=1, z=1, v=1} = \sum_{i,j} i \overline{\eta}_{N,j,i}, \\
\alphasSummed[k](\eta) &=  \left.\left(\balancedPart[k](\eta)\right) \right\rvert_{f=1, z=1, v=1} = \sum_{i,j} \overline{\eta}_{N,j,i}.
\end{align*} 

Notice then that, by \Cref{lemma:tupletsAsymptShifted}, a term of $W_N$ whose numerator is given by $\alpha_{N,j,i} z^{2k - 2(i-1)} f^i$ is such that
\begin{equation*}
\left. [z^n] \frac{\alpha_{N,j,i} z^{2k - 2(i-1)} f^i}{g^k} \right\lvert_{v=1, f = \lamClosedOGF(z)} = [z^n] \frac{\alpha_{N,j,i} \lamClosedOGF(z)^i}{z^{2(i-1)}} \sim  i ~ \alpha_{N,j,i} ~ [z^n] \lamClosedOGF(z) ,\qquad n \equiv 2\pmod{3}
\end{equation*}
and so we have 
\begin{equation}\label{eq:BalancedPartContribs}
[z^n] \left(  \left. \frac{\balancedPart[k](h_N)}{g^{k}} \right\lvert_{v=1, f = \lamClosedOGF(z)} \right) \sim\sum_{i,j} i \alpha_{N,j,i} [z^n] \lamClosedOGF(z) = \alphasWSummed[k](h_N) [z^n] \lamClosedOGF(z)
\end{equation}
while the coefficients of $z^n$ of monomials in the unbalanced part $R$ evaluated at $v=1, f = \lamClosedOGF(z)$ are asymptotically $O\left( [z^{n-3}] \lamClosedOGF(z) \right)$ for $n \equiv 2\pmod{3}$ and since there's only finitely many of them, we have 
\begin{equation}\label{eq:restPartDoesntContrib}
[z^n] \left(  \left. \frac{R}{g^{k}} \right\lvert_{v=1, f = \lamClosedOGF(z)} \right) = O\left( [z^{n-3}] \lamClosedOGF(z) \right)
\end{equation}

We will now proceed via induction to show that for any $N \geq 1$ the following hold:

\begin{align}
  \alphasWSummed[k]\left([v^0] h_N\right) = \sum_{i} i \alpha_{N,0,i} & = 1 \label{eq:v0Contributes1}, \\
  \alphasSummed[k]\left([v^0] h_N \right) = \sum_{i} \alpha_{N,0,i} & = 0  \label{eq:v0CoeffsSum0},
\end{align}
and for $j \geq 1$
\begin{align}
  \alphasWSummed[k]\left([v^j] h_N\right) = \sum_{i} i \alpha_{N,j,i} & = 0, \label{eq:vHigherContributes0}\\
  \alphasSummed[k]\left([v^j] h_N\right) = \sum_{i} \alpha_{N,j,i} &= 0. \label{eq:vHigherCoeffsSum0}
\end{align}
For the inductive base, notice that all four equations hold for
\[
  W_{1} = -\frac{f^3 v^2 z}{g} - \frac{f z^2}{g} + \frac{f^2}{g},
\]
which has balanced part $\balancedPart[1](h_{1}) = - f z^2 + f^2$.
For our inductive step, supposing that the desired properties hold for $W_N$, we begin by noting that, by the chain rule, $W_{N+1}$ may be written as
\[
W_{N+1} = \partial_f W_N W_1 + \partial_v W_N = \frac{\left(\partial_f h_N \cdot g^{k} - h_N \cdot \partial_f g^{k}\right) h_1}{g^{2k+1}} + \frac{\left(\partial_v h_N \cdot g^{k} - h_N \cdot \partial_v g^{k}\right) g}{g^{2k+1}}.
\]
By grouping together the summands and partially simplifying them, we obtain
\[
\begin{split}
W_{N+1} &=
\frac{\partial_f h_N \cdot g h_1 -k h_N \partial_f g \cdot h_1 + \partial_v h_N \cdot g^2 - k h_N \partial_v g \cdot g}{g^{k+2}} \label{eq:WnSummands} 
\end{split}
\]
which brings $W_{N+1}$ into the form of \Cref{eq:hNovergk} with
\begin{equation}\label{eq:WnBalancedSummands}
  h_{N+1} = \partial_f h_N \cdot g h_1 -k h_N \partial_f g \cdot h_1 + \partial_v h_N \cdot g^2 - k h_N \partial_v g \cdot g.
\end{equation}
Then, \Cref{lemma:Blinearity} allows us to compute the balanced part of $h_{N+1}$ as
\begin{equation}\label{eq:WnBalancedSummands}
\balancedPart[k+2](h_{N+1}) =  \balancedPart[k+2]\left(\partial_f h_N \cdot g \cdot h_1\right) - \balancedPart[k+2]\left(-k \cdot h_N \cdot \partial_f g \cdot h_1 \right) + \balancedPart[k+2]\left(\partial_v h_N \cdot g^2\right) - \balancedPart[k+2]\left(k \cdot h_N \cdot \partial_v g \cdot g \right)
\end{equation}

Therefore, to validate \Cref{eq:v0Contributes1,eq:v0CoeffsSum0,eq:vHigherContributes0,eq:vHigherCoeffsSum0}, it suffices to sum the contributions of each summand of \Cref{eq:WnBalancedSummands} to \Cref{eq:v0Contributes1,eq:v0CoeffsSum0,eq:vHigherContributes0,eq:vHigherCoeffsSum0}, which we now proceed to do.
\begin{enumerate}
\item For the first summand of \Cref{eq:WnBalancedSummands} we have
\begin{align*}
\balancedPart[k+2]\left({\partial_f h_N \cdot g \cdot h_1}\right) &= \balancedPart[k+1]\left( \partial_f h_N  \cdot h_1\right) \cdot \balancedPart[0]\left( g \right) & \text{by \Cref{lemma:Bmultiplicativity}} \\ 
&= \balancedPart[k-1]\left( \partial_f h_N \right) \cdot \balancedPart[1]\left( h_1 \right) \cdot \balancedPart[0](g) & \text{by \Cref{lemma:Bmultiplicativity}} 
\end{align*}  
By \Cref{lemma:BdiffByf} we have 
\begin{equation}
\begin{split}\label{eq:BpartialFhNTimesH1}
\balancedPart[k-1](\partial_f h_N) \cdot \balancedPart[1](h_1) &= \left(\sum_{i} i~\alpha_{N,j,i}~v^jz^{2k - 2(i-1)}f^{i-1}\right) \left( -z^2f + f^2 \right) \\
&= \sum_{i,j} ((i-1)~\alpha_{N,j,i-1} - i~\alpha_{N,j,i}) v^j z^{2(k+1)-2(i-1)}f^{i}.
\end{split}
\end{equation}
and so, by letting $a'_{i,j} = ((i-1)~\alpha_{N,j,i-1} - i~\alpha_{N,j,i})$, 
\begin{align*}
  \balancedPart[k+2](\partial_f h_N\cdot g \cdot h_1) &= \balancedPart[k-1]\left( \partial_f h_N \right) \cdot \balancedPart[1]\left( h_1 \right) \cdot \balancedPart[0](g) \\
                                                      &= \left( \sum_{i,j} a'_{j,i} v^j z^{2(k+1)-2(i-1)}f^{i} \right) \left(vz^2 - fv + f \right) \\
&= \sum_{i,j} (\alpha'_{j-1,i} - \alpha'_{j-1,i-1} + \alpha'_{j,i-1}) z^{2(k+2) - 2(i-1)} v^j f^i 
\end{align*}

Therefore, for any $j \geq 0$,

\begin{align}
\begin{split} \label{eq:firstSummandAlphaWSummed}
\alphasWSummed[k+2]\left([v^j] \left(\partial_f h_N \cdot h_1 \cdot g\right) \right)
&= \sum_{i} i \alpha'_{j-1,i} - \sum_{i} i \alpha'_{j-1,i-1} + \sum_{i} i \alpha'_{j,i-1} \\
&= \sum_{i} i \alpha'_{j-1,i} - \sum_{i} (i+1) \alpha'_{j-1,i} + \sum_{i} (i+1) \alpha'_{j,i} \\
&= \alphasWSummed[k+1]\left( [v^{j-1}] \left( \partial_f h_N \cdot h_1 \right) \right) \\
&- \alphasWSummed[k+1]\left( [v^{j-1}] \left( \partial_f h_N \cdot h_1 \right) \right) + \alphasSummed[k+1]\left([v^{j-1}] \left( \partial_f h_N \cdot h_1 \right) \right) \\
&+ \alphasWSummed[k+1]\left( [v^{j}] \left( \partial_f h_N \cdot h_1 \right) \right) + \alphasSummed[k+1]\left([v^{j}] \left( \partial_f h_N \cdot h_1 \right) \right) 
\end{split} \\
\begin{split} \label{eq:firstSummandAlphaSummed}
\alphasSummed[k+2]\left([v^j] \left(\partial_f h_N \cdot h_1 \cdot g\right) \right)
&= \sum_{i} \alpha'_{j-1,i} - \sum_{i} \alpha'_{j-1,i-1} + \sum_{i} \alpha'_{j,i-1} \\
&= \alphasSummed[k+1]\left( [v^{j}] \left( \partial_f h_N \cdot h_1 \right) \right)
\end{split}
\end{align}

By \Cref{eq:BpartialFhNTimesH1} we have
\begin{align}
\begin{split}\label{eq:BdiffFhNTimesH1alphaWSummed}
\alphasWSummed[k+1]\left([v^j] \left(\partial_f h_N \cdot h_1 \right)\right) &= \sum_{i} i((i-1)~\alpha_{N,j,i-1} - i~\alpha_{N,j,i})\\
&= -\alpha_{N,j,1} + (2\alpha_{N,j,1} - 4\alpha_{N,j,2}) + (6\alpha_{N,j,2} - 9\alpha_{N,j,3}) + \dots \\
&= \alpha_{N,j,1} + 2\alpha_{N,j,2} + 3\alpha_{N,j,3} + \dots = \sum_{i} i a_{N,j,i}
\end{split}\\
\begin{split}\label{eq:BdiffFhNTimesH1alphaSummed}
\alphasSummed[k+1]\left([v^j] \left(\partial_f h_N \cdot h_1 \right)\right) &= \sum_{i} ((i-1)~\alpha_{N,j,i-1} - i~a_{N,j,i}) \\
&= -\alpha_{N,j,1} + (\alpha_{N,j,1} - 2\alpha_{N,j,2}) + (2\alpha_{N,j,2} - 3\alpha_{N,j,3}) + \dots \\
&= 0
\end{split}
\end{align}

Since $[v^j] h_N = 0$ for $j < 0$, evaluating \Cref{eq:firstSummandAlphaWSummed} at $j=0$ and applying \Cref{eq:BdiffFhNTimesH1alphaWSummed,eq:BdiffFhNTimesH1alphaSummed} yields 
\begin{equation*}
\alphasWSummed[k+1]\left( [v^{0}] \left( \partial_f h_N \cdot h_1 \right) \right) + \alphasSummed[k+1]\left([v^{0}] \left( \partial_f h_N \cdot h_1 \right) \right) = \sum_{i} i a_{N,0,i} = 1
\end{equation*}
where the last step follows by applying \Cref{eq:v0Contributes1} inductively.
Similarly, evaluating \Cref{eq:firstSummandAlphaWSummed} at $j \geq 1$ yields 
\begin{equation*}
\alphasSummed[k+1]\left([v^{j-1}] \left( \partial_f h_N \cdot h_1 \right) \right) + \alphasWSummed[k+1]\left( [v^{j}] \left( \partial_f h_N \cdot h_1 \right) \right) + \alphasSummed[k+1]\left([v^{j}] \left( \partial_f h_N \cdot h_1 \right) \right) = 0
\end{equation*}
by applying \Cref{eq:vHigherContributes0} inductively.
Finally, \Cref{eq:firstSummandAlphaSummed} yields $0$ for any $j \geq 0$, due to \Cref{eq:BdiffFhNTimesH1alphaSummed}. %
Therefore the contribution of the first summand to \Cref{eq:v0Contributes1} is 1, while for \Cref{eq:vHigherContributes0,eq:v0CoeffsSum0,eq:vHigherCoeffsSum0} we have a contribution of 0.

\item Moving on to the next summand of \Cref{eq:WnBalancedSummands}, we have
\begin{align*}%
\balancedPart[k+2] \left( -k \cdot h_N \cdot \partial_f g \cdot h_1 \right) &= \balancedPart[k]\left( -k \cdot h_N \cdot \partial_f g \right)  \cdot \balancedPart[1](h_1) & \text{by \Cref{lemma:Bmultiplicativity}} \\ 
&= -k \cdot \balancedPart[k]\left(h_N\right) \cdot \balancedPart[-1]\left( \partial_f g \right)  \cdot \balancedPart[1](h_1) & \text{by \Cref{lemma:Bmultiplicativity}}
\end{align*}
From this we obtain
\begin{equation}
\begin{split}\label{eq:BsecondSummand}
\balancedPart[k+2] \left( -k \cdot h_N \cdot \partial_f g \cdot h_1 \right)  &= -k \cdot \balancedPart[k]\left(h_N\right) \cdot \balancedPart[-1]\left( \partial_f g \right)  \cdot \balancedPart[1](h_1) \\
&= -k \cdot \left(\sum_{i} \alpha_{N,j,i} v^j f^i z^{2k-2(i-1)}\right) \cdot (-v+1)\cdot (-fz^2 + f^2)   \\
&= -k \sum_{i}  (-(\alpha_{N,j,i-1} - \alpha_{N,j-1,i-1}) + (\alpha_{N,j,i-2} - \alpha_{N,j-1,i-2})) v^j f^i z^{2(k+2)-2(i-1)}
\end{split}
\end{equation}

As argued before, $a_{N,i,j} = 0$ for $j \leq 0$, so from \Cref{eq:BsecondSummand} we then obtain
\begin{align}\label{eq:secondSummandAlphaWSummed}
\begin{split}
\alphasWSummed[k+2]\left([v^j]\left( -k \cdot h_N \cdot \partial_f g \cdot h_1 \right) \right) &=  -k \sum_{i} i(-(\alpha_{N,j,i-1} - \alpha_{N,j-1,i-1}) + (\alpha_{N,j,i-2} - \alpha_{N,j-1,i-2})) \\
&= -k \left( 2\left( -\alpha_{N,j,1} + \alpha_{N,j-1,1} \right) + 3\left(-a_{N,j,2} + \alpha_{N,j-1,2} + \alpha_{N,j,1} -  \alpha_{N,j-1,1}\right) + \dots \right) \\
&= -k \left( (a_{N,j,1} - \alpha_{N,j-1,1}) + (\alpha_{N,j,2} - \alpha_{N,j-1,2}) + \dots \right) \\
&= -k \left( \sum_{i} a_{N,j,1} - \sum_{i} a_{N,j-1,1} \right)
\end{split}
\end{align}
and 
\begin{align}
\begin{split}\label{eq:secondSummandAlphaSummed}
\alphasSummed[k+2]\left([v^j]\left( -k \cdot h_N \cdot \partial_f g \cdot h_1 \right) \right) &= \sum_{i} (-(\alpha_{N,j,i-1} - \alpha_{N,j-1,i-1}) + (\alpha_{N,j,i-2} - \alpha_{N,j-1,i-2})) \\
&= \left( -\alpha_{N,j,1} + \alpha_{N,j-1,1} \right) +  \left(-\alpha_{N,j,2} + \alpha_{N,j-1,2} + \alpha_{N,j,1} -  \alpha_{N,j-1,1}\right) + \dots \\
&= 0
\end{split}
\end{align}

For $j=0$, \Cref{eq:secondSummandAlphaWSummed} becomes  $-k \sum_{i} a_{N,0,i}$ which yields 0 by induction and \Cref{eq:v0CoeffsSum0}, while for $j \geq 1$, \Cref{eq:secondSummandAlphaWSummed} it also yields 0 due to \Cref{eq:v0CoeffsSum0,eq:vHigherCoeffsSum0}. Therefore the contribution of the second summand to \Cref{eq:v0Contributes1,eq:vHigherContributes0} is 0. Finally, its contribution to \Cref{eq:v0CoeffsSum0,eq:vHigherCoeffsSum0} is obtained by evaluating \Cref{eq:secondSummandAlphaSummed} at any $j \geq 0$, which is also 0.
\hspace{1cm}

\item For the third summand of \Cref{eq:WnBalancedSummands} we have
\begin{align*}%
\balancedPart[k+2]\left(\partial_v h_N \cdot g^2\right) &= \balancedPart[k](\partial_v h_N) \cdot \balancedPart[1](g^2) 
\end{align*}
by \Cref{lemma:Bmultiplicativity}.
It is straightforward to notice that differentiation by $v$ preserves the balanced part of $h_N$ and only affects by a change of coefficients from $\alpha_{N,j,i}$ to $i \alpha_{N,j-1,1}$ so that %
\begin{equation}
\begin{split}
\balancedPart[k+2]\left(\partial_v h_N \cdot g^2\right) &= \balancedPart[k](\partial_v h_N) \cdot \balancedPart[1](g^2) \\
  &= \left(\sum_{i,j} j~\alpha_{N,j,i}~v^{j-1}z^{2k - 2(i-1)}f^{i}\right) \left(v^2z^4 - 2fv^2z^2 + f^2v^2 + 2fvz^2 - 2f^2v + f^2\right)\\
&= \sum_{i} \left((j - 1)\alpha_{N, j - 1, i} - 2(j - 1)\alpha_{N, j - 1, i - 1} + (j - 1)\alpha_{N, j - 1, i - 2} \right. \\
&\qquad + \left. 2j\alpha_{N, j, i - 1} - 2j\alpha_{N, j, i - 2} + (j + 1)\alpha_{N, j + 1, i - 2} \right)~v^{j}z^{2(k+2) - 2(i-1)}f^{i} 
\end{split} 
\end{equation}

From the above we obtain
\begin{equation}
\begin{split}\label{eq:thirdSummandAlphasWSummed}
\alphasWSummed[k+2]\left( [v^j] \left(\partial_v h_N \cdot g^2\right)\right)  &= \sum_{i} i (j - 1)\alpha_{N, j - 1, i} - \sum_{i} i 2(j - 1)\alpha_{N, j - 1, i - 1} + \sum_{i} i (j - 1)\alpha_{N, j - 1, i - 2}  \\
&+ \sum_{i} i 2j\alpha_{N, j, i - 1} - \sum_{i} i 2j\alpha_{N, j, i - 2} + \sum_{i} i  (j + 1)\alpha_{N, j + 1, i - 2}
\end{split}
\end{equation}
and
\begin{equation}
\begin{split}\label{eq:thirdSummandAlphasSummed}
\alphasSummed[k+2]\left( [v^j] \left(\partial_v h_N \cdot g^2\right)\right)  &= \sum_{i} (j - 1)\alpha_{N, j - 1, i} - \sum_{i} 2(j - 1)\alpha_{N, j - 1, i - 1} + \sum_{i} (j - 1)\alpha_{N, j - 1, i - 2}  \\
&+ \sum_{i} 2j\alpha_{N, j, i - 1} - \sum_{i} 2j\alpha_{N, j, i - 2} + \sum_{i} (j + 1)\alpha_{N, j + 1, i - 2}
\end{split}
\end{equation}

For $j = 0$, \Cref{eq:thirdSummandAlphasWSummed} reduces to $\sum_{i} i \alpha_{N,1,i-2}$ which is zero due to induction and \Cref{eq:vHigherContributes0}. Similarly, for $j \geq 1$, all summands of \Cref{eq:thirdSummandAlphasWSummed} are zero again due to \Cref{eq:vHigherContributes0}. By similar arguments, \Cref{eq:thirdSummandAlphasSummed} is zero for any $j \geq 0$ due to induction and \Cref{eq:vHigherCoeffsSum0}. Therefore the contributions of the third summand to each of \Cref{eq:v0Contributes1,eq:v0CoeffsSum0,eq:vHigherContributes0,eq:vHigherCoeffsSum0} are 0. 
\hspace{1cm}

\item Finally, for the fourth summand of \Cref{eq:WnBalancedSummands} we have
\begin{align*}
\balancedPart[k+2]\left(-k \cdot h_N \cdot \partial_v g \cdot g \right) &= -k \cdot \balancedPart[k+1]\left(h_N \cdot \partial_v g\right) \cdot \balancedPart[0](g) & \text{by \Cref{lemma:Bmultiplicativity}} \\
&= -k \cdot \balancedPart[k]\left(h_N\right) \cdot \balancedPart[0]\left(\partial_v g\right) \cdot \balancedPart[0](g) & \text{by \Cref{lemma:Bmultiplicativity}}
\end{align*}
From this we obtain
\begin{equation}
\begin{split}\label{eq:BFourthSummandPartial}
\balancedPart[k+2]\left(-k \cdot h_N \cdot \partial_v g \cdot g \right)
  &= -k \cdot \balancedPart[k]\left(h_N\right) \cdot \balancedPart[0]\left(\partial_v g\right) \cdot \balancedPart[0](g) \\
  &= -k \left(\sum_{i} \alpha_{N,j,i} v^i z^{2k-2(i-1)} f^i\right) \left(z^2 - f \right)  \\
&= k \sum_{i} \left(\alpha_{N,j,i-1} - \alpha_{N,j,i} \right)  v^i z^{2(k+1)-2(i-1)} f^i.
\end{split}  
\end{equation}
and so, by letting $a''_{i,j} = k (\alpha_{N,j,i-1} - \alpha_{N,j,i})$, we have
\begin{equation}
\begin{split}
\balancedPart[k+2]\left(-k \cdot h_N \cdot \partial_v g \cdot g \right) &= \left(vz^2 - fv + f \right) \left( \sum_{i,j} a''_{j,i} v^j z^{2(k+1)-2(i-1)}f^{i} \right) \\
&= \sum_{i,j} (\alpha''_{j-1,i} - \alpha''_{j-1,i-1} + \alpha''_{j,i-1}) z^{2(k+2) - 2(i-1)} v^j f^i 
\end{split}
\end{equation}

Summing the coefficients of \Cref{eq:BFourthSummandPartial} we obtain
\begin{align}
\begin{split}\label{eq:fourthSummandPartialAlphasWsummed}
\alphasWSummed[k+1]\left( [v^j] \left( -k \cdot h_N \cdot \partial_v g \right)\right) &= k\sum_{i} i(\alpha_{N,j,i-1} - \alpha_{N,j,i}) \\
&= k \left(-\alpha_{N,j,1} + 2(\alpha_{N,j,1} - \alpha_{N,j,2}) + 3(\alpha_{N,j,2} - \alpha_{N,j,3}) + \dots \right) \\
&= k \left(\alpha_{N,j,1} + \alpha_{N,j,2} + \alpha_{N,j,3} + \dots \right)
\end{split} \\
\begin{split}\label{eq:fourthSummandPartialAlphaSummed}
\alphasSummed[k+1]\left( [v^j] \left( -k \cdot h_N \cdot \partial_v g \right)\right) &= \sum_{i} (\alpha_{N,j,i-1} - \alpha_{N,j,i})\\
&= -\alpha_{N,j,1} + (\alpha_{N,j,1} - \alpha_{N,j,2}) + (\alpha_{N,j,2} - \alpha_{N,j,3}) + \dots \\
&= 0.
\end{split}
\end{align}

By induction, \Cref{eq:fourthSummandPartialAlphasWsummed} gives $0$ for any $j \geq 0$ due to \Cref{eq:v0CoeffsSum0,eq:vHigherCoeffsSum0}. Adapting the arguments used for the first summand of \Cref{eq:WnBalancedSummands} by swapping $a''$ for $a'$, we obtain, for all $j\geq0$
\begin{align*}
\alphasWSummed[k+2]\left( [v^j] \left( -k \cdot h_N \cdot \partial_v g  \cdot g \right)\right) &= 0, \\
\alphasSummed[k+2]\left( [v^j] \left( -k \cdot h_N \cdot \partial_v g  \cdot g \right)\right) &= 0
\end{align*}
Therefore the contributions of the fourth summand to each of \Cref{eq:v0Contributes1,eq:v0CoeffsSum0,eq:vHigherContributes0,eq:vHigherCoeffsSum0} are 0.
\end{enumerate}

Finally, we have, by \Cref{eq:BalancedPartContribs,eq:restPartDoesntContrib}, that $[z^n] W_{N} \sim [z^n]\frac{\balancedPart[k+2]\left(h_N\right)}{g^{k+2}}$. Since 
\begin{equation}
\balancedPart[k+2]\left(\partial_f W_N \cdot W_1\right) = \balancedPart[k+2]\left(\partial_f h_N \cdot g \cdot h_1\right) - \balancedPart[k+2]\left(-k \cdot h_N \cdot \partial_f g \cdot h_1 \right)
\end{equation}
and
\begin{equation}
\balancedPart[k+2]\left(\partial_v W_N\right) = \balancedPart[k+2]\left(\partial_v h_N \cdot g^2\right) - \balancedPart[k+2]\left(k \cdot h_N \cdot \partial_v g \cdot g \right),
\end{equation}

the computations of steps 1 to 4 above together with \Cref{eq:vHigherContributes0,eq:vHigherCoeffsSum0,eq:v0Contributes1,eq:v0CoeffsSum0} show that 
\begin{align*}
[z^n] \left. \left( \partial_f W_{N-1}W_1  \right) \right\lvert_{f = \lamClosedOGF(z), v=1} &\sim [z^n] W_{N-1}\lvert_{f = \lamClosedOGF(z), v=1} 
\sim [z^n] \lamClosedOGF(z), \\
[z^n] \left. \left( \partial v W_{N-1}\right)\right\lvert_{f = \lamClosedOGF(z), v=1} &= O\left([z^{n-3}] \lamClosedOGF(z) \right)
\end{align*}

Therefore \Cref{lemma:poissonSchema} applies for $n \equiv 2 \pmod{3}$, yielding our desired result. %

\end{proof} 

An application of the bijection shown in \Cref{fig:rootings} and of \Cref{lemma:rootingDoesntAffectDist} show that \Cref{lemma:poissonBridges} holds for unrooted trivalent maps too:

\begin{corollary}\label{cor:bridgestrivalent}
  Let $\chi_{b}$ be the parameter corresponding to the number of internal bridges in unrooted trivalent maps.
  Then we have $X^{b}_{n} \overset{D}{\rightarrow} Poisson(1)$ for the random variables $X^{b}_{n}$ counting $\chi_{b}$ over maps of size $n$.
\end{corollary}

Finally, by $\Cref{lemma:poissonBridges}$ we have that the probability of an object in $\closedTerms$ to be bridgeless is $1/e$ which together with $\Cref{thm:closedTermAsympts}$ yields the following asymptotic form for the number of bridedgeless rooted trivalent maps and closed linear $\lambda$-terms without closed proper subterms:

\begin{corollary}
The number of bridgeless rooted trivalent maps and closed linear $\lambda$-terms without closed proper subterms of size $p \equiv 2 \pmod{3}$ is
\begin{equation}
[z^p] \sim \frac{3}{e\pi} 6^n n! ,\qquad p = 3n + 2.
\end{equation}
\end{corollary}

%% file: sections/compositions.tex
It is combinatorial folklore that a large not-necessarily-connected trivalent map is almost surely connected. Such phenomena are abundant in the study of maps and graphs. As another example of this phenomenon, let us take $G(z)$ to be the exponential generating function for not-necessarily-connected labeled graphs and $C(z)$ the one enumerating connected labeled graphs. Then,

\begin{align}\label{eq:logExpConnected}
\begin{split}
  1 + G(z) &= e^{C(z)} \\
  C(z) &= \ln(G(z) + 1)
\end{split}
\end{align}
and a famous theorem by Bender, \cite[Theorem 1]{bender1975asymptotic}, may then be used to show that $[z^n] \ln(G(z) + 1) \sim [z^n] G(z)$ proving that asymptotically almost all labelled graphs of size $n$ are connected. A crucial element of this proof is the fact that the number of labelled graphs, $n! [z^n]G(z) = 2^{n \choose 2}$, grows much faster than $n!$.

More generally, this theorem by Bender shows that for appropriate formal power series $F(z,y)$ analytic at $0$ and $G(z)$ satisfying (among other requirements) $[z^{n-1}] G(z) = o\left([z^n] G(z)\right)$, one can deduce the coefficient asymptotics of the composition $F(z, G(z))$ from those of $G(z)$. The example given above corresponds to $F(z,y) = \ln(y + 1)$. Other examples of this can be given by taking instead $F$ and $G$ to be enumerating objects of some classes $\mathcal{F}$ and $\mathcal{G}$ whose number of structures grows rapidly with $n$. In such cases, the combinatorial intuition behind Bender's result is that, asymptotically, most of the structures in $\mathcal{F(G)}$ are constructed by taking a small $\mathcal{F}$-structure and replacing one of its atoms with the biggest appropriate $\mathcal{G}$-structure. The rapid growth conditions on the coefficients of $G(z)$ then precisely reflect the fact that there's many more ways to pick an element of $\mathcal{G}_n$ and compose it into a small $\mathcal{F}$-structure than there are ways to pick one in $\mathcal{G}_{n-1}$ and do so.

Given the above, one then expects that if $\chi$ is some parameter of a rapidly-growing class $\mathcal{C}$ of connected structures, then the parameter $\chi^*$ defined by summing $\chi$ over the connected components of not-necessarily-connected $\mathcal{G}$-structures behaves similarly. In the same vein, one would expect that for rapidly-growing combinatorial classes $\mathcal{F}, \mathcal{G}$, parameters $\chi$ defined over $\mathcal{G}$ and its natural extention over $\mathcal{F}(\mathcal{G})$ formed by summing $\chi$ over $\mathcal{G}$-substructures both behave in a similar way. Indeed, the result of the following subsection serves to formalise this intuition.

\subsection{Composition schema}

In this subsection we present a theorem inspired by Bender's theorem \cite[Theorem 1]{bender1975asymptotic} (as well as its extensions presented in \cite[Theorem 32]{borinsky2018generating} and \cite[Lemma B.8]{panafieu2019analytic}), which formalises the above discussion: under sufficient conditions, the limit law of a parameter marked by $u$ in a power series $G(z,u)$ remains unchaged when composing with some $F(z,u,y)$ provided that $F(z,1,y)$ is analytic at the origin and the coefficients of $G(z,1)$ grow rapidly enough.

Before proceeding with the main result of this subsection, we present a series of useful lemmas first.

\begin{lemma}\label{lemma:cauchyish}
Let $(a_n)_{n \in \mathbb{N}}$ be a positive increasing sequence such that $\frac{a_{n-1}}{a_{n}} = O\left(n^{-\sigma}\right)$ for $\sigma > 0$. Then for $n_0 \in \mathbb{N}$

\begin{equation}
\sum\limits_{k=n_0}^{n-n_0} a_{k} a_{n-k} = O(a_{n-n_0}).
\end{equation}
\end{lemma}
\begin{proof}
Extracting the extremal terms we have

\begin{align}\label{eq:cauchyProd}
\sum\limits_{k=n_0}^{n-n_0} a_{k} a_{n-k} &= 2 a_{n_0} a_{n-n_0} \left( 1 + \sum\limits_{k = 1}^{\lfloor \frac{n}{2} \rfloor - n_0} \frac{a_{n_0 + k} a_{n-n_0-k}}{a_{n_0}  a_{n-n_0}} \right)
\end{align}
For large enough $n$, we have for $k \leq \sigma^{-1}$
\begin{equation*}
\frac{a_{n_0 + k} a_{n-n_0-k}}{a_{n_0}  a_{n-n_0}} = O(n^{-k \sigma})
\end{equation*}
while for $k > \sigma^{-1}$ we have
\begin{equation*}
\frac{a_{n_0 + k} a_{n-n_0-k}}{a_{n_0}  a_{n-n_0}} = O(n^{-1 - \sigma})
\end{equation*}
and there's at most $\lfloor n /2 \rfloor$ such terms, so that their sum overall is $O(n^{-\sigma})$. Therefore the righthand-side of \Cref{eq:cauchyProd} is $ 2 a_{n_0} a_{n-n_0} + O(n^{-\sigma}) = O(a_{n-n_0})$ as desired.

\end{proof}

\begin{lemma}\label{lemma:composingQparts}
Let $(a_n)_{n \in \mathbb{N}}$ be a positive increasing sequence such that $\frac{a_{n-1}}{a_{n}} = O\left(n^{-\sigma}\right)$ for $\sigma > 0$. Then there exists a constant $C > 0$ such that for all $q \geq 2$

\begin{equation}
\sum\limits_{\overset{k_1 + \dots + k_q = n}{\forall j . k_j \geq 1}} \prod_{j=1}^{q} a_{k_j} \leq C^{q-1} a_{n-q+1}
\end{equation}
\end{lemma}
\begin{proof}
We proceed by induction. For $q=2$ the result holds by \Cref{lemma:cauchyish}. Let $q > 2$ and rewrite the sum as

\begin{equation}
\sum\limits_{k_{q} = 1}^{n-q+1} a_{k_{q}} \sum\limits_{\overset{k_1 + \dots + k_{q-1} = n - k_q}{\forall j . k_j \geq 1}} \prod_{j=1}^{q-1} a_{k_j}
\end{equation}
we then have by our inductive hypothesis
\begin{equation}
\sum\limits_{k_{q+1} = 1}^{n-q+1} a_{k_{q}} \sum\limits_{\overset{k_1 + \dots + k_{q-1} = n - k_q}{\forall j . k_j \geq 1}} \prod_{j=1}^{q-1} a_{k_j} \leq C^{q-2} \sum\limits_{k_q = 1}^{n-q+1} a_{k_{q}} a_{n - k_q - q + 2}
\end{equation}
from which, by applying once more our inductive hypothesis for $q = 2$, we obtain the desired result.

\end{proof}

\begin{lemma}\label{lemma:polyCauchyish}
Let $(a_n)_{n \in \mathbb{N}}$ be a positive increasing sequence such that $\frac{a_{n-1}}{a_{n}} = O\left(n^{-\sigma}\right)$ for $\sigma > 0$. Then for sufficiently large $n$

\begin{equation}
\sum_{k=n_0}^{n} C^k P(k) a_{n-k} = O(a_{n-n_0})
\end{equation}
for any constant $C \in \mathbb{R^{+}}$ and $P \in \mathbb{R^{+}}[k]$ a polynomial in $k$.
\end{lemma}

\begin{proof}
The rapid growth of $(a_n)$ implies that there exists a constant $D > 0$ such that for every $k \geq 1$, $C^k P(k) \leq D a_k$ so that 

\begin{equation}
\sum_{k=n_0}^{n} C^k P(k) a_{n-k} \leq D \sum_{k=n_0}^{n} a_{k} a_{n-k}
\end{equation}
For $k \leq n-n_0$, \Cref{lemma:cauchyish} implies

\begin{equation}
\sum_{k=n_0}^{n-n_0} C^k P(k) a_{n-k} = O(a_{n-n_0})
\end{equation}
while for $n-n_0 < k \leq n$ we have
\begin{equation}
\sum_{k=n-n_0+1}^{n} C^k P(k) a_{n-k} = \sum_{k'=0}^{n_0-1} C^{n-k'} P(n-k') a_{k'} = O(a_{n-n_0})
\end{equation}
by the rapid growth of $(a_{n})_{n \in \mathbb{N}}$.

\end{proof}

\begin{theorem}\label{thm:benderDistr}
Let $G(z,u)$ be a bivariate powerseries 
\begin{equation*}
  \sum\limits_{n \geq 1} g_{n}(u) z^n
\end{equation*}
with positive coefficients and such that 
\begin{equation}\label{eq:benderGnCondition}
\frac{g_{n-1}(1)}{g_{n}(1)} = O\left( n^{-\sigma} \right)
\end{equation}
for $\sigma > 0$.

Let also $F(z, u, y)$ be a power series

\begin{equation}
\sum\limits_{i \geq 0}\sum\limits_{j \geq 0} f_{i,j}(u) z^i y^j
\end{equation}
with $f_{0,1}(1) = 1$, such that $F(z,1,y)$ represents a function analytic at the origin. Then the random variables $X_n$ whose probability generating function is given by

\begin{equation}
p_{n}(u) = \frac{[z^n] F(z, u, G(z,u))}{[z^n] F(z, 1, G(z,1))}
\end{equation}
admit the same limit distribution as the random variables $Y_n$ whose probability generating function is given by

\begin{equation}
q_{n}(u) = \frac{[z^n] G(z,u)}{[z^n] G(z,1)}.
\end{equation}
That is, if $Y_n \convergesInDist Y$, then $X_n \convergesInDist Y$ too.
\end{theorem}

\begin{proof}
We begin by formally expanding $F(z, u, G(z,u))$ as a power series and extracting the coefficient of $z^n$:

\begin{equation}\label{eq:formalSeriesComposition}
  [z^n] F(z, u, G(z,u)) = f_{n,0}  + \sum_{t = 0}^{n-1} \sum\limits_{q = 1}^{n-t} f_{t,q}(u) \sum_{k_1 + \dots + k_q = n-t} \prod\limits_{j=1}^{q} g_{k_j}(u).
\end{equation}

Our goal then, is to to show that after making the change of variables $u=e^{i\tau}$, the summand of \cref{eq:formalSeriesComposition} corresponding to $t=0,~q=1$ (which is exactly $g_n(e^{i\tau})$ by our assumption that $f_{0,1} = 1$), yields asymptotically the dominant contribution for any $\tau \in \mathbb{R}$. This, after normalising by $[z^n] F(z,1,G(z,1))$, is then enough to prove that the characteristic functions $p_n(e^{i\tau})$ converge to the characteristic function $q_n(e^{i\tau})$ corresponding to $Y_n$. To this end, we will need to provide upper bounds for the summands of \cref{eq:formalSeriesComposition} evaluated at $u = e^{i\tau}$. We will do so by exploiting the fact that $\lvert g_{k_{j}}(e^{i\tau}) \rvert$ is bounded above for any real $\tau$ by $\overline{g}_{k_{j}} := g_{k_{j}}(1)$, which allows us to make use of \Cref{lemma:cauchyish,lemma:polyCauchyish,lemma:composingQparts}. We provide bounds for the summands of \cref{eq:formalSeriesComposition} as follows:

\begin{itemize}
  \item Firstly, we deal with the summands corresponding to $\left ([z^{t \geq 1}] F(z,u,y) \right )\lvert_{y = G(z,u)}$ (i.e the summands of \Cref{eq:formalSeriesComposition} in which $f_{t,q}(u)$ is such that $t \geq 1$), showing they are asymptotically negligible.
  \item Secondly, we deal with the summands corresponding to $\left ([z^0] F(z,u,y) \right )\lvert_{y = G(z,u)}$ (i.e the summands of \Cref{eq:formalSeriesComposition} which contain only factors of the form $f_{0,q}(u)$ for $q \geq 1$). Here, we distinguish two sub-cases: 
  \begin{itemize}
    \item One in which $k_1 = n$ is the sole summand appearing in the innermost sum of \cref{eq:formalSeriesComposition}. This case provides the main asymptotic contributions.
    \item The other corresponding to summands with $k_j < n - 1$ for all $j$, which we show are asymptotically negligible.
  \end{itemize}
  
\end{itemize}

{\em Summands corresponding to $\left ([z^{t \geq 1}] F(z,u,y) \right )\lvert_{y = G(z,u)}$.} 
By the growth of the coefficients of $G$ and $F$, we have $f_{n,0} = o(g_{n-1})$, therefore we can focus on the case of $f_{t,q}$ with $q \geq 1$. Since $f$ is analytic, there exists some $D$ such that $\lvert f_{t,q} (1) \rvert \leq D^{t+q}$ so that
by \Cref{lemma:cauchyish,lemma:polyCauchyish,lemma:composingQparts} we have that the restriction of \cref{eq:formalSeriesComposition} to $t \geq 1$ is

\begin{align}\label{eq:boundsForTGeq1}
\begin{split}
  \left\lvert \sum_{t = 1}^{n-1} \sum\limits_{q = 1}^{n-t} f_{t,q}(e^{i\tau}) \sum_{k_1 + \dots + k_q = n-t} \prod\limits_{j=1}^{q} g_{k_j}(e^{i\tau}) \right\rvert
  &\leq  \sum_{t = 1}^{n-1} \sum\limits_{q = 1}^{n-t} \lvert f_{t,q}(1) \rvert \sum_{k_1 + \dots + k_q = n-t} \prod\limits_{j=1}^{q} \overline{g}_{k_j} \\
  &\leq  \sum_{t = 1}^{n-1} \sum\limits_{q = 1}^{n-t} D^{t+q} C^{q-1} \overline{g}_{n-t-q+1} \qquad\text{by \Cref{lemma:composingQparts}} \\
  &\leq  \sum_{t = 1}^{n-1} D^{t} C^{-1} \sum\limits_{q = 1}^{n-t} D^{q} C^{q} \overline{g}_{n-t-q+1}  \\
  &\leq  \sum_{t = 1}^{n-1} D^{t} C^{-1} \sum\limits_{q = 1}^{n-t} K \overline{g}_{q} \overline{g}_{n-t-q+1}  \quad\text{by \cref{eq:benderGnCondition} $\rightarrow \exists K. \forall k. D^k C^k \leq K \overline{g}_{k}$} \\
  &\leq \sum_{t = 1}^{n-1} D^{t} C^{-1} K O(\overline{g}_{n-t}) \qquad\text{by \Cref{lemma:cauchyish}} \\
  &= O(\overline{g}_{n-1}) \qquad\text{by \Cref{lemma:polyCauchyish}} \\
\end{split}
\end{align}

{\em Summands corresponding to $\left ([z^0] F(z,u,y) \right )\lvert_{y = G(z,u)}$.} We will now focus on the summand of \cref{eq:formalSeriesComposition} corresponding to $t=0$

\begin{equation*}
\sum\limits_{q \geq 1}^{n} f_{0,q}(u) \sum_{k_1 + \dots + k_q = n} \prod\limits_{j=1}^{q} g_{k_j}(u).
\end{equation*}

We may rewrite this sum as follows, depending on whether $q=1$ or $q \geq 2$
\begin{equation}\label{eq:split}
f_{0,1}(u) g_{n}(u) + \sum\limits_{q \geq 2}^{n} f_{0,q}(u) \sum_{\overset{k_1 + \dots + k_q = n}{\forall j . k_j > 1}} \prod\limits_{j=1}^{q} g_{k_j}(u)
\end{equation}

We proceed by providing bounds for the second term of \cref{eq:split} evaluated at $u = e^{i\tau}$. Once again, we note that since $f$ is analytic at 0, $\lvert f_{0,q}(1) \rvert \leq D^q$ for some constant $D$. As such we have,

\begin{align}\label{eq:boundsForKSmallerThanNminusD}
\begin{split}
\left\lvert \sum\limits_{q \geq 2}^{n} f_{0,q}(e^{it}) \sum_{\overset{k_1 + \dots + k_q = n}{\forall j . k_j > 1}} \prod\limits_{j=1}^{q} g_{k_j}(e^{it}) \right\rvert & \leq \sum\limits_{q \geq 2}^{n} D^{q} \sum_{\overset{k_1 + \dots + k_q = n}{\forall j . k_j > 1}} \prod\limits_{j=1}^{q} \overline{g}_{k_j} \\ 
&\leq \sum\limits_{q \geq 2}^{n} D^{q} C^{q-1} \overline{g}_{n-q+1} \qquad\text{by \Cref{lemma:composingQparts}}\\
&\leq C^{-1} K \sum\limits_{q \geq 2}^{n-1} \overline{g}_q \overline{g}_{n-q+1} + D^{n} C^{n-1} \overline{g}_{1} \quad\text{by \cref{eq:benderGnCondition} $\rightarrow \exists K. \forall k. D^k C^k \leq K \overline{g}_{k}$}\\
&= O(\overline{g}_{n-1}) \qquad\text{by \Cref{lemma:cauchyish}}
\end{split}
\end{align}

Now, by Bender's theorem (\cite[Theorem 1]{bender1975asymptotic}) together with the fact that $f_{0,1} = 1$, we have that the coefficients of $F(z, u, G(z,1))$ grow asymptotically as $\overline{g}_n$. Therefore the second term of \cref{eq:split}, divided by $\overline{g}_n$, tends to $0$ as $n$ tends to infinity, due to the bound demonstrated in \cref{eq:boundsForKSmallerThanNminusD}. Similarly, the terms corresponding to $t \geq 1$ in \cref{eq:formalSeriesComposition}, when divided by $\overline{g}_n$, also tend to $0$, due to \cref{eq:boundsForTGeq1}.

Finally, since $f_{0,1}(1) = 1$, we obtain that, 

\begin{equation}
p_n(e^{it}) = g_{n}(e^{it}) + O\left( \frac{1}{n^{\sigma}}\right)
\end{equation}
for any real $t$ as $n \rightarrow \infty$, with the error being uniform in $t$, which by Levy's convergence theorem leads to our desired result.
\end{proof}

\subsection{Distribution of degree 1 vertices in $\oneThreeMaps$ and of free variables in $\openTerms$}\label{subsec:freeVars}

In this section, we will apply \Cref{thm:benderDistr} in order to determine the distribution of 1-valent vertices in (1,3)-valent maps, from which we derive the distribution of free variables in linear lambda terms considered up to variable exchange.

Let $\discoOneThreeMaps$ be the class of {\em not-necessarily-connected $(1,3)$-valent maps}.
Viewed as combinatorial maps, elements of $\discoOneThreeMaps$ consist of a permutation $v$ having cycles of length $3$ or $1$ and a fixed-point free involution $e$.
Using the symbolic method, we obtain the exponential generating function of maps in $\discoOneThreeMaps$ counted by number of half-edges (tracked by the variable $h$), which we can moreover refine to a bivariate generating function also keeping track of the number of 1-valent vertices (tracked by $u$):
\begin{equation}\label{eq:disco13Valent}
\left( \exp(h^2/2) \odot \exp(h^3/3 + uh) \right).
\end{equation}
Taking the logarithm of this expression, in full analogy with \Cref{eq:logExpConnected}, yields the generating function of {\em connected} (1,3)-maps, while an application of the $h \partial_h$ operator yields {\em half-edge rooted} connected $(1,3)$-maps counted by number of half-edges and 1-valent vertices.
Now, to switch from half-edge rooted $(1,3)$-maps to vertex-rooted $(1,3)$-valent maps, we apply the bijection explained in \Cref{par:rootedMaps} and seen in \Cref{fig:rootings}, which corresponds to multiplying by $h^4$ and adding the initial conditions\footnote{See \Cref{foot:empty} for an explanation of why $\loopMap$ must be considered as a special case, and a similar argument applies to $\oneEdgeMap$.} $uh^2 + h^4$ to yield
\begin{equation}\label{eq:hadamardspecHalfEdges}
    uh^2 + h^4 + h^5 \frac{\partial}{\partial h} \left( \ln \left( \exp(h^2/2) \odot \exp(h^3/3 + uh) \right) \right),
\end{equation}
and finally, to switch from counting half-edges to counting edges, we apply the change of variables $h^2 \mapsto z$:

\begin{equation}\label{eq:hadamardspec}
    \lamOpenOGF(z,u) = uz + z^2 + 2 z^3 \frac{\partial}{\partial z} \left( \ln \left( \left. \left( \exp(h^2/2) \odot \exp(h^3/3 + uh) \right) \right\rvert_{h=z^{1/2}} \right) \right).
\end{equation}
Although this equation was derived in a completely different manner from \Cref{eq:diff_open_linlams}, they both speak about the same bivariate generating function $\lamOpenOGF(z,u)$, which as explained in \Cref{subsection:intro} has an interpretation as counting linear $\lambda$-terms by number of subterms and free variables.
To be completely precise, the coefficient $[z^nu^k]\lamOpenOGF(z,u)$ counts the number of linear lambda terms with $n$ subterms and $k$ free variables, \emph{considered up to exchange of free variables.}
Equivalently, it counts the number of open rooted trivalent maps with $n$ edges and $k+1$ external vertices, considered up to relabelling of the non-root external vertices.
In other words, since each of the $k!$ possible relabellings of the free variables/external vertices yields a distinct labelled object (a property known as \emph{rigidity}), the variable $u$ in $\lamOpenOGF(z,u)$ may be interpreted as either of exponential type (when tracking variables in linear lambda terms or non-root external vertices in open rooted trivalent maps) or of ordinary type (when counting 1-valent vertices in vertex-rooted (1,3)-maps).
See \Cref{fig:openLinearMap} for an example making this correspondence more concrete.

\begin{figure}[h]
  \centering
  \includegraphics[scale=2]{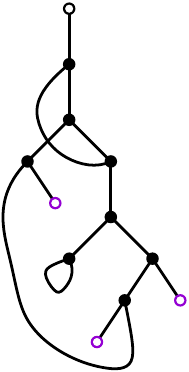}
  \caption{Example of a vertex-rooted (1,3)-valent map corresponding to the underlying map of the linear term $\lambda x.\lambda y.(((a ((y c) b)) (\lambda z.z)) x)$.  Since the 1-valent vertices are unlabelled (compare with the open rooted trivalent maps in \Cref{fig:boundary}), the map captures the structure of the $\lambda$-term up to permutation of the free variables ($a,b,c$).}
  \label{fig:openLinearMap}
\end{figure}

\begin{theorem}\label{lemma:univalentVertexDist}
Let $\chi_{uni}$ be the combinatorial parameter corresponding to the number of 1-valent vertices in unrooted (1,3)-maps $\oneThreeMaps$. Let, also, $X^{uni}_n$ be the random variable corresponding to $\chi_{uni}$ taken over $\oneThreeMaps_n$. Then the mean and variance of $X^{uni}_n$ are asymptotically $\mu_n = \sigma^2_n = \sqrt[3]{2n}$ and the standardised random variables converge to a Gaussian law:

\begin{equation}
\frac{X^{uni}_n - \mu_n}{\sqrt{\sigma^2_n}} \convergesInDist \mathcal{N}(0,1)
\end{equation}
\end{theorem}

An application of \Cref{lemma:rootingDoesntAffectDist} and a simple change of variables $n \mapsto n-2$ yields.
\begin{corollary}\label{lemma:freeVarsInLinearLams}
  Let $\chi_{free}$ be the combinatorial parameter of exponential type corresponding to the number of non-root external vertices in open rooted trivalent maps and the number of free variables in open linear $\lambda$-terms.
  Then for $\mu_n = \sigma^2_n = \sqrt[3]{2(n-2)}$, the random variables $X^{free}_n$ corresponding to $\chi_{free}$ taken over $\openTerms_n$, properly standardised, converge to a Gaussian law:

\begin{equation}
\frac{X^{free}_n - \mu_n}{\sqrt{\sigma^2_n}} \convergesInDist \mathcal{N}(0,1)
\end{equation}

\end{corollary}

\begin{figure}[h]
  \centering
  \includegraphics[scale=0.7]{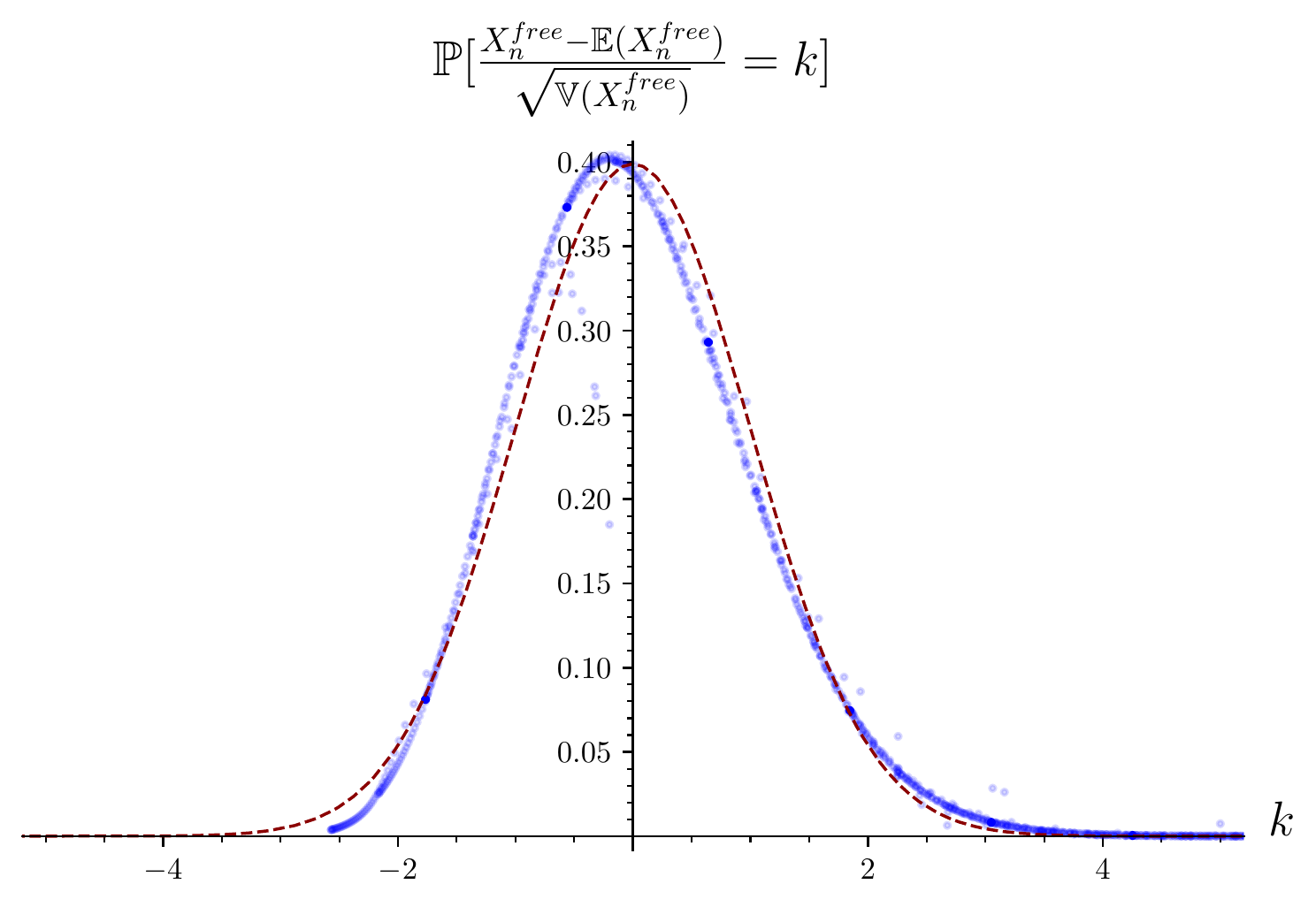}
  \caption{Overlayed density plots of standardized $X^{free}_n$ for $n = 2 \dots 100$ along with that of $\mathcal{N}(0,1)$ in red.}
  \label{fig:distPlotVars}
\end{figure}

Our plan is to first determine the probability generating function for the number of degree 1 vertices in large random maps in $\discoOneThreeMaps$. We will then make use of \Cref{thm:benderDistr,lemma:rootingDoesntAffectDist} to obtain the analogous result for $\openTerms$.

Exploiting the structure of exponential Hadamard products, we obtain the result for $\discoOneThreeMaps$ by combining the coefficient asymptotics of $\exp(z^3/3 + uz)$ and $\exp(z^2/2)$, as presented in the following lemmas.

\begin{lemma}\label{expzthree}
We have that, as $n$ tends to infinity,
\begin{equation}
[z^n] \exp{\left(\frac{z^3}{3} + uz\right)} = \exp{\left({un^{\frac{1}{3}} + \frac{n}{3}}\right)} ~ n^{-\frac{n}{3}} 
\left(
\frac{1}{\sqrt{6n\pi}} - \frac{u^2}{6\sqrt{6\pi}n^{5/6}} + O\left( \frac{1}{n^{7/6}}\right)
\right)
\end{equation}
\end{lemma}
\begin{proof}
We carry out a saddle-point analysis of the e-admissible, in the sense of \cite{drmota2005extended}, function ${\exp(z^3/3 + uz)}$. Let
\begin{align*}
   a(z,u) = z \frac{\partial_z f(z,u)}{f(z,u)},\\
   b(z,u) = z \partial_z a(z,u)
\end{align*}
Then we have that the saddle point is the unique real solution to $a(z,\zeta) = 0$, which can be computed to be

\begin{equation*}
\zeta_n = \frac{1}{6} 
\frac { \left( 108n +12 \sqrt {12u^3 +81n^2} \right) ^{2/3}-12u}{\left( 108n +12 \sqrt {12u^3 +81n^2} \right)^{1/3}}
\end{equation*}
As $n$ tends to infinity we then have, asymptotically, that 

\begin{equation*}
\zeta_n = n^{1/3} - \frac{u}{3n^{1/3}} - \frac{1}{3n^{2/3}} + O\left( \frac{1}{n^{4/3}} \right)
\end{equation*}
Substituting the above into the saddle-point formula of \cite[Theorem 1]{drmota2005extended}
\begin{equation}
 f_{n} \sim \frac{f(\zeta_n, u)}{{\zeta^{n+1}_n} \sqrt{2 \pi b(\zeta_n,u)}} \left( \exp\left(-\frac{(a(\zeta_n, u) - n^2)}{2b(\zeta_n,u)}\right)\right) \sim \frac{f(\zeta_n, u)}{{\zeta^{n+1}_n} \sqrt{2 \pi b(\zeta_n,u)}}
\end{equation}
where $f(z) = \exp(z^3/3 + uz)$,  we get the desired result with the error term being uniform in both $n$ and $u$.
\end{proof}

\begin{lemma}\label{involutions}
We have that
\begin{equation}
[z^n] \exp\left({\frac{z^2}{2}}\right) = \exp\left(\frac{n}{2}\right) n^{-\frac{n}{2}} \left( \frac{1+\exp({-in\pi})}{2\sqrt{n\pi}} +
O\left(\frac{1}{n^{3/2}}\right)
\right)
\end{equation}
\end{lemma}
\begin{proof}
The proof, once again, follows from an application of the saddle-point method. Let $h(z) = 1/2\,{z}^{2}- \left( n+1 \right) \ln  \left( z \right) $. We have that $h'(z)$ has roots at $\sqrt{1+n}$ and $-\sqrt{1+n}$ and so by combining the contributions from each individual saddle-point, we obtain the desired result.
\end{proof}

Let $\chi_{uni,d}$ be the following extension of $\chi_{uni}$ from $\oneThreeMaps$ to $\discoOneThreeMaps$: if a map $m \in \discoOneThreeMaps$ is connected, then $\chi_{uni,d}(m) = \chi_{uni}$, otherwise $\chi_{uni,d} = \sum\limits_{C \text{ is a connected component of } m} \chi_{uni}(C)$. We then have:
\begin{lemma}\label{lemma:freeVarsInDiscos}
The probability generating functions of the random variable $X^{uni,d}_n$ corresponding to $\chi_{uni,d}$ taken over the set $\discoOneThreeMaps_n$ of not-necessarily-connected $(1,3)$-valent maps with $n$ edges is, for large n,
\begin{equation}
p_{X^{uni,d}_n}(u) = \exp\left({(u-1) (2n)^{1/3}}\right) \left(1 + \frac{1-u^2}{6(2n)^{1/3}} + O\left(\frac{1}{n^{2/3}}\right) \right) 
\end{equation}
\end{lemma}
\begin{proof}
We have
\begin{equation*}
p_{X_n}(u)  =
\frac{[z^n] \left( \exp\left(z^2/2\right) \odot \exp\left(z^3/3 + uz\right)\right)}
{[z^n] \left( \exp\left(z^2/2\right) \odot \exp\left(z^3/3 + z\right) \right)} 
=\frac{n! \left( [z^n] \exp\left(z^2/2\right) \cdot [z^n] exp\left(z^3/3 + uz\right) \right)}
{n! \left( [z^n] \exp\left(z^2/2\right) \cdot [z^n] exp\left(z^3/3 + z\right) \right)} 
\end{equation*} 

Where $X_n$ is the random variable corresponding to the number of degree 1 vertices in not-necessarily-connected (1,3)-maps counted by {\em number of half-edges}.
An application of \Cref{expzthree} and \Cref{involutions} then yields the desired asymptotic result after a shift of $n \mapsto 2n$.
\end{proof}

We can now proceed with a proof of \Cref{lemma:univalentVertexDist}.
\begin{proof}[Proof of \Cref{lemma:univalentVertexDist}]

Note that $f(z,u,g(z,u))$ with $f(z,u,y) = \ln(1+y)$ and 
\begin{equation}
g(z,u) = \left. \left(  \exp(h^2/2) \odot \exp(h^3/3 + uh) \right)\right \lvert_{h = z^{1/2}} - 1,
\end{equation}
is aperiodic and furthermore, by \Cref{expzthree,involutions}, 
\begin{equation}
\frac{\overline{g}_{n-1}}{\overline{g}_{n}} = O\left(n^{-1/3}\right)
\end{equation}
where $\overline{g}_n = [z^n] g(z,1)$. Therefore the composition $f(z,u,g(z,u))$ falls under the schema presented in \Cref{thm:benderDistr}. Therefore we can apply said lemma to conclude that the limit distribution of $X^{uni}_{n}$ is the same as that of $X^{uni,d}_{n}$.

The generating function of $X^{uni,d}_{n}$, as given by \Cref{lemma:freeVarsInDiscos}, is

\begin{equation}
p_{X^{uni,d}_{n}}(u) = \exp({(u-1) (2n)^{1/3}}) \left(1 + \frac{1-u^2}{6(2n)^{1/3}} + O\left(\frac{1}{n^{2/3}}\right) \right) 
\end{equation}
which is of the form required to apply the quasi-powers theorem, yielding our final result.
\end{proof}

\begin{remark}
The above asymptotic form of  $p_{X^{uni}_n}(u)$ suggests that a $Poisson((2n)^{1/3})$ approximation might be more appropriate (in the sense of a faster speed of convergence).

It is of interest to note here that there is a fair number of studies (see, for example, \cite{ercolani2014cycle,benaych2007cycles,yakymiv2007random}) on the structure of random permutations with restricted cycle lengths, which are quite relevant to this and the following subsections.

In \cite{benaych2007cycles} it is shown that if $A \subseteq \mathbb{N}^{+}$ is an infinite set of allowed cycle lengths, then for all $l \geq 1$, the joint distribution of the random vector 
\begin{equation}
  \left( N_{k}(\sigma_n) \right)_{1 \leq k \leq l \land k \in A}
\end{equation} 
converges weakly, as $n \rightarrow \infty$, to
\begin{equation}
\mathop{\otimes}\limits_{1 \leq k \leq l \land k \in A} Poisson(1/k)
\end{equation}

Of course, this is not possible when $A$ is finite. In such a situation, the author of \cite{benaych2007cycles} shows that, for $d = \max{A}$ and for all $l \in A$, $\frac{N_l(\sigma_n)}{n^{l/d}}$ converges in all $L^p$ spaces, with $p \in [1, +\infty)$, to $1/l$ and remarks that an approach based on analytic combinatorics may useful in obtaining the limit distributions of dilations of the random variables $N_l(\sigma_n)/n^{l/d} - 1/l$. Indeed, our approach yields such limit laws for the cases where $A = \{1,2\}$ or $A = \{1,3\}$ and {\em in principle} can be extended to any finite $A$, since in all such cases the resulting functions are e-admissible and therefore admit Gaussian limit laws as shown in \cite{gittenberger2006hayman}.
\end{remark}

\subsection{Distribution of degree 2 vertices in $\twoThreeMaps$ and unused abstractions in $\affineTerms$}\label{subsec:freeLams}
In this final subsection we show that the limit distributions of vertices of degree 2 in (2,3)-valent maps and rooted (2,3)-valent maps, as well as the limit distribution of unused abstractions in affine $\lambda$-terms, all admit a Gaussian law. 

So far we have studied families of $\lambda$-terms which satisfy a condition of linearity: a bound variable must appear exactly once inside the body of its corresponding abstraction. By relaxing this condition, we obtain the notion of {\em affine $\lambda$-terms}, in which bound variables may occur {\em at most once} inside the body of their abstractions.
Combinatorially, such terms may be obtained from linear $\lambda$-terms by replacing every vertex in their syntactic diagrams 
by a non-empty path. We therefore have the following equality between the generating functions $\affineTermsOGF$ of closed affine terms and $\lamClosedOGF$ of closed linear terms:

\begin{equation}\label{eq:pathSubsAffineSymb}
\affineTerms = \closedTerms(SEQ_{\geq 1} \singleton)
\end{equation}

In terms of maps, \Cref{eq:pathSubsAffineSymb} signifies a passing from rooted trivalent maps to elements of $\rootedTwoThreeMaps \times \pathClass$: rooted $(2,3)$-valent maps $m$ together with a path $p \in P_n$, which we can visualise as grafted on the root of $m$ as in \Cref{fig:affineMap}. 

\begin{figure}[h]
  \centering
  \includegraphics[scale=2]{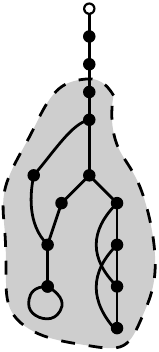}
  \caption{An element of $\rootedTwoThreeMaps \times \pathClass$ corresponding to the $\lambda$-term $\lambda x.\lambda y.(\lambda z.\lambda a.\lambda b.b a)(\lambda w.((\lambda d.d)(\lambda e.y)))$. The rooted $(2,3)$-valent submap corresponding to the element of $\rootedTwoThreeMaps$ is highlighted in grey.}
  \label{fig:affineMap}
\end{figure}

Translating \Cref{eq:pathSubsAffineSymb} to an equality between generating functions and marking unused abstractions/vertices of degree 2 with $t$ we have

\begin{equation}\label{eq:pathSubsAffine}
\affineTermsOGF(z, t) = \lamClosedOGF\left(\frac{z}{1-tz}\right),
\end{equation}
from which, by extracting coefficients, we obtain

\begin{equation}\label{eq:affineBinom}
a_n =  \sum\limits_{k = 0}^{n} t^k {i-1 \choose k} l_{k+1}  = \sum\limits_{k = 0}^{\lceil \frac{n}{3} \rceil} t^{n-1-(3k+1)}{i-1 \choose 3k+1} l^{*}_{k}
\end{equation}
where $a_n = [z^n] \affineTermsOGF, l_n = [z^n] \lamClosedOGF(z), l^*_n = l_{3n+2}$. The above recurrence can be used to derive quickly, if a bit informally, a lower bound to the asymptotic mean for the number of unused abstractions and vertices of degree 2. Let $S(n,k) = {i-1 \choose 3k+1} l^{*}_{k}$ be the summand of \Cref{eq:affineBinom}, evaluated $t=1$. We note that $S(n,k)$ is convex and so we may estimate, asymptotically as $n \rightarrow \infty$, the index $k$ for which it is maximum by solving the following equation for $k$

\begin{equation}\label{eq:discreteSaddlepoint}
\frac{S(n,k+1)}{S(n,k)} = 1
\end{equation}

Plugging $l^{*}_k \sim \frac{3}{\pi} 6^k k!$ into \Cref{eq:discreteSaddlepoint} we obtain, asymptotically, a solution of the form

\begin{equation}
k_m \sim \frac{n}{3} - \frac{(2n)^{2/3}}{6} + \frac{(2n)^{1/3}}{9} - \frac{19}{18}.
\end{equation}

Doing the same for $S'(n,k) = (n-1-(3k+1)) {i-1 \choose 3k+1} l^{*}_{k}$, the summand of the derivative with respect to $t$ of \Cref{eq:affineBinom} evaluated at $t=1$, we obtain an optimal index of the form

\begin{equation}
k_{m}' \sim \frac{n}{3} - \frac{(2n)^{2/3}}{6} + \frac{(2n)^{1/3}}{9} - \frac{25}{18}.
\end{equation}

Therefore we have, for large $n$,

\begin{equation}
\frac{S'(n,k)}{S(n,k)} \geq \frac{S'(n,k_m)}{S(n,k'_m)} \sim \frac{(2n)^{2/3}}{2} - \frac{(2n)^{1/3}}{3} - {2}
\end{equation}

This lower bound is, as we will prove below, quite tight: it is accurate to the first two orders. However, to go from the lower bound above to the precise asymptotic result, one must show that for indices other than $k \neq k_m$ and $k \neq k'_{m}$ the contributions of $S(n,k)$ and $S'(n,k)$ respectively are neglible, uniformly for $t$ in a fixed neighbourhood of $1$. This proves a bit tedious and so we will now seek an alternative specification for $\affineTermsOGF(z,t)$.

Our new specification for the generating function of the class $\affineTerms$, where $t$ again marks unused abstractions, can be obtained by a straightforward extension of the results in \cite{bodini2013asymptotics}:

\begin{align}
  \twoThreeMapsOGF(z,t) &= h \frac{\partial}{\partial h} \left( \ln \left( \exp\left({\frac{h^2}{2}}\right) \odot \exp\left({\frac{h^3}{3} + \frac{th^2}{2}}\right) \right) \right) \label{eq:twoThreeMapsEq} \\
  \affineTermsOGF(z,t) &= \frac{z^2 + z^2 \twoThreeMapsOGF(z^{\frac{1}{2}},t)}{1-tz}  \label{eq:finalSpecAffine}
\end{align}

Starting with \Cref{eq:twoThreeMapsEq} we note that the presence of the term $tz^2/2$ inside the Hadamard product denotes that we allow for, and mark by $t$, vertices of degree 2 in our maps. This yields the ordinary generating function $\twoThreeMapsOGF(z,t)$ of the class $\rootedTwoThreeMaps$ of {\em rooted $(2,3)$-maps} with vertices of degree 2 tagged by $t$. Next, we notice that an affine $\lambda$-term can start with an arbitrary number of abstractions whose bound variable is discarded. In the realm of maps, as discussed before, this corresponds to an element of $\rootedTwoThreeMaps \times \pathClass$: a rooted $(2,3)$-valent map $m$ together with a path $p \in P_n$ (see, again, \Cref{fig:affineMap}). To count such maps we need only multiply the generating function of rooted $(2,3)$-valent maps by the generating function $\frac{1}{1-tz}$ counting paths $P_n$ with $n$ edges and with vertices of degree 2 marked by $t$.

Notice that adding $p \in P_n$ on top of a rooted $(2,3)$-valent map $m$ shifts the value of $\chi_{bi}(m)$ from, say, $\chi_{bi}(m)$, to $\chi_{bi}(m) + n$. As such, one could imagine that maps that are extremal, in the sense that almost all their vertices lie on such path, skew the distribution of the extended $\chi_{bi}$ taken over $[\rootedTwoThreeMaps \times \pathClass]_n$. Thankfuly, the rapid growth of $\rootedTwoThreeMaps$ implies the number of maps in $\lvert \pathClass_{k} \times \rootedTwoThreeMaps_{n-k} \rvert = \lvert \rootedTwoThreeMaps_{n-k} \rvert$ pales in comparison to $\lvert \rootedTwoThreeMaps_n \rvert$ for any $k \geq 1$. Intuitively, this means that almost all maps in $\rootedTwoThreeMaps \times \pathClass$ are plain rooted $(2,3)$-valent maps, i.e., elements of $\rootedTwoThreeMaps$ and so the limit distribution dictated by $\chi_{bi}$ remains the same.

Using \Cref{eq:twoThreeMapsEq,eq:finalSpecAffine} we will now proceed to show the following.

\begin{theorem}\label{lemma:bivalentVertexDist}
Let $\chi_{bi}$ be the combinatorial parameter of $\twoThreeMaps$ corresponding to the number of degree 2 vertices. Let, also, $X^{bi}_n$ be the random variable corresponding to $\chi_{bi}$ taken over $\twoThreeMaps_n$. Then the mean and variance of $X_n$ are asymptotically $\mu_n = \sigma^2_n = \frac{(2n)^{2/3}}{2}$, while the standardised random variables converge to a Gaussian law:

\begin{equation}
\frac{X^{uni}_n - \mu_n}{\sqrt{\sigma^2_n}} \convergesInDist \mathcal{N}(0,1)
\end{equation}
\end{theorem}

\begin{figure}[h]
  \centering
  \includegraphics[scale=0.7]{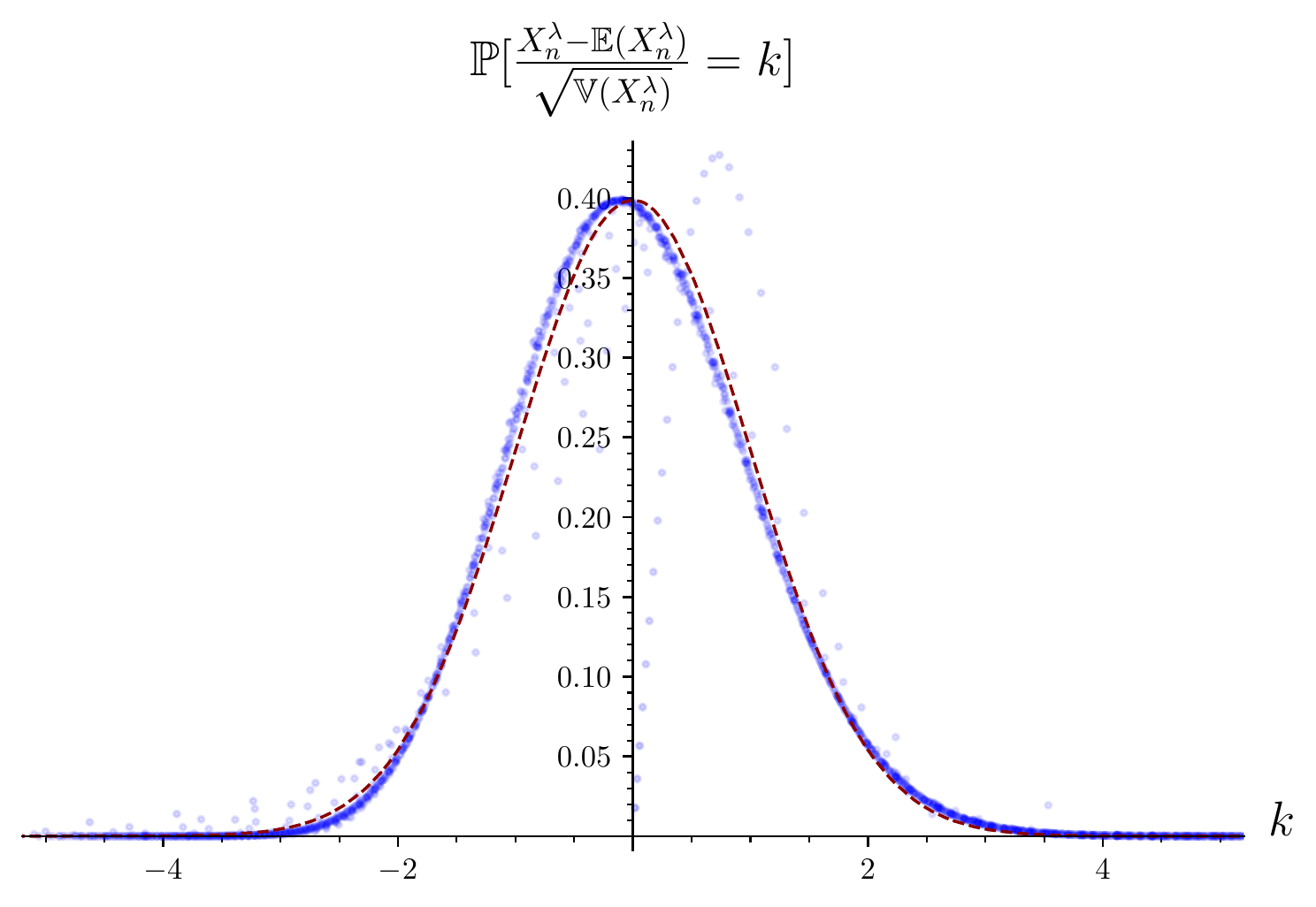}
  \caption{Overlayed density plots of standardized $X^{\lambda}_n$ for $n = 2 \dots 100$ along with $\mathcal{N}(0,1)$ in red.}
  \label{fig:distPlotAffine}
\end{figure}

Once again, we obtain our desired result for $\lambda$-terms as a corollary of the above.
\begin{corollary}\label{lemma:freeLamsInAffineLams}
Let $\chi_{\lambda}$ be the combinatorial parameter of $\affineTerms$ corresponding to to the number of abstractions discarding their variable in closed affine $\lambda$-terms. Then for $\mu_n = \sigma^2_n = \frac{(2(n-2))^{2/3}}{2}$, the standardised random variable $X^{\lambda}$ corresponding to $\chi_{\lambda}$ taken over $\affineTerms_n$ converges to a Gaussian law:

\begin{equation}
\frac{X^{\lambda}_n - \mu_n}{\sqrt{\sigma^2_n}} \convergesInDist \mathcal{N}(0,1)
\end{equation}
\end{corollary}

\begin{lemma}\label{expzthreeplustwo}
We have that, as $n$ tends to infinity,
\begin{align}
[z^n] \exp\left({\frac{tz^2}{2} + \frac{z^3}{3}} \right) &= n^{-\frac{n}{3}} \exp\left({\frac{tn^{2/3}}{2} - \frac{t^2n^{2/3} }{2} + \frac{n}{3}}\right) \left( \frac{\sqrt{6}e^{\frac{t^3}{18}}}{6\sqrt{n \pi}} + O\left( \frac{1}{n^{5/6}} \right) \right)
\end{align}
\end{lemma}

\begin{proof}
Follows from a saddle-point analysis of the e-admissible function $\exp(z^3/3 + tz^2/2)$.
\end{proof}

Combined with \Cref{involutions} we get the following asymptotics for unrooted not-necessarily-connected $(2,3)$-valent maps, where $\chi_{bi,d}$ extends $\chi_{bi}$ by summing the number of vertices of degree 2 for every connected component of a map.

\begin{lemma}\label{lemma:bivalentPGF}
For $n$ going to infinity, the probability generating function of $X^{bi,d}_n$ corresponding to $\chi_{bi,d}$ taken over $\discoTwoThreeMaps_n$ is
\begin{equation}\label{eq:bivalentPGF}
p_{X^{bi,d}_n}(t) = \exp\left({-\frac{n^{1/3} (t-1) (t - 3n^{1/3} +1)}{6}}\right) \left( \exp\left({\frac{(t-1)(t^2+t+1)}{18}}\right) + O \left(\frac{1}{n^{1/3}}\right) \right)
\end{equation}
\end{lemma}

Finally, we have:
\begin{proof}[Proof of \Cref{lemma:bivalentVertexDist}]
Let $f(z,u,y) = \ln(y+1)$ and $g(z,y)=\left. \left( \exp(z^2/2) \odot \exp(z^3/3 + tz^2/2)\right) \right\rvert_{h=z^{1/2}}- 1$. Then, by \Cref{involutions,expzthreeplustwo} we have for $\overline{g}_{n} = [z^n] g(z,1)$:
\begin{equation*}
\frac{\overline{g}_{n-1}}{\overline{g}_{n}} = O \left( n^{-{1/3}} \right)
\end{equation*}
so we can apply \Cref{thm:benderDistr} to conclude that the limit distribution of $X^{bi}$ is the same as that of $X^{bi,d}$, the last of which has been obtained in \Cref{lemma:bivalentPGF}.

Then, around $t=1$, \Cref{eq:bivalentPGF} may be written as

\begin{equation}
p_{X^{bi,d}_n}(t) =\exp\left( \frac{(2n)^{1/3}}{6} - \frac{(2n)^{2/3}}{2}\right) \exp\left(t(t-3(2n)^{1/3}) \frac{-(2n)^{1/3}}{6} \right) \left(1 + O \left(\frac{1}{n^{1/3}}\right) \right)
\end{equation}
and application of L\'evy's continuity theorem then yields the desired result.
\end{proof}

\begin{proof}[Proof of \Cref{lemma:bijectionOneBridgeNoBridge}]
Firstly, by \Cref{lemma:rootingDoesntAffectDist} we have that the limit distribution of $\chi_{b}$ taken over $\rootedTwoThreeMaps$ is the same as that of $\chi_{b}$ taken over $\twoThreeMaps$.
Then, an application of \Cref{thm:benderDistr} for $f(z,t,y) = \frac{z^2 + y}{1-tz}$ and $g(z,t) = \twoThreeMapsOGF(z^{1/2},t)$ shows that limit distribution of $X^{\lambda}_n$ is the same as that of $X^{bi}_n$. As such, the results of \Cref{lemma:bivalentVertexDist} remain the same when passing from $\twoThreeMaps$ to $\rootedTwoThreeMaps$ (apart from the shift of size $n \mapsto n - 2$). 
\end{proof}

%% file: sections/conclusion.tex
We have explored how tools drawn from the study of maps and $\lambda$-terms may be used in tandem to study the structure of typical trivalent maps, linear $\lambda$-terms and related families of objects. Making use of the interplay between maps and terms we have derived a number of new bijections and decompositions for our families of interest, which we then analysed using two new schemas developed here: one based on differential equations and another based on functional compositions. Using this toolkit we have studied a number of parameters of interest
as presented in \Cref{table:pairs}.

As for potential future directions, we would like to suggest two main avenues of study. Firstly, it would be of interest to expand \Cref{table:pairs} by identifying and studying pairs of parameters which naturally occur in the study of both maps and $\lambda$-terms. As an example, the number of $\beta$-redices in closed linear $\lambda$-terms is of particular interest: it serves as a useful lower bound to the distance of a term from its $\beta$-normal form and is also related to the occurence of corresponding ``patterns'' in rooted trivalent maps. Another interesting parameter is given by the path length as defined in the term-tree of a map, which serves to give more information on the structure of typical terms and maps.

Secondly, it is known that the bijection between rooted trivalent maps and linear $\lambda$-terms restricts to one between planar rooted trivalent maps and planar $\lambda$-terms. Planar maps are known to behave quite differently
from maps of arbitrary genus. For example, the associated generating functions are frequently algebraic, which would allow us to make use of a rich toolkit of related techniques to explore parameters of these classes.